\documentclass[amsart,a4paper,11pt]{amsart}
\usepackage[utf8]{inputenc}
\usepackage[T1]{fontenc}
\usepackage{amsmath}
\usepackage{amssymb}
\usepackage{graphicx}
\usepackage[counterclockwise]{rotating}
\usepackage{mathtools}
\usepackage{wasysym}

\usepackage{enumerate}
\usepackage{amsfonts}
\usepackage{dirtytalk}
\usepackage{chngcntr}
\counterwithin{figure}{section}
\usepackage{xy}
\usepackage{amsthm}
\usepackage{xfrac}
\usepackage[dvipsnames]{xcolor}

\usepackage[
backend=bibtex,
giveninits=true,
style=alphabetic,
maxbibnames=99,mincitenames=1, maxcitenames=4]{biblatex}

\DeclareFieldFormat[article,periodical]{volume}{\mkbibbold{#1}}

\usepackage[urlcolor=pink, 
linkcolor=red, 
citecolor=Green,
colorlinks=true,
linktocpage=true]{hyperref}

\usepackage{tikz}
\usepackage{tikz-cd}
\tikzstyle{lien}=[->,>=stealth,rounded corners=5pt,thick]
\usepackage{cleveref}

\crefformat{section}{\S#2#1#3} 
\crefformat{subsection}{\S#2#1#3}
\crefformat{subsubsection}{\S#2#1#3}

\newtheorem{thm}{Theorem}[section]
\newtheorem{prop}[thm]{Proposition}
\newtheorem{lem}[thm]{Lemma}
\newtheorem{cor}[thm]{Corollary}

\theoremstyle{definition}
\newtheorem{defi}[thm]{Definition}
\newtheorem{rem}{Remark}[section]

\numberwithin{equation}{section}

\author{Julien Moy}
\title[Spectral statistics of the Laplacian on random covers]{Spectral statistics of the Laplacian on random covers of a closed negatively curved surface}

\address{Laboratoire de Mathématiques d'Orsay, Universit{\'e} Paris-Saclay, 91405 Orsay cedex, France}
\email{julien.moy@universite-paris-saclay.fr}

\begin{document}
	
	\begin{abstract}
		Let $(X,g)$ be a closed, connected surface, with variable negative curvature. We consider the distribution of eigenvalues of the Laplacian on random covers $X_n\to X$ of degree $n$. We focus on the ensemble variance of the smoothed number of eigenvalues of the square root of the positive Laplacian $\sqrt{\Delta}$ in windows $[\lambda-\frac 1L,\lambda+\frac 1L]$, over the set of $n$-sheeted covers of $X$. We first take the limit of large degree $n\to +\infty$, then we let the energy $\lambda$ go to $+\infty$ while the window size $\frac 1L$ goes to $0$. In this \emph{ad hoc} limit, local energy averages of the variance converge to an expression corresponding to the variance of the same statistic when considering instead spectra of large random matrices of the Gaussian Orthogonal Ensemble (GOE). By twisting the Laplacian with unitary representations, we are able to observe different statistics, corresponding to the Gaussian Unitary Ensemble (GUE) when time reversal symmetry is broken. These results were shown by F. Naud for the model of random covers of a hyperbolic surface.

		For an individual cover $X_n\to X$, we consider spectral fluctuations of the counting function on $X_n$ around the ensemble average. In the large energy regime, for a typical cover $X_n\to X$ of large degree, these fluctuations are shown to approach the GOE result, a phenomenon called ergodicity in Random Matrix Theory. An analogous result for random covers of hyperbolic surfaces was obtained by Y. Maoz.

	\end{abstract}
	\maketitle
	
	\tableofcontents
	\section{Introduction}

	\subsection{Reminders on quantum chaos} \label{parintro} A central observation in \emph{quantum chaos}, going back at least to Percival \cite{ICPercival1973}, is that the statistical properties of spectra of quantum systems should be influenced by the dynamical properties of their classical limit. 
	
	\subsubsection{Spectral statistics} 
	Before proceeding, let us clear out what we mean by \emph{statistical properties} of a spectrum, which is at first sight a deterministic ensemble. Consider a nondecreasing sequence \[ \lambda_1\le \ldots \le \lambda_n\to +\infty,\]
	typically representing the spectrum of a quantum system. To introduce randomness, one may for example consider the distribution of energy levels in an interval picked at random, with fixed width $W$. The simplest quantity to consider is the number of levels lying in such a window. We introduce the counting function, defined, for $\lambda,W>0$, by
	\[n(\lambda,W):=\# \{j, \ \lambda_j\in [\lambda,\lambda+W]\}.\]
	Typically, statistical properties of quantum spectra are investigated in the large energy limit. We now let $W=W(\lambda)$ be a function of $\lambda$ and take $\delta=\delta(\lambda)$ such that\footnote{this means that the ratio $W/\delta$ goes to $0$ as $\lambda$ goes to $+\infty$} $\delta \gg W$. Then, we define the average number of energy level 
	\[\langle n(\lambda,W)\rangle:=\frac{1}{\delta}\int_\lambda^{\lambda+\delta} n(\mu,W)\mathrm d\mu.\]
	The hypothesis $\delta \gg W$ ensures that the above mean incorporates many disjoint windows of width $W$. In the physics literature, it is usually assumed that $\langle n(\lambda,W)\rangle$ has negligible variations over energy ranges of width $\delta$, which is quite natural because the choice of $\delta$ is arbitrary. To measure the fluctuations of $n(\cdot,W)$ around $\langle n\rangle$, we introduce the (spectral) \emph{number variance}
	\[\sigma^2(\lambda,W):=\frac{1}{\delta}\int_\lambda^{\lambda+\delta}  \big|n(\mu,W)-\langle n(\lambda,W)\rangle \big|^2\mathrm d\mu.\]
	The number variance --- and especially its behavior when $\langle n(\lambda,W)\rangle$ gets large --- encodes long range correlations between the energy levels.

	\vspace{2pt}
	
	\subsubsection{Integrability versus chaos}
	Percival argued that systems with chaotic underlying classical dynamics should exhibit some repulsion between energy levels. 
	It contrasts with integrable systems, whose quantum spectrum was later conjectured by Berry and Tabor to obey Poissonnian statistics, \emph{i.e.} their spectrum should behave like a pure point process with random jumps between consecutive energy levels \cite{BerryTabor}. In this case, the variance $\sigma^2(\lambda,W)$ is of order $\langle n(\lambda,W)\rangle $. We refer to the note of Marklof \cite{10.1007/978-3-0348-8266-8_36} for a presentation of mathematically rigorous results on the Berry--Tabor conjecture. 
	
	In \cite{PhysRevLett.52.1,BGS2}, Bohigas, Giannoni and Schmit formulated a conjecture (BGS) about the statistical properties of quantum spectra of systems with chaotic underlying dynamics. Not only does the BGS conjecture predict repulsion between the energy levels, it states that generically, for such a system, correlations of the energy levels should obey \emph{universal statistics} predicted by random matrix theory, when taking the limit of high energies. 
	In absence of spin, these statistics correspond to those of the Gaussian Orthogonal Ensemble  (GOE) for time reversal invariant systems, and the Gaussian Unitary Ensemble (GUE) otherwise. For the Gaussian ensembles GOE and GUE, the number variance\footnote{Here the number variance for a Gaussian ensemble refers to the ensemble variance of $n(\lambda,W)$. One can actually show that with high probability, the spectral number variance for an individual matrix of GOE/GUE is close to the ensemble number variance.} $\sigma^2(\lambda,W)$ is of order $\log \langle n(\lambda,W)\rangle$, it grows much slower with $\langle n\rangle$ than in the uncorrelated case, which reflects rigidity of the spectrum.

	A typical example of continuous dynamics is given by the geodesic flow on a closed surface, the energy levels of its quantum counterpart being given by the eigenvalues of the positive Laplacian. Concretely, let $(X,g)$ be a closed Riemannian surface, and given $\lambda,W>0$, according to the notation above, denote by $n(\lambda,W)$ the number of eigenvalues of the positive Laplacian $\Delta_g$ in the window $[\lambda,\lambda+W]$. According to Weyl's law, we have
	\[n(\lambda,W)=\frac{\mathrm{vol}(X)}{4\pi}W+ \mathcal O(\sqrt \lambda).\]
	This estimate holds for any surface, regardless of the dynamical properties of the geodesic flow. In this setting, the number variance is defined as
	\[\frac{1}{\delta}\int_{\lambda}^{\lambda+\delta}\Big(n(\mu,W)-\frac{\mathrm{vol}(X)}{4\pi}W\Big)^2 \mathrm d\mu.\]
	Typically, we expect a small number variance when the geodesic flow is chaotic.
	
	From now on, we assume that $X$ is negatively curved, which ensures that the geodesic flow is strongly chaotic; in fact it is an Anosov flow \cite{Ano67}. Berry \cite{Berry1} argued and gave evidence that in such a case, in the regime $1\ll W\ll \sqrt{\lambda}$, the number variance equals that of the eigenvalues of large matrices of the Gaussian Orthogonal Ensemble, whose value can be found in \cite{Mehta}. Although many numerical experiments support BGS in this situation, it has yet not been proved for any surface, and it was even shown by Luo and Sarnak \cite{Luo1994} that the number variance for \emph{arithmetic hyperbolic surfaces} doesn't match that of GOE matrices, after the numerical experiments of Bogomolny--Georgeot--Giannoni--Schmit \cite{PhysRevLett.69.1477}. We refer to the review of Ullmo \cite{Ullmo2016BohigasGiannoniSchmitC} for more background along with implications of the conjecture. 
	

	\subsection{Random surfaces and spectral statistics} The BGS conjecture seeming presently out of reach, a recent approach consists, instead of considering an individual system, in trying to obtain results \emph{on average} for an ensemble of surfaces. This is not illegitimate since the conjecture is expected to hold generically, and is known to fail in some cases.
	
	\subsubsection{The smoothed counting function}
	As sharp counting functions are harder to study, we shall take a smoothed counting function that allows to exploit trace formulas. Let $\psi$ be a smooth, even, real function, whose Fourier transform has compact support. Let $(X,g)$ be a closed, connected Riemannian surface. The positive Laplacian $\Delta_g$ on $C^\infty(X)$ is essentially self-adjoint, and its unique self-adjoint extension will still be denoted $\Delta_g$. Its spectrum consists of a nondecreasing sequence of eigenvalues 
	\[0=\lambda_1^2\le \lambda_2^2\le \ldots\le \lambda_n^2\to +\infty,\]
	so that the numbers $\lambda_j$ are the eigenvalues of the square root of the Laplacian $\sqrt{\Delta_g}$. 
	
	Define, for $\lambda,L>0$, the smoothed counting function
	\begin{equation}\label{eq:countdef}N(\lambda,L)=\sum_{j\ge 1} \psi(L(\lambda-\lambda_j))+\psi(L(\lambda+\lambda_j)).\end{equation}
	The convergence of the series is guaranteed by the asymptotics of Weyl's law and the fact that $\psi$ is fast decaying. It essentially takes into account eigenvalues in a window centered at $\lambda$ of width $\sim \frac 1L$. The typical number of eigenvalues in such a window is of order $\frac{\mathrm{vol}(X)\lambda}{L}$, according to Weyl's law. The terms $\psi(L(\lambda+\lambda_j))$ are here for technical reasons because of the need to consider \emph{even} counting functions. Still, for a fixed surface $X$ they decay very fast as $\lambda$ gets large. Note that $N(\lambda,L)$ is a smoothed version of $n(\lambda,\frac 1L)$, with notations of \S \ref{parintro}.
	
	\subsubsection{The setting.} One considers an ensemble of surfaces $\mathcal M$, which is endowed with a probability measure $\mathbf P$. Given the leading term in Weyl's law, it is natural to consider an ensemble of surfaces with the same volume. The counting function defined in \eqref{eq:countdef} now becomes a random variable on $\mathcal M$, that we denote $N_\omega(\lambda,L)$. We are interested in the fluctuations of $N_\omega(\lambda,L)$ around the \emph{ensemble average} $\mathbf E_\omega(N_\omega(\lambda,L))$, that is the \emph{ensemble variance} 
	\[\Sigma_\omega^2(\lambda,L):=\mathbf E_\omega\Big[\big(N_\omega(\lambda,L)-\mathbf E_\omega(N_\omega(\lambda,L))\big)^2\Big].\] 
	

	There are many models of random surfaces, but until recently, few averaging formulas were available to make computations tractable. Let us cite two developments: 
	\begin{itemize} \item Rudnick \cite{rudnick2022goe} studied the variance of the smoothed counting function $N(\lambda,L)$ (defined in \eqref{eq:countdef}) for the Weil--Petersson model of random hyperbolic surfaces of genus $g$. By first taking the limit of large genus, then shrinking the size of the window (\emph{i.e.} letting $L\to +\infty$) while $\lambda$ remains fixed, the ensemble variance converges to an expression consistent with GOE (see also the work of Rudnick--Wigman \cite{rudnick2023central,rudnick2023sure}). His proof relies on the Selberg trace formula, which relates the spectrum of the Laplacian to the lengths of closed geodesics; and on integration formulas for simple closed geodesics obtained by Mirzakhani \cite{Mirzakhani2007} and Mirzakhani--Petri \cite{MirzakhaniPetri}. We point out that Marklof and Monk recently extended the result of Rudnick for the Weil--Petersson model to the case of twisted Laplacians \cite{marklof2024modulispacetwistedlaplacians}.
		
		\item  Naud \cite{naud2022random} studied the ensemble variance of $N(\lambda,L)$ for the discrete model of $n$-sheeted random covers of a fixed hyperbolic surface $X$, introduced by Magee--Naud--Puder \cite{magee_puder_2023,mageepudernaud}, that we describe in \S \ref{sec : random covers} --- briefly, one sets a uniform probability distribution on the set of vertex-labeled $n$-covers $X_n\to X$. Naud also obtained the expression of GOE, by first letting the degree $n\to +\infty$, then $L\to +\infty$, while $\lambda$ is fixed. By introducing magnetic fluxes through the holes of the surfaces, he also considered systems that break time reversal symmetry, which were shown to exhibit the variance of GUE. The proof relies again on the Selberg trace formula (for the twisted Laplacian), along with estimates for the counting function of closed geodesics and probabilistic asymptotics for the model of random covers in the limit of large degree. \end{itemize}
	
	
	\subsection{Main results} 
	
	\subsubsection{Asymptotics for the variance} Our first goal is to adapt the results of Naud to the case of surfaces of variable negative curvature. 
	As in \cite{naud2022random}, instead of restricting our attention to the scalar Laplacian, we consider the twisted Laplacian $\Delta_\rho$ associated to a unitary representation $\rho$ of the fundamental group $\Gamma=\pi_1(X)$ (see \S \ref{sec : twisted laplacian} for the definition). When $\rho$ is Abelian (\emph{i.e.} $\rho$ takes values in $ \mathrm U(1)$), it amounts to introducing magnetic fluxes through the holes of the surface. It allows taking into account quantum systems that break time reversal symmetry, see \S \ref{time reversal} for details.
	
	Let us introduce some terminology:
	
	\begin{defi}[Variance of the smoothed counting function for the Gaussian ensembles] We set
		\[\Sigma_{\mathrm{GOE}}^2(\psi) = 4\int_{\mathbf R_+} t \widehat \psi^2(t) \mathrm dt, \qquad \Sigma_{\mathrm{GUE}}^2(\psi)=\frac 12 \Sigma_{\mathrm{GOE}}^2(\psi).\]
		These quantities correspond to the \emph{ensemble variance} of the smoothed counting function for the eigenvalues of matrices of the Gaussian Ensembles, in the limit of large matrices. According to Wigner's \emph{semicircle's law}, in the large $N$ limit the spectrum of a GOE/GUE matrix of size $N\times N$ is essentially contained in $[-2\sqrt N,2\sqrt N]$. Universal statistics in the large $N$ limit are valid in the \emph{bulk} of the spectrum, that is, when considering energy windows centered at $\lambda\in [(-2+\varepsilon)\sqrt N,(2-\varepsilon)\sqrt N]$ --- for any $\varepsilon>0$ ---, of width $\frac{1}{\sqrt N}\ll W\ll \sqrt N$. Since $\psi$ will be fixed in the paper, we shall not write the dependence of $\Sigma_{\rm GOE}$ and $\Sigma_{\rm GUE}$ in $\psi$. \end{defi}
	
	\begin{defi}[Breaking time reversal symmetry] A finite dimensional unitary representation $\rho:\Gamma\to \mathrm{U}(V_\rho)$ is said to \emph{break time reversal symmetry} if the map $\gamma\in \Gamma\mapsto \operatorname{Tr}(\rho(\gamma))\in \mathbf C$ isn't real valued, otherwise it is said to preserve time reversal symmetry. In particular, when $\rho$ is an Abelian character, $\rho$ breaks time reversal symmetry if and only if $\rho^2\neq 1$. We explain the grounds for this terminology in \S \ref{subsec:unitaryrep}.
		
		Now, if $\rho$ is a finite dimensional unitary representation of $\Gamma$, we let $\Sigma_{\rm GOE/GUE}(\rho)=\Sigma_{\rm GOE}$ if $\rho$ preserves time reversal symmetry, and let $\Sigma_{\rm GOE/GUE}(\rho)=\Sigma_{\rm GUE}$ otherwise. Since the representation $\rho$ we are dealing with will always be clear, we shall not mention it and just write $\Sigma_{\rm GOE/GUE}$.
	\end{defi}
	
	Consider a $n$-sheeted cover $X_n\to X$. If $\rho$ is a unitary representation of $\Gamma$, the Laplacian $\Delta_\rho$ on $X$ lifts to an operator $\Delta_{n,\rho}$ on $X_n$ --- this is described in more details in \S \ref{sec : random covers}. We denote by $N_n(\lambda,L,\rho)$ the smoothed counting function of eigenvalues of $\sqrt{\Delta_{n,\rho}}$, defined similarly to \eqref{eq:countdef}. We endow the ensemble of vertex-labeled $n$-sheeted covers $X_n\to X$ with the uniform probability distribution and denote by $\mathbf E_n[Z]$ the expected value of a random variable $Z$ with respect to this distribution. Then, $N_n(\lambda,L,\rho)$ becomes a random variable defined on the ensemble of $n$-sheeted covers of $X$. We show (compare with \cite[Theorem 1.1]{naud2022random}):
	\begin{thm}\label{mainthm} Let $(X,g)$ be a closed, connected surface, of negative curvature, with fundamental group $\Gamma$. Let $\rho$ be an Abelian character of $\Gamma$. Let $\psi$ be a smooth, real, even function, with compactly supported Fourier transform. Denote by $N_n(\lambda,L,\rho)$ the smoothed counting function for the twisted operator $\sqrt{\Delta_{n,\rho}}$ on the cover $X_n\to X$, defined similarly to \ref{eq:countdef}, and let
		\[\Sigma_n^2(\lambda,L,\rho)=\mathbf E_n\Big[\big(N_n(\lambda,L,\rho)-\mathbf E_n(N_n(\lambda,L,\rho))\big)^2\Big]\]
		be the variance of $N_n(\lambda,L,\rho)$ over the set of $n$-sheeted covers $X_n\to X$. 
		
		\begin{enumerate}[\normalfont(i)]
			\item \emph{\textbf{Existence of the limit.}} For any $\lambda,L>0$, the limit
			\[\Sigma^2(\lambda,L,\rho):=\underset{n\to +\infty}\lim \Sigma_n^2(\lambda,L,\rho)\]
			is well-defined, and finite.
			
			Assume that $\widehat \psi$ is supported in $[-1,1]$. There exists some $c_0>0$ depending only on the geometry of $X$ such that if $L=L(\lambda)$ satisfies $1\ll L\le c_0\log \lambda$, the following holds:
			
			\item \emph{\textbf{Oscillations between $0$ and $2\Sigma_{\mathrm{GOE/GUE}}^2(\psi)$.}} We have 
			\[0\le \Sigma^2(\lambda,L,\rho)\le 2\Sigma^2_{\mathrm{GOE/GUE}}+\mathcal O(L^{-2}).\]
			Moreover, we can construct a function $L=L(\lambda)$ such that there are sequences $(\lambda_j^{\pm})$ going to $+\infty$ with
			\[\Sigma^2(\lambda_j^{\pm},L(\lambda_j^{\pm}),\rho)\underset{j\to +\infty}\longrightarrow (1\pm 1)\Sigma^2_{\rm GOE/GUE},\] 
			implying that $\Sigma^2(\lambda,L(\lambda),\rho)$ fluctuates in the whole interval $[0,2\Sigma^2_{\rm GOE/GUE}]$.

			\item \emph{\textbf{Convergence on average to $\Sigma_{\mathrm{GOE/GUE}}^2$.}} Take $\delta=\delta(\lambda)$ such that $\delta(\lambda)\gg  L(\lambda)^{-1}$. Then,
			\begin{equation}\label{eqthmave}\underset{\lambda\to +\infty}\lim \left(\frac{1}{\delta}\int_{\lambda}^{\lambda+\delta }\Sigma^2(\mu,L(\lambda),\rho)\mathrm d\mu\right) =\Sigma^2_{\rm GOE/GUE} .\end{equation}
			\item \emph{\textbf{Quadratic convergence on average to $\Sigma_{\mathrm{GOE/GUE}}^2$.}} Take $\Lambda=\Lambda(\lambda)$ such that $\Lambda \gg 1$. Then,
			\begin{equation}\label{eqthmquad}\underset{\lambda\to +\infty}\lim\left(  \frac{1}{\Lambda}\int_\lambda^{\lambda+\Lambda} \big|\Sigma^2(\mu,L(\lambda),\rho)-\Sigma_{\mathrm{GOE/GUE}}^2\big|^2 \mathrm d\mu\right) =0.\end{equation}
		\end{enumerate}
	\end{thm}
	
	When $L$ is a reasonable function of $\lambda$, we can turn the quadratic convergence on average of Theorem \ref{mainthm} to convergence along a subset of density $1$ :
	\begin{cor}[Convergence along a subset of density $1$]\label{cor : density} For any $C^1$ function $L=L(\lambda)$ satisfying $1\ll L\le c_0\log \lambda$ and $L'(\lambda)\ll L(\lambda)$, there exists a subset $A\subset (0,+\infty)$ of density $1$ such that
		\[\underset{\lambda\in A\to +\infty}\lim \Sigma^2(\lambda,L(\lambda),\rho)=\Sigma^2_{\rm GOE/GUE}, \]We recall that $A$ having density $1$ means that the ratio $|A\cap [0,\Lambda]|/\Lambda$ converges to $1$ as $\Lambda$ goes to $+\infty$. Here $|\cdot|$ denotes the Lebesgue measure.
	\end{cor}
	It is already striking that  $ \lim_n \Sigma_n^2(\lambda,L,\rho)$ is finite. Indeed, the leading term in Weyl's Law for the smoothed counting function $N(\lambda,L)$ on $X_n$ is given by $n\times \frac{\mathrm{vol}(X)\lambda }{L} \widehat \psi(0)$, which goes to infinity as $n$ gets large. Still, the ensemble variance converges, which implies strong rigidity for the spectrum of covers of $X$. We make a few comments:
	
	\begin{itemize}\item The semiclassical approximation that we rely on only allows reaching $L=c_0\log \lambda$ above, while it is conjectured that universal statistics hold up to $L\ll \lambda$, that is when considering windows of width $\frac 1L$ that are large compared to the mean level spacing $\lambda^{-1}$. When working with $\Delta_g$ instead of $\sqrt{\Delta_g}$, this amounts to considering eigenvalues of $\Delta_g$ in windows $[E-W,E+W]$ in the regime $\frac{\sqrt{E}}{c \log E}\le W\ll \sqrt{E}$ for some other constant $c$, which is more or less the setting we discussed in the introduction, though the condition $W\ge \frac{\sqrt E}{c\log E}$ is way stronger than $W\gg 1$.
		

		\item The results of Naud hold for fixed $\lambda$, without averaging the ensemble variance over the energy $\lambda$. In variable curvature this cannot be done since the Gutzwiller trace formula only holds in the large $\lambda$ regime (see \S \ref{subsubsec: gutztrace}), and being in this high energy regime forces us to average the ensemble variance to recover GOE statistics, due to the presence of fast oscillating terms. However, this necessity isn't specific to surfaces of nonconstant curvature: if one tries to let $\lambda\to +\infty$ in Naud's proof, it also becomes necessary to perform an energy averaging to recover their result.
		
		

	\end{itemize}
	
	As in \cite{naud2022random}, we can include the case of higher dimensional representations, see \cite[Theorem 1.2]{naud2022random}.
	
	\begin{thm}\label{mainthm2} Let $\rho:\Gamma\to \mathbf G$ be a unitary representation of $\Gamma$, where $\mathbf G$ is a compact Lie subgroup of $\mathrm U(N)$ for some $N\ge 1$. Assume that $\rho(\Gamma)$ is dense in $\mathbf G$. Then all the statements of Theorem \ref{mainthm} hold true in this setting, by replacing $\Sigma_{\rm GOE/GUE}^2$ by the constant 
		\[\Sigma_{\mathbf G}^2:=\Big( \int_{\mathbf G} f(g)\mathrm dg\Big)\Big(\int_0^{+\infty} t\widehat \psi(t)^2\mathrm dt\Big),\]
		where $f(g)=\big(\operatorname{Tr}(g)+\overline{\operatorname{Tr}(g)}\big)^2$ and $\mathrm dg$ denotes the normalized Haar measure on $\mathbf G$. In particular, if $\mathbf G=\mathrm{SO}(N),\mathrm{Sp}(N)$ --- including $\mathrm{SU}(2)=\mathrm{Sp}(1)$ ---, then $\Sigma_{\mathbf G}=\Sigma_{\rm GOE}$, while for $\mathbf G=\mathrm{U}(N),\mathrm{SU}(N)$ for $N\ge 3$ we have $\Sigma_{\mathbf G}=\Sigma_{\rm GUE}$.
	\end{thm}
	
	\begin{rem}Let us explain the denseness assumption on the image of the representation $\rho$.
		
		\begin{itemize}
			\item When $\rho$ takes values in a Lie group $\mathbf G$, the geodesic flow on $S^*X$ induces a flow on a $\mathbf G$-principal bundle $P\to S^*X$ defined as the quotient
			\[P=\Gamma\backslash (S^*\widetilde X\times \mathbf G), \qquad \gamma\cdot \big((x,\xi),g\big)=((\mathrm d\gamma^{\top})^{-1} (x,\xi), \rho(\gamma) g).\]
			By a result of Ceki{\'c}--Lefeuvre \cite[Example 4.2.6]{cekić2024semiclassicalanalysisprincipalbundles}, the representation $\rho$ has dense image in $\mathbf G$ if and only if the extension of the geodesic flow to $P$ is mixing with respect to the product measure $\nu_0\otimes \mathrm dg$ (here $\nu_0$ is the normalized Liouville measure on $S^*X$). Therefore, the denseness assumption is natural since we aim to study chaotic dynamics.
			\item Naud makes a different assumption on the representation $\rho$, namely that its image is Zariski-dense in $\mathbf G$ \cite[Theorem 1.2]{naud2022random}. We have preferred the above one as it relates naturally to the mixing properties of the extended dynamics. 
	\end{itemize}\end{rem}

	\subsubsection{Almost sure GOE fluctuations} In the following, to make the presentation simpler, we go back to the case where $\rho$ is Abelian.
	
	Consider the centered variable $\widetilde N_n=N_n-\mathbf E_n(N_n)$. We already studied the \emph{ensemble average} of $\widetilde N_n^2$ in Theorem \ref{mainthm}. Now we are interested, for an individual cover $X_n\to X$, in the \emph{energy average}
	\[\frac{1}{\Lambda}\int_\lambda^{\lambda+\Lambda} \widetilde N_n^2(\mu,L,\rho) \mathrm d\mu,\]
	where again $\Lambda=\Lambda(\lambda)\gg 1$. We show that in the high energy regime, with high probability as $n\to +\infty$, the \emph{energy average} of $\widetilde N_n^2$ is close to $\Sigma_{\mathrm{GOE/GUE}}^2$. This is an analogous result to \cite[Theorem 1.6]{Maoz2} in variable curvature. The version for the Weil--Petersson model was shown by Rudnick--Wigman \cite{rudnick2023sure}. 
	\begin{thm}[Almost sure GOE fluctuations]\label{thm:asGOEfluctuations} Let $\rho$ be an Abelian character of $\Gamma$. There is a small constant $c_0>0$ such that in the regime $1\ll L\le c_0\log \lambda$, and $\Lambda=\Lambda(\lambda)\gg 1$, for any $\varepsilon>0$,
		\[\underset{\lambda \to +\infty}\lim \underset{n\to +\infty}\limsup \ \mathbf P_n\left(\left|\frac{1}{\Lambda}\int_\lambda^{\lambda+\Lambda } \widetilde N_n^2(\mu,L,\rho)\mathrm d\mu-\Sigma_{\mathrm{GOE/GUE}}^2\right|>\varepsilon\right)=0,\]
		where the value of $\Sigma_{\mathrm{GOE/GUE}}^2$ depends on whether $\rho$ breaks time reversal symmetry. \end{thm}
	
	As we will see in the proof, we can allow $\varepsilon$ to depend on $\lambda$ and take $\varepsilon=\varepsilon(\lambda)$ such that $\varepsilon^2\gg (L^{-1}+\Lambda^{-1})$.
	
	\begin{rem} On the one hand, by Corollary \ref{cor : density}, for almost any high energy level $\lambda$, in the large $n$ limit, the \emph{ensemble average} of $\widetilde N_n(\lambda,L)^2$ is close to $\Sigma_{\mathrm{GOE/GUE}}^2$. On the other hand, Theorem \ref{thm:asGOEfluctuations} tells us that in the high energy regime, in the large $n$ limit, for almost any cover $X_n\to X$, the \emph{energy average} of $\widetilde N_n(\mu,L)^2$ over $[\lambda,\lambda+\Lambda]$ is close to $\Sigma_{\mathrm{GOE/GUE}}^2$. Combining these two facts shows that in the high energy regime, with high probability as $n \to +\infty$, the energy average of $\widetilde N_n^2$ for an individual cover approaches the ensemble average of $\widetilde N_n^2$. This feature is called \emph{ergodicity} in Random Matrix Theory. With this perspective, as noted by Ullmo \cite{Ullmo2016BohigasGiannoniSchmitC}, exceptional surfaces for which BGS fails may just occur as a subset of measure $0$ of an ensemble of systems under consideration, which would be coherent with the behavior observed for the Gaussian Ensembles of matrices.
	\end{rem}
	
	\subsubsection{The Central Limit Theorem for the counting function on random covers} The second moment $\mathbf E_n[\widetilde N_n^2]$ is by now quite well understood by Theorem \ref{mainthm}. Actually, we can show that moments of higher order converge to that of a Gaussian, implying a central limit theorem for $\widetilde N_n$. An analogous version was proven in constant curvature by Maoz \cite{Maoz2}, whereas the version for the Weil--Petersson random model was obtained by Rudnick--Wigman \cite{rudnick2023central}.
	\begin{thm}[Central limit theorem for $\widetilde N_n(\lambda,L)$]\label{thm : CLT}Let $\rho$ be an Abelian character of $\Gamma$. Take two sequences $(L_j), (\lambda_j)$ such that $L_j\le c_0 \log \lambda_j$, and $L_j\to +\infty$. Assume $\Sigma^2(\lambda_j,L_j)\gg \frac{1}{L_j^2}$. Then, for any bounded continuous function $h$, 
		\[\underset{j\to +\infty}\lim \left[  \underset{n\to +\infty}\lim \mathbf E_n\left(h\left(\frac{\widetilde N_n(\lambda_j,L_j,\rho)}{\Sigma(\lambda_j,L_j,\rho)}\right)\right)\right]=\frac{1}{\sqrt{2\pi}}\int_{\mathbf R} h(t)\mathrm{e}^{-\frac 12t^2}\mathrm dt.\]
	\end{thm}
	
	Our statement is more intricate than those in \cite{Maoz2,rudnick2023central}. Indeed, since we have to let $\lambda\to +\infty$, the value of $\Sigma^2(\lambda,L)$ may fluctuate in the whole interval $[0,2\Sigma^2_{\rm GOE/GUE}]$ and need not converge to $\Sigma_{\rm GOE/GUE}^2$. Our proof allows us to get the CLT under the condition $\Sigma^2(\lambda,L)\gg L^{-2}$. Actually, by Corollary \ref{cor : density} we may take a reasonable function $L=L(\lambda)$ with $1\ll L\le c_0\log \lambda$, such that $\Sigma^2(\lambda,L)\to \Sigma^2_{\rm GOE/GUE}$ along a subset $A$ of density $1$, in which case we have
	\[\underset{\lambda\in A\to +\infty}\lim \left[  \underset{n\to +\infty}\lim \mathbf E_n\left(h\left(\frac{\widetilde N_n(\lambda,L,\rho)}{\Sigma_{\rm GOE/GUE}}\right)\right)\right]=\frac{1}{\sqrt{2\pi}}\int_{\mathbf R} h(t)\mathrm{e}^{-\frac 12t^2}\mathrm dt.\]

	\subsection{Outline of the proof} \label{subsec:strategie}	Let us sketch the proof of Theorem \ref{mainthm} in the case of the scalar Laplacian. Similar techniques will be used in the proofs of Theorems \ref{thm:asGOEfluctuations} and \ref{thm : CLT}, along with probabilistic results of Maoz \cite{Maoz1} for the model of random covers. Again, let $(X,g)$ be a compact, negatively curved surface, denote by $\widetilde X$ its universal cover, and $\Gamma$ the group of deck transformations of this cover, so that $X=\Gamma\backslash \widetilde X$. Observe that $\Gamma$ can be identified with the fundamental group of $X$.
	
	\subsubsection{The wave trace and the Gutzwiller trace formula}\label{subsubsec: gutztrace}
	Following the seminal work of Berry \cite{Berry1}, fluctuations of the density of energy levels have been studied by using the Gutzwiller trace formula, relating the distribution of energy levels of the system to the periodic orbits of its classical limit. Let us describe it for the square root of the positive Laplacian on $(X,g)$.
	
	Denote $0=\lambda_1\le \lambda_2\le \ldots$ the square roots of the eigenvalues of the positive Laplacian $\Delta_X$. The spectral distribution $\sum_j \delta(\lambda-\lambda_j)$ is nothing but the Fourier transform of the wave trace 
	\begin{equation}\label{eq:wavetrace}\operatorname{Tr} \mathrm{e}^{-\mathrm it\sqrt{\Delta}}:=\sum \mathrm{e}^{-\mathrm i\lambda_jt}.\end{equation} Here $\mathrm{e}^{-\mathrm it\sqrt{\Delta}}$ is the unitary propagator associated to the half-wave operator $-\mathrm i\partial_t +\sqrt{\Delta}$. In the mathematical literature, the description of the singularities of $\operatorname{Tr} \mathrm{e}^{-\mathrm it\sqrt{\Delta}}$ goes back to the works of Chazarain \cite{Chazarain1974} and Duistermaat--Guillemin \cite{Duistermaat1975} --- see also the work of Colin de Verdière \cite{ColindeVerdière1973} on the heat equation.
	For a negatively curved surface, the singular support of the wave trace is given by $\mathcal L\cup \{0\}$, where $\mathcal L$ is the set of lengths of closed geodesics.
	The leading term of the singularity at $\ell(\gamma)$ is given by
	\[\frac{1}{2\pi}\frac{\ell(\gamma)^{\sharp} }{|\det(I-\mathcal P_\gamma)|}(t-\ell(\gamma)+\mathrm i0)^{-1}.\] 
	Here, $\ell(\gamma)$ denotes the period of a closed geodesic $\gamma$, $\ell(\gamma)^\sharp$ its primitive period, and $\mathcal P_\gamma$ is the linearized Poincaré map. The function $\psi$ involved in the definition of $N(\lambda,L)$ is real and even, thus, recalling \eqref{eq:wavetrace}, we may write
	\begin{equation}\label{eq:N as wave}N(\lambda,L)=\frac 2L\big\langle \operatorname{Tr} \cos(t\sqrt{\Delta}),\cos(\lambda t)\widehat \psi(t/L)\big\rangle_{\mathcal S'(\mathbf R),\mathcal S(\mathbf R)}.\end{equation} 
	Here $\cos(t\sqrt{\Delta})$ simply denotes the real part of $\mathrm{e}^{-\mathrm it\sqrt{\Delta}}$. We will allow $L$ going to $+\infty$, therefore we need a quantitative version of the Duistermaat--Guillemin trace formula that holds for long times (thus involving an exponentially increasing number of closed geodesics), with adequate control on the error term. Such a formula was used by Jakobson--Polterovitch--Toth \cite{JakobsonWeyl2007} to get a lower bound on the remainder in Weyl's law. We state it for twisted Laplacians.
	
	\begin{prop}[Trace formula for the twisted Laplacian] \label{trace formula for long times} Let $\rho:\Gamma\to \mathrm{GL}(V_{\rho})$ be a finite dimensional unitary representation of $\Gamma$, and denote by $\chi_\rho$ the associated character. Since oriented closed geodesics are in bijection with conjugacy classes of $\Gamma$, we can consider $\chi_\rho(\gamma)$ for $\gamma$ in the set of oriented closed geodesics $\mathcal G$. Then, there is some constant $C$ depending on $X,\psi$ such that
		\begin{equation}\label{eq : trace intro}\begin{array}{l} \big\langle \operatorname{Tr} \cos(t\sqrt{\Delta_{\rho}}), \cos(\lambda t)\widehat \psi(t/L)\big\rangle\displaystyle = \dim(V_\rho)\Big( \frac 12\lambda\mathrm{vol}(X)\widehat \psi(0)+\mathcal O\big(\lambda^{-1} \mathrm{e}^{CL}\big)\Big)\\ \qquad \qquad \qquad  \displaystyle +\sum_{\gamma\in \mathcal G} \chi_\rho(\gamma)\left(\cos(\lambda \ell(\gamma))\widehat\psi\Big(\frac{\ell(\gamma)}{L}\Big)\frac{\ell(\gamma)^{\sharp}}{|I-\mathcal P_\gamma|^{\frac 12}} +\mathbf 1_{\ell(\gamma)\le L}\mathcal O\big(\lambda^{-\frac 12}\mathrm{e}^{CL}\big)\right).  \end{array}\end{equation}
		
		Here and from now on we denote shortly $|I-\mathcal P_\gamma|=|\det(I-\mathcal P_\gamma)|$. The error terms $\mathcal O(\lambda^{-\frac 12}\mathrm{e}^{CL})$ are functions that do not depend on the representation $\rho$ (but depend on $\psi$). Also, bounds on the error terms are uniform with respect to $\gamma\in \mathcal G$.
	\end{prop}
	
	The first term of the right-hand side of \eqref{eq : trace intro}, often called the \emph{smooth part}, depends only on the volume of $X$, modulo the error term $\mathcal O(\lambda^{-1}\mathrm{e}^{CL})$. It yields the leading term in Weyl's law. Actually, if we restrict ourselves to the study of hyperbolic surfaces of genus $g$ --- that share the same volume $4\pi(g-1)$ ---, or to the set of $n$-sheeted covers of a fixed surface $X$, the smooth part does not depend on the surface under consideration. The second term, often called the \emph{oscillating part}, depends on the surface; it carries geometrical data of closed geodesics and determines the structure of the correlations between the energy levels.

	We point out that in the case of hyperbolic surfaces, there is an exact identity involving the trace of $\cos(t\sqrt{\Delta-\frac 14})$, given by the Selberg trace formula. In variable curvature, due to the error terms in the trace formula \eqref{eq : trace intro}, we are naturally bound to times $L\sim\log(\lambda)$. It means that we consider windows centered around an energy $\lambda$ that contain at least $\sim \frac{\lambda}{\log \lambda}$ eigenvalues. To resolve individual energy levels with the trace formula --- \emph{i.e.} to consider windows containing one eigenvalue on average ---, we would have to reach times $L\sim \lambda$, making error terms impossible to control.
	Notice that even though it may seem silly to keep the remainders in each term of the sum, it will appear useful as we will be led to allow the representation to vary.
	
	\subsubsection{From twisted Laplacians on $X_n$ to twisted Laplacians on $X$} 
	We shall see in \S \ref{sec : random covers} that the scalar Laplacian $\Delta_n$ on $X_n$ is unitarily conjugated to the twisted Laplacian $\Delta_{f_n}$ on $X$ associated to a certain unitary representation $f_n$ of rank $n$ of $\Gamma$. More generally, the twisted Laplacian $\Delta_{n,\rho}$ on $X_n$ is conjugated to the twisted Laplacian $\Delta_{\rho\otimes f_n}$ on $X$. 
	
	Notice that as the first term of the right hand-side of (\ref{eq : trace intro}) depends only on the rank of the representation, it takes the same value for every cover $X_n\to X$ of degree $n$, thus will vanish when computing the variance.

	\subsubsection{The correlations between lengths of closed geodesics} \label{subsubsec:difficulties} Let us first describe the difficulties that one faces when studying the number variance for an individual surface. Write, according to Proposition \ref{trace formula for long times} --- here we discard all error terms outside the sums ---,
	\begin{equation}\label{eq:difficulties}N(\lambda,L) = \frac\lambda L\mathrm{vol}(X)\widehat \psi(0) \displaystyle+\frac 2L\sum_{\gamma\in \mathcal G} \left(\cos(\lambda \ell(\gamma))\widehat\psi\left(\frac{\ell(\gamma)}{L}\right)\frac{\ell(\gamma)^{\sharp}}{|I-\mathcal P_\gamma|^{\frac 12}} \right) +\mathcal O(\lambda^{-\frac 12}\mathrm{e}^{CL}),\end{equation}
	where $C$ may have increased, compared to that of Proposition \ref{trace formula for long times}. We denote the leading term of the right hand-side of \eqref{eq:difficulties} by $\overline N(\lambda,L)$, while the oscillating sum over closed geodesics is denoted $N_{\rm osc}(\lambda,L)$. When considering an individual surface, one is primarily concerned by the fluctuations of $N(\lambda,L)$ around the smooth part $\overline{N}(\lambda,L)$, that is \emph{energy averages} of $N_{\rm osc}(\cdot,L)^2$ over a window centered at $\lambda$. Studying these fluctuations requires understanding sums of terms of the form $A_{\gamma}A_{\gamma'}\mathrm{e}^{\rm i\lambda(\ell(\gamma)\pm\ell(\gamma'))}$. Indeed, letting $A_\gamma$ denote the dynamical coefficient
	\[A_\gamma=\frac 2L\frac{\ell(\gamma)^\sharp}{|I-\mathcal P_\gamma|^{\frac 12}}\widehat \psi(\ell(\gamma)/L),\]
	we can write
	\begin{equation}\label{eq:nondiag}N_{\rm osc}^2(\lambda,L)=\sum_{\gamma,\gamma'} \cos(\lambda \ell(\gamma))\cos(\lambda \ell(\gamma'))A_\gamma A_{\gamma'}.\end{equation}
	In the physics literature, Berry \cite{Berry1}  --- who considered sharp counting functions --- showed that by performing a diagonal approximation (\emph{i.e.} keeping only pairs $(\gamma,\gamma')$ for which $\gamma'\in\{\gamma,\gamma^{-1}\}$), we recover the number variance of GOE. Later, Sieber--Richter \cite{Sieber_2001} identified certain pairs of closed geodesics which allow recovering lower order terms in the asymptotics of the number variance. The full GOE expansion for the number variance was derived (non rigorously) in \cite{PhysRev}. This cannot however be considered a satisfying proof of BGS, as even though the authors identified pairs of orbits that allow to recover the full expansion of GOE, they used the Gutzwiller trace formula up to times of order $\lambda$, where its validity is speculative, and didn't show that the remaining pairs of orbits have a negligible contribution. 
	
	A major difficulty is to obtain nontrivial cancellations between summands in the oscillating sum \eqref{eq:nondiag}. By now, we only seem to be able to bound each non-diagonal term individually. Indeed, in order to control off-diagonal terms, one needs to understand the correlations of lengths of pairs of closed geodesics, that is the ensemble of differences $\{\ell(\gamma)-\ell(\gamma')\}$, of which not much is known. Let us still mention the work of Pollicott--Sharp \cite{Pollicott2006} on pair correlations --- see also \cite{Pollicott2013} for an extension to higher dimensions ---, and its analogue for billiard flows obtained by Petkov--Stoyanov \cite{Petkov2009}. See also \cite{DmitryDolgopyat2016JournalofModernDynamics,EmmanuelSchenck2020JournalofModernDynamics} for discussions on the size of gaps in the length spectrum.

	We point out that the failure of BGS for arithmetic surfaces is due to the presence of huge degeneracies in the length spectrum of such surfaces --- in which case the diagonal approximation fails badly. This feature is however expected to be exceptional. We refer to the expository note of Marklof \cite{Marklof2006ArithmeticQC} for a review of different conjectures in quantum chaos for arithmetic surfaces.  
	
	\subsubsection{The ensemble variance in the large $n$ limit.}
	
	Working with random covers allows suppressing the non-diagonal terms rigorously, the counterpart being that we obtain results on average for an ensemble of surfaces, and with a slightly different flavor, see \S \ref{subsec:diffRud}. From now on we assume $\rho=1$ in order to make the presentation simpler. We apply the Gutzwiller trace formula for the representation $ f_n$ (it depends on the cover $X_n\to X$, thus $f_n$ is a random variable) to write $N_n(\lambda,L)$ as the sum of a smooth part and an oscillating sum involving closed geodesics. The computation of the ensemble variance makes use of asymptotics for the covariances of the random variables $\mathrm{Tr}(f_n(\gamma)), \ \gamma\in \Gamma$ reviewed in \S\ref{sec : random covers}, together with estimates on the contribution of nonprimitive closed geodesics. We work in the regime $L\le c_0\log \lambda$ in order to make error terms in the trace formula negligible. After taking the limit $n\to +\infty$, we find
	\begin{equation}\label{eq:Sigmacarréintro}\Sigma^2(\lambda,L)=\underset{n\to +\infty}\lim \Sigma_n^2(\lambda,L)=\frac{8}{L^2}\sum_{\gamma\in \mathcal G} \frac{\ell(\gamma)^{\sharp}\ell(\gamma) \cos^2(\lambda \ell(\gamma)) \widehat\psi^2(\ell(\gamma)/L)}{|I-\mathcal P_\gamma|}+\mathcal O\left(\frac{1}{L^2}\right).\end{equation}
	Here, we cannot exclude fast oscillations of $\Sigma^2(\cdot,L)$ due to the presence of the cosines. Averaging over a sufficiently wide range of frequencies allows taming these oscillations. Indeed, we obtain
	\begin{equation}\label{sketch moyenne}\frac{1}{\delta}\int_\lambda^{\lambda+\delta} \Sigma^2(\mu,L)\mathrm d\mu=\frac{4}{L^2}\sum_{\gamma\in \mathcal G} \frac{\ell(\gamma)^{\sharp}\ell(\gamma) \widehat\psi^2(\ell(\gamma)/L)}{|I-\mathcal P_\gamma|}+\mathcal O\left(\frac{1}{\delta L}+\frac{1}{L^2}\right).\end{equation} Observe that the leading term does not depend on $\lambda$ anymore. The assumption $\delta \gg \frac 1L$ appears naturally in order to be able to discard the error terms.
	\subsubsection{The equidistribution of closed geodesics}
	
	To estimate \eqref{sketch moyenne} in the large $L$ limit, we need to exploit some equidistribution result on lengths of closed geodesics. In constant curvature $-1$, since the relevant dynamical quantities associated to the orbits only depend on their lengths, it is somehow sufficient to use the standard asymptotic for the counting function of primitive closed geodesics (that can be derived from Selberg trace formula):
	\begin{equation}\label{comptage}\# \big\{\gamma\in \mathcal G, \ \ell(\gamma)\le L\big\}\underset{L\to +\infty}\sim \frac{\mathrm{e}^L}{L}.\end{equation} 
	Although a formula like \eqref{comptage} holds in variable negative curvature, to find the limit of (\ref{sketch moyenne}), we will rather use the following formula, deriving from a result of Jin--Zworski \cite{Jin2016}, that is exactly suited to our problem: if $f$ is a smooth, real, compactly supported function, then
	\begin{equation}\label{JinZwo}\lim_{L\to +\infty} \frac 1L\sum_{\gamma\in \mathcal G }\frac{\ell(\gamma)^\sharp f(\ell(\gamma)/L)}{|I-\mathcal P_\gamma|}=\int_0^{+\infty}f(t)\mathrm dt.\end{equation}
	To find the limit of \eqref{sketch moyenne} when $L\to +\infty$, it suffices to apply this formula to $t\mapsto t\widehat \psi^2$, giving
	\[\underset{\lambda\to+\infty}\lim \Big(\frac{1}{\delta}\int_\lambda^{\lambda+\delta} \Sigma^2(\mu,L)\mathrm d\mu\Big)=4\int_{\mathbf R_+} t\widehat \psi^2(t)\mathrm dt=\Sigma_{\rm GOE}^2.\]
	
	\begin{rem} Formula (\ref{JinZwo}) can be thought of as a particular version of the semiclassical sum rule of Hannay--Ozorio de Almeida \cite{JHHannay1984} that appears in the physics literature, and was used by Berry \cite{Berry1} to derive the first term in the GOE expansion.
		
	\end{rem}

	\subsection{Differences between $\mathbf E_n(N_n)$ and the smooth part.}	\label{subsec:diffRud} We stress out that we are not proving \say{BGS on average}. Let us write
	\[N_n=\overline{N_n} +N_{n,\rm osc}=\mathbf E_n[N_n]+\widetilde N_n.\]
	Theorem \ref{thm:asGOEfluctuations} tells us that in the high energy regime, with high probability as $n\to +\infty$, the spectral fluctuations of $\widetilde N_n$ obey GOE/GUE statistics, but it says nothing about spectral fluctuations of $N_{n,\rm osc}$ itself (which is the focus of the BGS conjecture). This is due to the fact that $\mathbf E_n(N_n)$ doesn't coincide with the Weyl term $\overline{N_n}$. Indeed, if $L \ll \log \lambda$, using the asymptotics for the model of random covers, we show 
	\begin{equation}\label{expectation}\underset{n\to +\infty}\lim \big(\mathbf E_n(N_n(\lambda,L))-\overline N_n(\lambda,L)\big)=\frac{2}{L}\sum_{\gamma\in \mathcal G} \Big(\cos(\lambda \ell(\gamma)) \frac{\ell(\gamma)^\sharp\widehat \psi(\ell(\gamma)/L)}{|I-\mathcal P_\gamma|^{\frac 12}}\Big) +\mathcal O(1),\end{equation} 
	Up to a bounded error term, the right-hand side of \eqref{expectation} equals the oscillating part $N_{1,\rm osc}(\lambda,L)$ of the counting function on the base surface $X$, allowing to write, for $n$ large,
	\[N_{n,\rm osc}=\widetilde N_n+N_{1,\rm osc} +\mathcal O(1),\]
	where $N_{1,\rm osc}$ denotes the oscillating part on the base surface $X$. Theorem \ref{mainthm} holds for any base surface $X$, in particular for \emph{arithmetic surfaces}. Since the covers of an arithmetic surface are still arithmetic surfaces, the result of Luo--Sarnak tells that $N_{n,\rm osc}$ has large spectral fluctuations. It doesn't contradict our result, that is only concerned by $\widetilde N_n$. 
	
	This differs somehow from the result of Rudnick \cite{rudnick2022goe}. For the Weil--Petersson model, Rudnick showed that
	\[ \lim_{g\to +\infty} \left(\mathbf E_g\big(N(\lambda,L)\big)-\overline N(\lambda,L)\right)=\frac 4L\int_{\mathbf R_+}\sum_{k=1}^{\infty} \widehat \psi\left(\frac{kx}{L}\right)\frac{\sinh^2(x/2)}{\sinh(kx/2)}\cos(\lambda kx)\mathrm dx.\]
	This quantity converges to $0$ as $L$ gets large, provided that $L\le c\log \lambda$ for some sufficiently small $c$. When averaging functions over the moduli space of hyperbolic surfaces of genus $g$, one gets some oscillatory integrals that can be shown to go converge to $0$ by performing integrations by parts. Such techniques fail in the case of the discrete model of random covers.
	
	\subsection{The effect of unitary representations} \label{subsec:unitaryrep}
	
	\subsubsection{Abelian characters and the Ahronov--Bohm effect} \label{time reversal}
	
	What is the physical meaning of introducing a unitary character $\rho:\Gamma\to \mathrm U(1)$ ? There is a well-known correspondence between unitary representations of $\Gamma$ and flat vector bundles over $X$. We try to make it explicit in this simple case. If $\rho:\Gamma \to \mathrm U(1)$ is a character, the Laplacian $\Delta_\rho$ is conjugated to the magnetic Laplacian associated to a vector potential $\mathbf A\in C^\infty(X,T^*X)$, for which the magnetic field vanishes everywhere --- meaning $\mathrm d\mathbf A=0$. The classical dynamics are left unchanged, but despite the absence of an electromagnetic field, a wave function can gain some phase factors due to the presence of $\mathbf A$, a phenomenon called the Ahronov--Bohm effect in physics. With a mathematical perspective, this is explained by the fact that flat bundles over $X$ may have nontrivial holonomies when $X$ has a nontrivial fundamental group. 
	
	If $\mathbf A$ is a closed $1$-form on $X$, one can lift $\mathbf A$ to a $\Gamma$-invariant vector potential $\widetilde{\mathbf A}$ on $\widetilde X$. Since $\widetilde X$ is simply connected and $\mathrm d\widetilde {\mathbf  A}=0$, by Poincaré Lemma, we can write $\widetilde{\mathbf A}=\mathrm d\varphi$ for some function $\varphi$. The condition $\mathrm d\varphi(\gamma x)=\mathrm d\varphi(x)$ ensures that $\varphi(\gamma x)-\varphi(x)$ does not depend on $x$, by Stokes theorem. Indeed, this quantity equals $\int_{\gamma} \mathbf A$, which can be interpreted as the flux of an exterior magnetic field through the holes of the surface.
	
	Set $\rho(\gamma)=\mathrm{e}^{-\mathrm i(\varphi(\gamma x)-\varphi(x))}=\mathrm{e}^{-\mathrm i\int_{\gamma} \mathbf A}$. Then the transformation $f\mapsto \mathrm{e}^{\mathrm i\varphi}f$ maps $\Gamma$-invariant functions to $\rho$-equivariant functions. The unitary transformation $\mathrm{e}^{\mathrm i\varphi}$ gives a conjugation 
	\[\mathrm d+\mathrm i\widetilde{\mathbf A}=\mathrm d+\mathrm i(\mathrm d\varphi)=\mathrm{e}^{-\mathrm i\varphi}\mathrm d(\mathrm{e}^{\mathrm i\varphi}\bullet).\]
	Consequently, the magnetic Laplacian $(\mathrm d+\mathrm i\widetilde{\mathbf A})^*(\mathrm d+\mathrm i\widetilde{\mathbf A})$ acting on $\Gamma$-invariant functions $\widetilde X\to \mathbf C$ is conjugated to the Laplacian $\Delta_{\widetilde X}$ acting on $\rho$-equivariant functions $\widetilde X\to \mathbf C$. It descends to a unitary conjugation between $(\mathrm d+\mathrm i\mathbf A)^*(\mathrm d+\mathrm i\mathbf A)$ and the twisted Laplacian $\Delta_\rho$, that acts on sections of the flat bundle $\Gamma\backslash (\widetilde X\times\mathbf C)$, defined as the quotient of $\widetilde X\times \mathbf C$ by the relation $(\tilde x,z)\sim (\gamma \tilde x,\rho(\gamma) z)$.
	
	The geodesic flow $\phi^t$ on the cotangent bundle $T^*X$ is time reversal invariant under the transformation $(x,\xi)\mapsto (x,-\xi)$, meaning $\phi^t(x,\xi)=\phi^{-t}(x,-\xi)$. Reversing the direction of $\xi$ amounts, at the quantum level, to consider the complex conjugate of the wave function. We introduce the \emph{time reversal operator} $\mathbf T$ on $C^{\infty}( \widetilde X,\mathbf C)$, defined by $\mathbf T\psi(x)=\overline{\psi(x)}$. It is clear that $\mathbf T$ commutes with $\Delta_{\widetilde X}$, but for $\mathbf T$ to preserve the space of  $\rho$-equivariant functions, one needs $\rho(\gamma)=\rho(\gamma^{-1})$ for every $\gamma$. In this case, $\mathbf T$ descends to an operator on $C^{\infty}(X,\Gamma\backslash(\widetilde X\times \mathbf C))$ that commutes with the twisted Laplacian $\Delta_\rho$. Thus, if $\psi(t,x)$ solves $(-\mathrm i\partial_t +\sqrt{\Delta_\rho})\psi=0$ then so does $\overline{\psi(-t,x)}$, and the system is said to be invariant by time reversal symmetry. The factor $\rho(\gamma)$ corresponds to the phase factor gained by a wave function when doing a loop around $\gamma$, while $\rho(\gamma)^{-1}$ corresponds to the phase factor gained while traveling the loop backwards. To sum up, the system breaks time reversal symmetry if and only if there is some closed geodesic $\gamma$ such that $\rho(\gamma)\neq \rho(\gamma^{-1})$.
	
	\subsubsection{Exhibiting GSE statistics.} The following discussion is inspired by Bolte--Keppeler \cite{Bolte1999}. We consider the case $\rho:\Gamma\to \mathrm{SU}(2)$. In this case, the \emph{time reversal operator} $\mathbf T$ is defined on the space $C^\infty(\widetilde X,\mathbf C^2)$  by 
	\[\mathbf T \psi(x)=\begin{pmatrix}0 & -\mathrm i \\ \mathrm i & 0\end{pmatrix}\overline{\psi(x)}.\]
	Since $\rho(\gamma)\in \mathrm{SU}(2)$, it is not hard to see that $\mathbf T$ preserves the space of $\rho$-equivariant functions $\widetilde X\to \mathbf C^2$. Thus, $\mathbf T$ descends to an operator on smooth sections of $\Gamma\backslash \widetilde X\times \mathbf C^2$. The operator $\mathbf T$ is antiunitary, satisfies $\mathbf T^2=-\mathbf I$, and commutes with the twisted Laplacian $\Delta_\rho$. Consequently, if $s$ is an eigenstate of $\Delta_\rho$, so is $\mathbf T s$, with same eigenvalue, and since $\mathbf T$ is antiunitary and satisfies $\mathbf T^2=-\mathbf I$, we can show that $s$ and $\mathbf Ts$ are orthogonal. Consequently, eigenstates of $\Delta_\rho$ always come in doublets $(s,\mathbf T s)$, a phenomenon called \emph{Kramers degeneracy}. Physicists argued that in this case, the pair $(s,\mathbf Ts)$ should count for multiplicity one. Thus, we are led to considering the counting function $\frac 12 N(\lambda, L,\rho)$ instead. 
	Provided that $\rho$ checks the denseness assumption, according to Theorem \ref{mainthm2} for $\mathbf G=\mathrm{SU}(2)$, the energy averages of the ensemble variance of $\frac 12 N_n(\lambda,L,\rho)$ converge, in our \emph{ad hoc} limit, to
	\[\Sigma_{\mathrm{GSE}}^2(\psi):=\frac 14\Sigma_{\mathrm{GOE}}^2(\psi).\]
	Here $\Sigma_{\mathrm{GSE}}^2$ is the number variance for eigenvalues of matrices of the Gaussian Symplectic Ensemble. We stress out that the result may fail when $\rho$ doesn't satisfy the denseness assumption. Indeed, in the case where $\rho$ is the trivial $2$-dimensional representation of $\Gamma$, we just have $N_n(\lambda,L,\rho)=2N_n(\lambda,L)$ thus after halving the counting function for $\Delta_{n,\rho}$ we simply recover the counting function of the scalar Laplacian $\Delta_n$, whose variance converges to $\Sigma_{\rm GOE}^2$.
	
	\subsection{Weak magnetic flux --- Transition from GOE to GUE} \label{subsec:transition}
	
	Since \eqref{eq : trace intro} depends continuously on the character $\rho$, we should be able to observe an intermediate behavior between GUE and GOE, by letting the representation $\rho$ depend on $\lambda$. We adopt again the point of view of vector potentials, and consider a family of flat connections $(\mathrm d+\mathrm i\alpha \mathbf A)$, where $\mathbf A\in C^\infty(X,T^*X)$ breaks time reversal symmetry while satisfying $\mathrm d\mathbf A=0$, and $\alpha$ is a parameter in $[0,1]$. This defines a continuous family of representations by letting
	\[\rho_\alpha(\gamma)=\exp\left(-\mathrm i\alpha\int_\gamma \mathbf A\right),\]
	where $\gamma$ is understood as a closed curve on $X$. The cases $\alpha=0$ and $\alpha=1$ correspond to GOE and GUE statistics, respectively. An analogous symmetry breaking was studied in the physics literature by Bohigas--Giannoni--Ozorio de Almeida--Schmit \cite{OBohigas1995} for chaotic dynamical systems, by the means of periodic orbit theory; and for random matrices by Pandey--Mehta \cite{cmp/1103922128}. Before stating our result, let us recall the following definition:
	
	\begin{defi}[Dynamical variance of an observable] Let $M$ denote the unit cotangent bundle of $X$. Consider a smooth function $a\in C^\infty(M)$ with zero mean (with respect to the normalized Liouville measure $\nu_0$). The \emph{variance} of $a$ is defined as the limit of the damped correlations
		\[\mathrm{Var}_{\nu_0}(a):=\underset{\lambda\to 0}\lim \int_{\mathbf R_+} \mathrm{e}^{-\lambda t} \langle a\circ \phi^t, a\rangle_{L^2(M,\nu_0)}\mathrm dt.\]
		where $\phi^t$ denotes the geodesic flow on $M$. 
	\end{defi}
	
	We let $\alpha$ vary in order to observe transition between the two behaviors (GOE $\to$ GUE). The following result is new even for $X$ of constant negative curvature.
	
	\begin{thm}[Smooth transition from GOE to GUE]  \label{thm:smooth transition} Take $\alpha=\alpha(\lambda )$ such that  $\sqrt{L}\alpha$ converges to some value $s\in [0,+\infty)$ when $\lambda$ goes to $+\infty$. Then, 
		\[\underset{\lambda\to +\infty}\lim \Big(\frac{1}{\delta}\int_{\lambda}^{\lambda+\delta}\Sigma^2(\mu,L,\rho_\alpha)\mathrm d\mu\Big) =\Sigma^2(s),\]
		where we have set
		\[\Sigma^2(s)=2\int_{\mathbf R_+} (1+\mathrm{e}^{-2\mathrm{Var}_{\nu_{\scalebox{.7}{$\scriptscriptstyle 0$}}}(a) s^2 t})t\widehat \psi^2(t)\mathrm dt.\]
		The function $\Sigma^2(s)$ is decreasing, with $\Sigma^2(0)=\Sigma_{\rm GOE}^2$ and $\lim_{s\to +\infty} \Sigma^2(s)=\Sigma^2_{\rm GUE}$ if the variance is positive. Here, $\nu_0$ is the normalized Liouville measure on $S^*X$, and $\mathrm{Var}_{\nu_0}(a)$ is the \emph{variance} of the function $a\in C^\infty(S^*X)$ defined by $a(x,\xi)=\langle \mathbf A(x),\xi\rangle_{T^*X}$. Here, $\langle \cdot,\cdot\rangle_{T^*X}$ denotes the scalar product on $T^*X$ induced by the metric.
	\end{thm}
	
	\begin{rem}When $a(x,\xi)=\langle \mathbf A(x),\xi\rangle_{T^*X}$, the variance $\mathrm{Var}_{\nu_0}(a)$ is positive if and only if there exists a closed geodesic $\gamma$ satisfying $\int_\gamma \mathbf A\neq 0$, see \S \ref{subsec:variance of a closed orbits}.\end{rem}
	
	\begin{rem}The variable $\sqrt{L}\alpha$ measures the average amplitude of the magnetic fluxes through surfaces enclosed by geodesics of lengths of order $L$. The counting function $N_n(\lambda,L)$ involves closed geodesics of lengths up to $L$. Thus, if $\sqrt{L}\alpha$ goes to $0$, the perturbation is too weak to affect the limit of the ensemble variance. 
		
		We give a heuristic explanation of the claim above, and will make a quantitative statement in Proposition \ref{prop : appendice variance et flux}. Take a presentation of the fundamental group 
		\[\Gamma=\big\langle a_1,b_1,\ldots,a_g,b_g\ | \  [a_i,b_i]=e\big\rangle .\]
		Let us forget the relations $[a_i,b_i]=e$ for a moment. Fix a magnetic flux $\Phi$ through each surface enclosed by one of the generators. If $\gamma$ is a closed geodesic, the magnetic flux through $\gamma^{-1}$ is the opposite of the magnetic flux through $\gamma$. Take at random a geodesic $\gamma$ of word length $L$. Roughly speaking, we have to take $L$ letters in $\{a_1,a_1^{-1},b_1,b_{1}^{-1},\ldots, a_g,a_g^{-1},b_g,b_{g}^{-1}\}$, each of which give a contribution $\pm \Phi$ to the total flux through the surface enclosed by $\gamma$. Thus, the magnetic flux through a random geodesic of word length $L$ behaves like the sum of $L$ independent Bernoulli variables. On average, the amplitude of the total flux $\Phi_L$ through $\gamma$ is expected to be of order $\sqrt{L}\Phi$. Even better, if we take a geodesic $\gamma_L$ of word length $L$ at random, the normalized flux $\frac{\Phi_L}{\sqrt L}$ through $\gamma_L$ should converge in distribution to a Gaussian random variable (we will give a rigorous statement in Proposition \ref{prop: central limit theorem for closed orbits}). Of course $\Gamma$ is not a free group, and we have to take in account the fact that $[a_i,b_i]=e$, thus the combinatorics have to be more intricate; nevertheless we recover the order of the amplitude of the flux.\end{rem}

	\subsection{Notations} \begin{enumerate}[---]
		
		\item The Fourier transform of a function $f$ will be taken as
		\[\widehat f(\xi)=\frac{1}{2\pi}\int_{\mathbf R} \mathrm{e}^{-\mathrm ix\xi} f(x)\mathrm dx,\]
		so that $f(x)=\int_{\mathbf R}\mathrm{e}^{\mathrm{i} x\xi} \widehat f(\xi)\mathrm d\xi$.
		\item We denote by $\mathbf 1_Y$ the indicator function of a set $Y$. 
		\item We write $f=\mathcal O(g)$ if $|f|\le C|g|$ for some constant $C$. Alternatively, we may write $f\lesssim g$ if $f$ and $g$ are both positive. If $C$ depends on some relevant parameters, we denote it with a subscript, \emph{e.g.} $f=\mathcal O_A(g)$. If $f$ and $g$ are two functions of $\lambda$, we write $f\ll g$ if $f(\lambda)/g(\lambda)$ goes to $0$ as $\lambda\to +\infty$.
		
		\item If $M$ is a manifold, we denote by $\mathcal D'(M)$ the set of distributions on $M$, and $\mathcal E'(M)$ the set of compactly supported distributions. Since we work on a Riemannian manifold, a function $f$ canonically identifies with the $1$-density $f\mathrm dx$, where $\mathrm dx$ denotes the Riemannian measure. Thus, we shall implicitly take Schwartz kernels of linear operators with respect to this measure. If $\mathcal V$ is a vector bundle over $M$, we denote $\mathcal D'(M,\mathcal V)$ the set of distributions with values in $\mathcal V$. Distributional pairings will be denoted with brackets $\langle \bullet, \bullet\rangle$.
		
		\item If $\xi\in \mathbf C$, we denote $\langle \xi\rangle=\sqrt{1+|\xi|^2}$.
		
		\item 	If $Y$ is a random variable with $\mathbf E(|Y|)<\infty$, we denote by $\widetilde Y$ the centered variable $Y-\mathbf E(Y)$.
	\end{enumerate}

	\subsection{Organization of the paper} 
	
	\begin{itemize}
		\item In \S \ref{sec : twisted laplacian}, we review basic facts about twisted Laplacians and discuss how standard properties of the wave trace in the scalar case translate to this setting.
		\item In \S \ref{sec : Hadamard parametrix} we briefly expose the construction of the Hadamard parametrix for the wave propagator $\cos(t\sqrt \Delta)$, that we use in the proof of the Gutzwiller trace formula. We do not provide much detail, since great expositions can be found in the literature \cite{Berger1971LeSD,Bérard1977,Sogge2014}. 
		\item In \S \ref{sec : trace formula} we prove the trace formula of Proposition \ref{trace formula for long times}, by relying on estimates of the previous section. We refer to \cite{JakobsonWeyl2007} for a similar result without twists. The strategy of the proof is standard, we point out that the computations for test functions supported around a single closed geodesic length were already performed by Donnelly \cite{Donnelly1978}. We also refer to \cite{Schenk} for a long time trace formula in a noncompact setting.
		\item In \S \ref{sec:équidistributions} we study twisted weighted counting functions of closed geodesics, more precisely the quantity
		\[\frac 1L\sum_{\gamma\in\mathcal G}\operatorname{Tr}(\rho(\gamma))\frac{\ell(\gamma)^\sharp \varphi(\ell(\gamma)/L)}{|I-\mathcal P_\gamma|},\]
		where $\rho$ is a representation of $\Gamma$ and $\varphi$ a smooth function with compact support. We obtain asymptotics for large $L$ by relying on a local trace formula for Anosov flows obtained by Jin--Zworski and Jin--Tao, see the references mentioned in \S \ref{sec:équidistributions}.
		\item  In \S \ref{sec : proof theorems} we review statistical properties of the model of random covers before turning to the proofs of Theorems \ref{mainthm} and \ref{mainthm2}. The strategy is similar in spirit to that of Naud \cite{naud2022random} (although demanding new ideas because we are dealing with surfaces of nonconstant curvature, and in a high energy regime) and has been described in \S \ref{subsec:strategie}. 
		
		\item In \S \ref{sec:prooftransition}, we investigate the smooth transition from GOE to GUE by introducing a varying magnetic potential on the surface. It boils down to providing uniform estimates on the resonances of a perturbed Ruelle operator, whose proofs are postponed to Appendix \ref{sec:appendix}. As a byproduct, we recover a central limit theorem for long periodic orbits, picked at random with weights proportional to $\frac{1}{|I-\mathcal P_\gamma|}$.

		\item Finally, \S\ref{sec:CLT} and \S\ref{sec:asGOE} are devoted to the proofs of the Central Limit Theorem \ref{thm : CLT} and Theorem \ref{thm:asGOEfluctuations} on almost sure GOE fluctuations. We rely on probabilistic results of Maoz \cite{Maoz1} on the model of random covers together with results of \S \ref{sec:équidistributions}.
	\end{itemize}

	\subsection*{Acknowledgements} I would like to thank St{\'e}phane Nonnenmacher for our numerous discussions and his guidance during the writing of this paper. I would also like to thank S{\'e}bastien Gou{\"e}zel, Tristan Humbert, Thibault Lefeuvre and Fr{\'e}d{\'e}ric Naud for discussions related to this work. 
	\section{Unitary representations of $\Gamma$, twisted Laplacians and the wave trace}
	
	\label{sec : twisted laplacian}
	\subsection{Twisted Laplacians} We introduce twisted Laplacians associated to unitary representations of the fundamental group $\Gamma$ of $X$. As we explained in the introduction, in the Abelian case, twisted Laplacians model the Ahronov--Bohm effect.

	Let $X$ be a connected, closed, negatively curved Riemannian surface. Let $\widetilde X$ denote its universal cover, and $\Gamma$ the group of deck transformations of the cover $\widetilde X\to X$, that is, the group of isometries $\gamma$ of $\widetilde X$ satisfying $\pi\circ \gamma=\pi$, where $\pi:\widetilde X\to X$ is the projection --- note that $\Gamma$ identifies with the fundamental group of $X$. The surface $X$ can then be identified with the quotient $\Gamma\backslash \widetilde X$. 
	
	Let $\rho:\Gamma\to \mathrm{GL}(V_\rho)$ be a finite dimensional unitary representation of $\Gamma$, where $(V_\rho,\langle \cdot,\cdot\rangle_\rho)$ is a Hermitian vector space. We define a Hermitian vector bundle $\mathcal V_\rho$ over $X$ by setting $\mathcal V_\rho= \Gamma \backslash (\widetilde X\times V_\rho)$, that is the quotient of $\widetilde X\times V_\rho$ by the relation
	\[(\gamma \tilde x, v)\sim (\tilde x,\rho(\gamma)^{-1}v).\]
	This bundle comes with a flat connection $\nabla:C^\infty(M,\mathcal V_\rho)\to C^\infty(M,T^*M\otimes \mathcal V_\rho)$, that acts in the following way: if $f$ is a smooth section of $\mathcal V_\rho$, that can be locally expressed as $f=\sum f_j e_j$ where $f_j\in C^{\infty}(X)$ are smooth functions and $(e_j)$ is a basis of the vector space $V_\rho$, we set $\nabla f:=\sum \mathrm df_j\otimes e_j$. Smooth sections of $\mathcal V_\rho$ are identified with smooth $\rho$-equivariant functions $\widetilde X\to V_\rho$, that is functions $f$ satisfying the \emph{equivariance property}
	\[f(\gamma \tilde x)=\rho(\gamma )f(\tilde x), \qquad \gamma\in \Gamma, \qquad \tilde x\in \widetilde X.\]
	Since $\rho$ is unitary, if $f:\widetilde X\to V_\rho$ is equivariant, one has $|f(\gamma x)|_\rho=|f(x)|_\rho$, hence if $s\in C^{\infty}(X,\mathcal V_\rho)$ we can define unambiguously the norm $|s(x)|$. In turn, we endow $C^{\infty}(X,\mathcal V_\rho)$ with the norm
	\[\|s\|^2_{L^2(X,\mathcal V_\rho)}=\int_X |s(x)|_\rho^2\mathrm dx.\]	
	The space $L^2(X,\mathcal V_\rho)$ is defined as the completion of $C^{\infty}(X,\mathcal V_\rho)$ with respect to this norm. 
	
	The twisted Laplacian $\Delta_\rho$ is defined as $\nabla ^*\nabla:C^\infty(X,\mathcal V_\rho)\to C^\infty(X,\mathcal V_\rho)$. In local charts, we have
	\begin{equation}\label{eq : laplacien twisté formule locale} \Delta_\rho u=\sum_j (\Delta_X f_j) e_j.\end{equation}
	Working with equivariant functions on the universal cover instead of sections of $\mathcal V_\rho$ happens to be very convenient, since formula \eqref{eq : laplacien twisté formule locale} holds globally, i.e. there is no need to take local charts: the lift of $\Delta_\rho$ to the space of $\rho$-equivariant function is just given by $\Delta_{\widetilde X}\otimes \mathrm{Id}_{V_\rho}$. The Laplacian $\Delta_\rho$ essentially self-adjoint, that is,  it admits a unique self-adjoint extension to $L^2(X,\mathcal V_\rho)$, still denoted $\Delta_\rho$. By ellipticity, its spectrum consists of a nondecreasing sequence of eigenvalues 
	\[0\le \lambda_0^2\le \lambda_1^2\le \ldots \le \lambda_n^2\to +\infty.\]
	Note that $\lambda_0=0$ if and only if $\rho$ contains the trivial representation.

	\subsection{The wave trace} 
	
	Let $U^\rho(t)=\mathrm{e}^{-\mathrm it\sqrt{\Delta_\rho}}$ be the unitary group associated to the evolution equation 
	\[\big(-\mathrm{i} \partial_t +\sqrt{\Delta_\rho}\big)u(t,x)=0.\]
	We denote by $U^\rho(t,x,y)$ the Schwartz kernel of $U^\rho$. It is a distribution on $\mathbf R\times X\times X$ taking values in $\mathrm{Hom}(\mathcal V_\rho,\mathcal V_\rho)$, that is the bundle over $X\times X$ with fiber $\mathrm{Hom}(V_\rho,V_\rho)$. Denote by $D:\mathbf R\times X\to \mathbf R\times X\times X$ the diagonal embedding. By wavefront set properties --- see \cite[Theorem 8.2.4]{HormanderI} and also \cite{Duistermaat1975}, we can define $U^\rho(t,x,x):=D^*U^\rho(t,x)$, which is a distribution on $\mathbf R\times X$ taking values in $\mathrm{End}(\mathcal V_\rho)$, that is the bundle over $X$ with fiber $\mathrm{End}(V_\rho)$. Then, we let
	\[\operatorname{Tr^{End}} U^\rho(t,x)\in \mathcal D'(\mathbf{R}\times X),\]
	which is defined by taking fiberwise the trace of the endomorphism $U^\rho(t,x,x):V_\rho\to V_\rho$. Eventually, the trace of $U^\rho(t)$ is defined by
	\[\langle \operatorname{Tr} U^\rho(t),\varphi \rangle_{\mathcal D'(\mathbf{R}),\mathcal D(\mathbf{R})}=\langle \operatorname{Tr^{End}} U^\rho(t,x), \varphi\otimes \mathbf 1_X\rangle_{\mathcal D'(\mathbf{R}\times X),\mathcal D(\mathbf{R}\times X)}.\]

	In the following we shall rather work with the kernel of $\cos(t\sqrt{\Delta_\rho})$, \emph{i.e.} the real part of $U^\rho(t,x,y)$, that we denote $E^\rho(t,x,y)$. We drop the $\rho$ superscript in the scalar case. The wave trace if related to the eigenvalues of the twisted Laplacian by the following proposition (see \emph{e.g.} \cite[Chapter 12]{Grigis_Sjöstrand_1994}).
	
	\begin{prop}Let $\rho$ be a finite dimensional unitary representation of $\Gamma$. Let $(\varphi_j)_{j\ge 1}$ be an orthogonal basis of normalized eigensections of $\Delta_\rho$, associated to the sequence of eigenvalues $0\le \lambda_1^2\le \lambda_2^2\le \ldots$. Then, we have the following equality of distributions on $\mathbf{R}\times X$:
		\[\operatorname{Tr}^{\rm End} E^\rho(t,x,x)=\sum_{j\ge 0} |\varphi_j(x)|^2 \cos(\lambda_j t).\]
		Moreover, the following equality of distributions on $\mathbf R$ holds:
		\[\operatorname{Tr} E^\rho(t)=\sum_{j\ge 0} \cos(\lambda_jt).\]\end{prop}
	
	The purpose of the next section is to provide explicit approximations of the wave kernel $E^\rho$.
	
	\section{The Hadamard Parametrix for the wave equation}
	\label{sec : Hadamard parametrix}
	In this section, we review the construction of the Hadamard parametrix for the wave propagator. It is modeled on the wave propagator on the Euclidean space, and provides an expansion of the kernel $E(t,x,y)$ in powers of $(d(x,y)^2-t^2)_-^{k-\frac 32}$, $k\ge 0$. Here, $(d(x,y)^2-t^2)_-^\alpha$ denotes the pullback of the distribution $u_-^{\alpha}$ (see \cite[\S 3.2]{GelfandShilov} for a definition) by the submersion $(t,x,y)\mapsto d(x,y)^2-t^2$, well-defined for $t\neq 0$. This parametrix will be used in the proof of the trace formula \eqref{eq : trace intro}. Many accounts of its construction can be found in the literature, see for example \cite{Berger1971LeSD,Bérard1977,Sogge2014}; we adopt notations of B\'erard \cite{Bérard1977}.

	For now, let $X$ be a smooth Riemannian surface (typically a closed surface of negative curvature or its universal cover, for our purposes), with positive injectivity radius denoted by $\rho_{\mathrm{inj}}>0$. The parametrix of order $N$ for $\cos(t\sqrt{\Delta})$ is the distribution defined on $\{d(x,y)<\rho_{\mathrm{inj}}, t\neq 0\}$, by
	\begin{equation} \label{eq:defparametrix}E_N(t,x,y)=\frac{1}{\sqrt{\pi}}\sum_{k=0}^{N} (-1)^k u_k(x,y)|t| \frac{(d(x,y)^2-t^2)_-^{k-\frac 32}}{4^k\Gamma(k-\frac 12)}.\end{equation}
	Here, $(u_k)_{k\ge 0}$ is a sequence of smooth functions defined inductively by a sequence of transport equations, on the set $\{d(x,y)<\rho_{\mathrm{inj}}\}$. Their construction can be found in the references aforementioned. In particular, $u_0(x,y)=\Theta(x,y)^{-\frac 12}$ where $\Theta(x,y)$ is the volume element\footnote{ Consider the map $\exp_x:T_x X\to X$ defined in a neighborhood of $0\in T_xX$. If $y=\exp_x(v)$, then $\mathrm d\exp_x(v)$ is a linear map $T_xX\to T_yX$, and we let $\Theta(x,y)=|\det \mathrm d\exp_x(v)|$, where the determinant is taken with respect to two orthonormal bases on $T_xX$ and $T_yX$.} at $y$ in normal coordinates centered at $x$. Note that on the diagonal  $u_0(x,x)=1$. 
	
	We denote by $R_N$ the difference between the exact kernel $E(t,x,y)=\cos(t\sqrt{\Delta})(x,y)$ and $E_N(t,x,y)$, defined on $\{d(x,y)<\rho_{\rm inj},t\neq 0\}$. It is clear from the definition that $E_N$ is supported in $\{d(x,y)\le t\}$, and the same property holds for the exact kernel $E$, by the finite propagation speed of the wave equation. Thus, by taking the difference, $R_N$ is also supported in $\{d(x,y)\le t\}$. 
	
	Since $E_N$ vanishes outside $\{d(x,y)>t\}$, if we take some $\varepsilon>0$ smaller than $\rho_{\rm inj}$, then $E_N$ extends to $\big([-\varepsilon,\varepsilon]\backslash \{0\}\big)\times X\times X$. One can show if $N$ is large enough, the distribution $R_N=E-E_N$ is actually a continuous function for small times, that extends continuously to $t=0$. More precisely, we have:
	\begin{prop}[{\cite[Theorem 3.1.5]{Sogge2014}}] Assume $X$ is closed. For any integer $N\ge 5$, there is some $\varepsilon>0$ such that $R_N\in C^{N-5}([-\varepsilon,\varepsilon]\times X\times X)$. Moreover, we have the following bounds on derivatives for small $t$:
		\[\partial_{t,x,y}^\alpha R_N(t,x,y)=\mathcal O_N(|t|^{2N-|\alpha|}), \qquad |\alpha|\le N-4.\]
	\end{prop}

	\subsection{Study of $E(t,x,x)$ for short times} In this subsection, we restrict the study to the case where $X$ is closed. The Hadamard parametrix is defined for nonzero times; however, we aim to show a trace formula that holds for test functions that do not vanish at $0$. We explain how to extend $E_N$ to a distribution well-defined around $t=0$, while keeping the identity $E=E_N+R_N$ true at $t=0$.
	
	The following argument is given in a rather sketchy way in \cite[(30)]{Bérard1977} so we provide some details. We start from the equality of distributions on $\mathbf R^*\times X$:
	\begin{equation}\label{eqenzero}E(t,x,x)=\frac{1}{\sqrt{\pi}}\sum_{k=0}^{N}\frac{(-1)^k}{4^k\Gamma(k-\frac 12)} \   t^{2k-2}\otimes u_k(x,x) +R_N(t,x,x).\end{equation}
	The left-hand side is well-defined at $t=0$, but the function $t\mapsto t^{-2}$ appearing in the right-hand side is not well-defined at $t=0$. However, we can replace it by a distribution $t_{\rm reg}^{-2}$ defined on the whole real line\footnote{This distribution is defined by $\langle t_{\rm reg}^{-2},\varphi\rangle = \int_{\mathbf R_+} \frac{\varphi(t)+\varphi(-t)-2\varphi(0)}{t^2}$.}, that coincides with $t^{-2}$ on $\mathbf R^*$, without altering the previous inequality for $t\neq 0$. Now that both sides of \eqref{eqenzero} are well-defined at $t=0$ --- recall that $R_N(t,x,x)$ extends continuously to $t=0$ ---, we can write 
	\begin{equation}\label{eq : montrer que les fj sont zero} E(t,x,x)=\frac{1}{\sqrt{\pi}}\sum_{k=0}^{N}\frac{(-1)^k}{4^k\Gamma(k-\frac 12)} \   t^{2k-2}_{\mathrm{reg}}\otimes u_k(x,x)+\sum_{j\le N_0} \delta^{(j)}(t) f_j(x)  +R_N(t,x,x),\end{equation}
	where $t^{2k-2}_{\mathrm{reg}}$ is the regularization of $t^{2k-2}$ in the sense of Gelfand--Shilov \cite[\S 3.2]{GelfandShilov}, and the $f_j$'s are distributions to be determined. This comes from the fact that a distribution with support in $\{0\}\times X \subset \mathbf R\times X$ writes as a combination of derivatives of Dirac masses at $t=0$, weighted by distributions on $X$. Since $E(t,x,x)$ is even in the time variable, we immediately get $f_j=0$ when $j$ is odd. To treat the case of even values of $j$, it is enough to know the singularity of $E(t,x,x)$ at $t=0$. It was computed by Duistermaat--Guillemin \cite[\S 2]{Duistermaat1975}, who gave asymptotics for
	\[\int_{\mathbf R} E(t,x,x)\cos(\lambda t)\widehat \psi(t) \mathrm dt\]
	when $\widehat \psi$ is localized around $0$.
	
	\begin{prop}[{\cite[Proposition 2.1]{Duistermaat1975}}] \label{prop : Duistermaat trace en zero} Assume $X$ is closed. Let $\psi$ be a smooth real function with Fourier transform $\widehat \psi$ compactly supported in a sufficiently small neighborhood of $0$, satisfying moreover $\widehat \psi\equiv 1$ in another neighborhood of $0$. Then, uniformly in $x$, as $\lambda\to +\infty$, we have
		\begin{equation}\label{eq:premiere facon de calculer}\int_{\mathbf R}\operatorname{Tr} E(t,x,x) \cos(\lambda t)\widehat \psi(t) \mathrm dt=\frac \lambda 2 +\mathcal O(\lambda^{-\infty}). \end{equation}
	\end{prop}
	
	Using \eqref{eq : montrer que les fj sont zero} one may also just compute 
	\begin{equation}\label{eq: seconde facon}\int_{\mathbf R}\operatorname{Tr} E(t,x,x) \cos(\lambda t)\widehat \psi(t) \mathrm dt =\frac \lambda 2+(f_0(x)+\lambda^2 f_2(x)+\ldots +\lambda^{N_0}f_{N_0}(x)) +\mathcal O(\lambda^{-N}).\end{equation}
	The first term comes from $t_{\mathrm{reg}}^{-2}$, using that $u_0(x,x)=1$; the distributions $t_{\mathrm{reg}}^{2k-2}$ are smooth functions for $k\ge 1$ and yield a contribution $\mathcal O(\lambda^{-\infty})$, whereas $R_N$ gives a contribution $\mathcal O(\lambda^{-N})$ --- we perform as many integrations by parts as allowed by the regularity of $R_N(t,x,x)$ as a function of $t$, for small 
	$t$.
	By comparing asymptotics \eqref{eq:premiere facon de calculer} and \eqref{eq: seconde facon} we get that all $f_j$'s are zero. Therefore, we have shown
	\begin{prop}\label{prop : kernel at 0} We extend $E_N(t,x,x)$ to $t=0$ by the formula 
		\begin{equation}\label{eq : def f_N} E_N(t,x,x)=\frac{1}{\sqrt{\pi}}\sum_{k=0}^{N}\frac{(-1)^k}{4^k\Gamma(k-\frac 12)} \   t^{2k-2}_{\mathrm{reg}}\otimes u_k(x,x).\end{equation}
		Then, as distributions on $\mathbf R\times X$, we have $E(t,x,x)= E_N(t,x,x)+R_N(t,x,x)$. \end{prop}
	
	\begin{rem} Since we aim to study the trace $\int_X E(t,x,x)\mathrm dx$ on long times scales $1\ll L\le c_0 \log(\lambda)$, we need to bound $R_N(t,x,y)$ up to long times. Unfortunately, we only have an appropriate control on $R_N$ for small times. Indeed, bounds on $R_N$ are obtained by energy estimates for the function $R_N(t,x,\cdot)$ where $t$ and $x$ are fixed, combined with a Sobolev embedding theorem. But when $|t|>\rho_{\mathrm{inj}}$, the function $R_N(t,x,\cdot)$ is not defined on the whole manifold anymore, which prevents applying the aforementioned method for longer times. This issue can be solved by lifting the wave equation to the universal cover, which allows getting control on the remainder for any times. \end{rem}
	
	\subsection{Lift to the universal cover} From now on, let $X$ be a closed, negatively curved surface, and denote by $\widetilde X$ its universal cover. We consider wave kernels associated to twisted Laplacians instead of focusing only on the scalar case. Let $\rho$ be a finite dimensional unitary representation of $\Gamma$. Then, the kernel of $E^\rho(t,\cdot,\cdot)=\cos(t\sqrt{\Delta_\rho})$ can be expressed in terms of the kernel of the wave operator on $\widetilde X$, by the formula
	\begin{equation}\label{eq:liftcover}E^\rho(t,x,y)=\sum_{\gamma\in \Gamma} \widetilde E(t,\tilde x,\gamma \tilde y)\rho(\gamma),\end{equation}
	where $\widetilde E$ is the wave kernel on $\widetilde X$ and $\tilde x,\tilde y$ are any lifts of $x,y$. Here, we should understand the left-hand side as the pullback on $\widetilde X\times \widetilde X$ of the distribution $E^\rho(t,x,y)$ by the submersion $\pi\times \pi: \widetilde X\times \widetilde X\to X\times X$. If we restrict $\tilde x$ and $\tilde y$ to live in fixed compact subsets of $\widetilde X$, then the sum over $\gamma\in \Gamma$ in \eqref{eq:liftcover} has a finite number of nonzero terms. Indeed, we can show by a volume argument the existence of a constant $C>0$ such that for any $\tilde x, \tilde y\in \widetilde X$,
	\[\# \{\gamma\in \Gamma \ | \ d(\tilde x,\gamma \tilde y)\le L\}=\mathcal O(\mathrm{e}^{CL}),\]
	and the bound is uniform when $\tilde x$ and $\tilde y$ are in compact sets, see \emph{e.g.} \cite{ColindeVerdière1973} for a proof.
	
	Now, let $\widetilde E_N$ be the Hadamard parametrix for $\widetilde E$, defined as in \eqref{eq:defparametrix}. Since the injectivity radius of $\widetilde X$ is infinite, $\widetilde E_N$ is defined on $\mathbf R^*\times \widetilde X\times \widetilde X$. The construction of the functions $\tilde u_k$ is purely geometrical, hence $\tilde u_k(\gamma \tilde x,\gamma \tilde y)=\tilde u_k(\tilde x,\tilde y)$ for all $\tilde x,\tilde y\in \widetilde X$ and $\gamma\in \Gamma$. We also point out that $\tilde u_k(\tilde x,\tilde x)=u_k(x,x)$ for any lift $\tilde x$ of $x$. 
	
	Let $\widetilde R_N=\widetilde E-\widetilde E_N$ denote the difference between the exact kernel and the parametrix on $\widetilde X$. Lifting the problem to the universal cover allows obtaining control on the remainder for any time: when $N$ is large enough, $\widetilde R_N$ extends continuously to $\mathbf R\times \widetilde X\times \widetilde X$ and one can show: %
	\begin{prop}[{\cite[eq. (4.24)]{blair2023strichartz}}]\label{prop:borneresteSogge} Take an integer $N_0\in \mathbf N$. Then, for $N$ large enough, $\widetilde R_N \in C^{N_0}(\mathbf R\times \widetilde X\times \widetilde X)$ and there is some constant $C>0$ such that
		\[|\partial_{t}^j \widetilde R_N(t,\tilde x,\tilde y)|\le C\mathrm{e}^{C |t|}, \qquad j=0,\ldots,N_0. \]
		with uniform estimates with respect to $j\le N_0$ and $t\in \mathbf R$, $\tilde x,\tilde y\in \widetilde X$.\end{prop}
	A detailed proof can be found in \cite[Chapter 3]{Sogge2014}. This result allows getting a small remainder when testing $\widetilde R_N$ against an oscillating function with large frequency compared to the timescale, by a simple integration by parts.
	\begin{lem}\label{majoration reste intégral} Let $\psi$ be as in Theorem \ref{trace formula for long times}. There is some constant $C$ depending only on $X$ and $N$, such that for any $\tilde x,\tilde y\in \widetilde X$ and $L\ge 1$,
		\[\int_{\mathbf R}  \mathrm{e}^{\mathrm i\lambda t}\widehat \psi\left(\frac tL\right)\widetilde R_N(t,x,y)\ \mathrm dt=\mathcal O_{N,\psi}(\lambda^{-1}\mathrm{e}^{CL}),\]
		where the error is uniform in $\tilde x,\tilde y$.
		
	\end{lem}

	For short times ($|t|<\ell_0$ where $\ell_0$ is the length of the shortest closed geodesic), it follows from \eqref{eq:liftcover} that $\widetilde E(t,\tilde x,\tilde x)=E(t,x,x)$. Also, since $\tilde u_k(\tilde x,\tilde x)=u_k(x,x)$, for any time $t$ we have $\widetilde E_N(t,\tilde x,\tilde x)=E_N(t,x,x)$. Therefore, by Proposition \ref{prop : kernel at 0}, on the diagonal of $\widetilde X$, the kernel is given by \cite[Proposition 38]{Bérard1977}:
	\begin{equation}\label{prop : extension trace end}\widetilde E(t,\tilde x,\tilde x)=E_N(t,x,x)+\widetilde R_N(t,\tilde x,\tilde x),\end{equation}
	where $E_N(t,x,x)$ has been defined in \eqref{eq : def f_N}.

	When performing computations, we will need upper bounds on the derivatives of the functions $\tilde u_k(\tilde x,\tilde y)$:
	\begin{prop}[see \emph{e.g.} {\cite[Proposition B.1.]{keeler2022logarithmic}}]\label{borne un} Let $P$ and $Q$ be some lifts to $C^\infty(\widetilde X)$ of smooth differential operators acting on $C^\infty(X)$. Then, for some constant $C$, and for any $\tilde x,\tilde y\in \widetilde X$,
		\[P_xQ_y \tilde u_k(\tilde x,\tilde y)=\mathcal O(\mathrm{e}^{C d(\tilde x,\tilde y)}).\]
	\end{prop}
	Actually the result is true for more general differential operators on $C^\infty(\widetilde X)$, provided that they have bounded coefficients in $C^\infty$ norm.
	
	

	\section{The Gutzwiller trace formula for the twisted Laplacian} \label{sec : trace formula}
	This section is devoted to the proof of the trace formula of Proposition \ref{trace formula for long times}. The strategy of proof goes back to the work of Colin de Verdière on the Schrödinger equation \cite{ColindeVerdière1973}. Here we are rather interested in the wave kernel. The proof for the non-twisted Laplacian is fairly classical, although here one has to provide uniform estimates up to time $L$. 
	
	Jakobson--Polterovitch--Toth \cite{JakobsonWeyl2007} obtained a long time trace formula for the scalar wave kernel, by using the Hadamard parametrix and microlocal analysis, combined with a property of separation of periodic orbits in phase space for Anosov flows. Computations were already performed by Donnelly \cite{Donnelly1978} in the case where the support of $\widehat \psi$ is localized around a single length $\ell(\gamma)$ in the length spectrum. A similar result was obtained by Sunada \cite{SunadaHeat1982} for the heat equation. 
	
	The extension to the case of twisted Laplacians is rather straightforward. Contrarily to \cite{JakobsonWeyl2007} though, we do not explicitly use the separation of the orbits in phase space, nevertheless the ideas are very similar. We use the Hadamard parametrix for the wave equation and first compute asymptotics for integrals in the time variable, following the work of B\'erard \cite{Bérard1977}. A stationary phase expansion --- in the space variable --- is then performed, following ideas of Colin de Verdière.

	
	\subsection{Reduction to hyperbolic cylinders} 
	We briefly recall the correspondence between the set $\mathcal G$ of oriented closed geodesics of $X$ and the set of nontrivial conjugacy classes of $\Gamma$, and refer to \cite[Chapter 12]{doCarmo} for a more detailed exposition. If $\gamma$ is a nontrivial element of $\Gamma$, the \emph{displacement function} $\tilde x\mapsto d(\tilde x,\gamma \tilde x)$ is minimized on a unique geodesic $\tilde \gamma$ of $\widetilde X$, called the \emph{axis} of $\gamma$. This geodesic projects on a closed geodesic on $X$ that we will denote $\gamma$ by an abuse of notation. Note that if $\gamma'\in \Gamma$ is conjugated to $\gamma$ then the axes of $\gamma$ and $\gamma'$ project on the same closed geodesic on $X$. This procedure yields all closed geodesics of $X$. 
	
	A nontrivial element of $\Gamma$ is called \emph{primitive} if it cannot be written as a nontrivial power of another element of $\Gamma$. Primitive elements of $\Gamma$ yield primitive closed geodesics on $X$, that is, closed geodesics that cannot be obtained by repeating a shorter closed geodesic. Every nontrivial $\gamma\in \Gamma$ can be uniquely written $\gamma_*^n$ for some primitive element $\gamma_*$ and integer $n\ge 1$.
	
	For $\gamma\in \Gamma$, the centralizer of $\gamma$ is defined as the set of $g\in \Gamma$ such that $g\gamma =\gamma g$; it is given by $\Gamma_\gamma=\{\gamma_*^n, \ n\in \mathbf Z\}$, where $\gamma_*$ is the primitive element associated to $\gamma$. Note that $\Gamma_\gamma$ is a normal subgroup of $\Gamma$. We take $\mathcal P\subset \Gamma$ a subset of primitive elements different of the identity, such that every conjugacy class of $\Gamma$ is represented by a unique power (with positive exponent) of an element of $\mathcal P$. Then, any element of $\Gamma\backslash \{\mathrm{Id}\}$ writes uniquely as a product $g^{-1}\gamma^k g$ with $\gamma\in \mathcal P$, $k\ge 1$ and $g\in \Gamma_\gamma\backslash \Gamma$. 
	
	When $X$ is a hyperbolic surface, the quotient $\Gamma_\gamma\backslash \widetilde X$ is usually called a \emph{hyperbolic cylinder}. It contains two closed primitive geodesic that project on the geodesics of $X$ associated to $\gamma_*$ and $\gamma_*^{-1}$. We point out that the displacement function $\tilde x\mapsto d(\tilde x,\gamma \tilde x)$ is invariant under the action of $\Gamma_\gamma$, hence descends to a function on $\Gamma_\gamma\backslash \widetilde X$. The same remark holds for the wave kernel $\widetilde E(t,\tilde x,\gamma \tilde x)$.
	
	The goal of this section is to show the identity
	\begin{prop}\label{prop:decompositioncylindre} We have the following equality of distributions on $\mathbf R$,
		\begin{equation}\label{eq : regroupements intégrales}\int_X \operatorname{Tr^{End}} E^\rho(t,x,x)\mathrm dx=\dim(V_\rho)\int_X \widetilde E(t,x,x)\mathrm dx+ \sum_{\gamma\in \mathcal P}\sum_{k\ge 1} \chi_\rho(\gamma^k)\int_{\Gamma_\gamma\backslash \widetilde X} \widetilde E(t,\tilde x,\gamma^k \tilde x)\mathrm d\tilde x.\end{equation}
	\end{prop}
	
	Formula \eqref{eq : regroupements intégrales} can be derived formally, however one has to be careful when performing computations as we are dealing with distributions. The right way to proceed is to start by lifting integration on $X$ to integration on $\widetilde X$ by introducing a smooth cutoff.
	
	\begin{lem}[{\cite[Lemme 1]{ColindeVerdière1973}}] \label{intégration relevé 1} There is a smooth compactly supported function $\eta:\widetilde X\to \mathbf R$, taking values in $[0,1]$, such that for any continuous function $f:X\to \mathbf R$, if $\pi:\widetilde X\to X$ denotes the projection,
		\[\int_X f(x)\mathrm dx=\int_{\widetilde X} (f\circ \pi)(\tilde x) \eta(\tilde x)\mathrm d\tilde x.\]
		Also, $\sum_{\gamma\in \Gamma} \eta( \gamma x)=1$, and we can ask for $\mathrm{diam}\ \mathrm{supp}(\eta)\le 2\mathrm{diam}(X)+1$. Since $\pi$ is a submersion the identity extends to distributions (\cite[Theorem 6.1.2]{HormanderI}), and we deduce that if $T\in \mathcal D'(X)$ then,
		\begin{equation}\label{eq:firstextension}\langle \pi^* T,\eta\rangle_{\mathcal D'(\widetilde X),\mathcal D(\widetilde X)}=\langle T,\mathbf 1_X\rangle_{\mathcal D'(X),\mathcal D(X)}.\end{equation}\end{lem}

We also need the following analogous result, that reduces the computation of certain integrals on $\widetilde X$ to the computation of integrals on the cylinder $\Gamma_\gamma\backslash \widetilde X$. It can be proved by tweaking the proof of Lemma \ref{intégration relevé 1} found in \cite{ColindeVerdière1973}.

\begin{lem}\label{intégration cylindre} Let $\pi_\gamma:\widetilde X\to \Gamma_\gamma \backslash \widetilde X$ denote the projection.  Fix a set of representatives of elements of $\Gamma_\gamma\backslash \Gamma$ and define
	\[\eta_\gamma(\tilde x):=\sum_{[g]\in \Gamma_\gamma\backslash \Gamma}\eta(g\tilde x).\]
	The sum over $[g]\in \Gamma_\gamma\backslash \Gamma$ contains a finite number of nonzero terms. Let $f:\Gamma_\gamma \backslash \widetilde X\to \mathbf C$ be a smooth, \textbf{compactly supported} function. Then, $\eta_\gamma\cdot \pi_\gamma^*f\in C^\infty(\widetilde X)$ is compactly supported and
	\[\int_{\widetilde X} f\circ \pi(\tilde x)\cdot \eta_\gamma(\tilde x) \ \mathrm d\tilde x=\int_{\Gamma_\gamma \backslash \widetilde X} f(x) \mathrm dx,\]
	Also, if $T\in \mathcal E'(\Gamma_\gamma \backslash \widetilde X)$ is a compactly supported distribution, then $\eta_\gamma \cdot \pi_\gamma^* T$ is a compactly supported distribution on $\widetilde X$ and
	\begin{equation}\label{eq:secondextension}\left\langle \eta_\gamma\cdot \pi_\gamma^*T,\mathbf 1_{\widetilde X}\right\rangle_{\mathcal E'(\widetilde X),\mathcal E(\widetilde X)}=\langle T,\mathbf 1_{\Gamma_\gamma\backslash \widetilde X}\rangle_{\mathcal E'(\Gamma_\gamma \backslash \widetilde X),\mathcal E(\Gamma_\gamma \backslash \widetilde X)}.\end{equation}
\end{lem}

In order to apply Lemma \ref{intégration cylindre} we need to show that for fixed $t$, the distribution $\widetilde E(t,x,\gamma^kx)\in \mathcal D'(\Gamma_\gamma\backslash \widetilde X)$ is compactly supported. By finite propagation speed of the wave equation, it follows from this next lemma, relying on the hyperbolicity of the flow:

\begin{lem}\label{gamma écarte les points} There is some constant $C_0>0$ such that for any $\gamma\in \Gamma$ and $L>0$, we have 
\[\big\{\tilde x\in \Gamma_\gamma\backslash \widetilde X, \ d(\tilde x,\gamma \tilde x)\le L\big\}\subset \big\{\tilde x\in \Gamma_\gamma\backslash \widetilde X, \ \mathrm{dist}(\tilde x,\tilde \gamma)\le \frac{L-\ell(\gamma)+C_0}{2}\big\},\]
where $\tilde \gamma$ denotes the axis of $\gamma$. As a consequence, there is some $C>0$ such that we have the volume estimate
\[\mathrm{vol}\ \{\tilde x\in \Gamma_\gamma\backslash \widetilde X, \ d(\tilde x,\gamma \tilde x)\le L\}=\mathcal O(\mathrm{e}^{CL}).\] 
\end{lem}

\begin{proof}By uniform hyperbolicity of the flow (see \emph{e.g.}  \cite[Lemma 9]{Parkkonen2015} we can show that $d(x,\gamma x)\ge 2\mathrm{dist}(x,\tilde \gamma)+\ell(\gamma)-C_0$, with $C_0$ depending only on the geometry of $X$. The volume estimate follows by covering the set  $\big\{\tilde x\in \Gamma_\gamma\backslash \widetilde X, \ \mathrm{dist}(\tilde x,\tilde \gamma)\le \frac{L-\ell(\gamma)+C_0}{2}\big\}$ by a ball of radius $CL$.\end{proof}

With these lemmas in hand, we can perform computations rightfully, by writing
\[\int_{\mathbf R}\operatorname{Tr} E^\rho(t) \varphi(t)\mathrm dt=\int_{X\times \mathbf R}\operatorname{Tr}^{\rm End} E^\rho(t,x) \varphi(t)\mathrm dt\mathrm dx\underset{\eqref{eq:firstextension}}=\int_{\widetilde X\times \mathbf R} \pi^*\operatorname{Tr}^{\rm End} E^\rho(t,\tilde x) \eta(\tilde x)\varphi(t)\mathrm dt\mathrm d\tilde x.\]
Now we expand $\pi^*\operatorname{Tr^{End}} E^\rho$ as a sum over conjugacy classes
\[\pi^*\operatorname{Tr^{End}} E^\rho(t,\tilde x)=\dim(V_\rho)\widetilde E(t,\tilde x,\tilde x)+\sum_{\gamma\in \mathcal P}\sum_{k\ge 1}\left(\sum_{g\in \Gamma_\gamma\backslash \Gamma} \chi_\rho(g\gamma^k g^{-1})\widetilde E(t,\tilde x,g\gamma^k g^{-1}\tilde x)\right). \]
By cyclicity of the trace, $\chi_\rho(g\gamma^k g^{-1})=\chi_\rho(\gamma^k)$. Also, by invariance under isometries of the wave kernel, $\widetilde E(t,\tilde x,g\gamma^k g^{-1}\tilde x)=\widetilde E(t, g^{-1}\tilde x,\gamma^k g^{-1}\tilde x)$. We integrate over $\widetilde X$ and perform a change of variable $g^{-1}\tilde x\mapsto \tilde x$ in every integral of the sum to get
\[\sum_{g\in \Gamma_\gamma\backslash \Gamma}\int_{\widetilde X}\widetilde E(t, g^{-1}\tilde x,\gamma^k g^{-1}\tilde x) \eta(\tilde x)\mathrm d\tilde x=\int_{\widetilde X} \widetilde E(t,\tilde x,\gamma^k \tilde x)\eta_\gamma(\tilde x)\mathrm d\tilde x\underset{\eqref{eq:secondextension}}=\int_{\Gamma_\gamma\backslash \widetilde X} \widetilde E(t,\tilde x,\gamma^k \tilde x)\mathrm d\tilde x,\]
this concludes the proof of Proposition \ref{prop:decompositioncylindre}. 

At this point, we replace the kernel by the parametrix (modulo the error terms) and write $\widetilde E(t,\tilde x,\tilde x)=E_N(t,x,x)+\widetilde R_N(t,\tilde x,\tilde x)$ and $\widetilde E(t,\tilde x,\gamma \tilde x)=\widetilde E_N(t,\tilde x,\gamma \tilde x)+\widetilde R_N(t,\tilde x,\gamma \tilde x)$. We first handle terms involving remainders $\widetilde R_N$, when testing $\operatorname{Tr}E^\rho$ against the function $t\mapsto \cos(\lambda t)\widehat\psi(t/L)$, when $\psi$ is as in Theorem \ref{trace formula for long times}.
\begin{prop}There is a constant $C>0$ such that for any $\gamma\in \Gamma$, we have
\begin{equation}\label{intégrale du reste}\int_{\Gamma_\gamma\backslash \widetilde X}\int_{\mathbf R} \widetilde R_N(t,\tilde x,\gamma \tilde x)\cos(\lambda t)\widehat \psi(t/L)\mathrm dt\mathrm d\tilde x=\mathbf 1_{\ell(\gamma)\le L}\mathcal O_{\psi}(\lambda^{-1}\mathrm{e}^{CL}).\end{equation}
\end{prop}

\begin{proof}
Since the support of $\widehat \psi$ lies in $[-1,1]$, the integrand is supported in $d(\tilde x,\gamma \tilde x)\le |t|\le L$. Therefore, the integral vanishes if $\ell(\gamma)\ge L$ (since $d(\tilde x,\gamma \tilde x)\ge \ell(\gamma)$). 
Thanks to Lemma \ref{majoration reste intégral}, for any $\tilde x$, the integral in the time variable in (\ref{intégrale du reste}) is bounded by $\mathcal O(\lambda^{-1}\mathrm{e}^{CL})\mathbf 1_{d(\tilde x,\gamma \tilde x)\le L}$. Integrating over $\Gamma_\gamma\backslash \widetilde X$, we can invoke Lemma \ref{gamma écarte les points} to obtain \eqref{intégrale du reste}, up to increasing the constant $C$.
\end{proof}
So far, we have proved (we drop the $\sim$ superscripts on $x$ and $y$ at this point)
\begin{prop} \label{prop:sofar} For any $N$ large enough, there is some $C>0$ such that
\[\big\langle \operatorname{Tr} E^\rho(t),\cos(\lambda t)\widehat \psi(t/L)\big\rangle =\dim(V_\rho)\left(\big\langle E_N(t,x,x),\cos(\lambda t)\widehat \psi(t/L)\otimes \mathbf 1_X\big\rangle +\mathcal O(\lambda^{-1}\mathrm{e}^{CL})\right) \]\[+\sum_{\gamma \in \mathcal P	}\sum_{k\ge 1}\chi_\rho(\gamma^k )\mathbf 1_{k\ell(\gamma)\le L}\left(\big\langle E_N(t,x,\gamma^k x), \cos(\lambda t)\widehat \psi(t/L)\otimes \mathbf 1_{\Gamma_\gamma\backslash\widetilde X}\big\rangle +\mathcal O(\lambda^{-1}\mathrm{e}^{CL}) \right). \]
The error terms are uniform with respect to $\gamma,k$, and \emph{do not} depend on the representation $\rho$.\end{prop}
Note the necessity to impose $L\le c\log \lambda$ to make error terms negligible.
\subsection{Integrating the time variable} Looking at Proposition \ref{prop:sofar}, we have to provide estimates for $\big\langle E_N(t,x,x),\cos(\lambda t)\widehat \psi(t/L)\otimes \mathbf 1_X\big\rangle$ and $\big\langle E_N(t,x,\gamma x), \cos(\lambda t)\widehat \psi(t/L)\otimes \mathbf 1_{\Gamma_\gamma\backslash\widetilde X}\big\rangle$ for any $\gamma\in \mathcal G$, in the regime $\lambda\to +\infty$.

We fix $x$ in $\widetilde X$ and first compute integrals in the time variable. Let $\gamma\in \Gamma$ and $k\ge 0$. Recalling the definition of the parametrix \eqref{eq:defparametrix}, we have to estimate
\[\begin{array}{rl}J_k&:=\displaystyle\int_{-\infty}^{+\infty} \cos(\lambda t) \widehat \psi\left(t/L\right) t^{2k-2}_{\mathrm{reg}} \ \mathrm dt, \\
I_k(\gamma,x)& \displaystyle:=\int_{-\infty}^{+\infty} \cos(\lambda t) \widehat \psi\left(t/L\right) |t|(d(x,\gamma x)^2-t^2)_-^{k-\frac 32} \ \mathrm dt, \ \ \ \ \ \gamma\in \Gamma\backslash\{\mathrm{id}\}.\end{array}\]
We rely on results of B\'erard:
\begin{prop}[{\cite[(58,59)]{Bérard1977}}] For all $N>0$ and $\lambda,L\ge 1$, we have
\[J_k = -\delta_{0,k} \pi \lambda \widehat \psi(0) +\mathcal O_{N,k}(L^{2k-1}\lambda^{-N})\]
where $\delta_{0,k}$ equals $1$ if $k=0$ and $0$ otherwise. Also,
\[\begin{array}{ll} I_k(\gamma,x)  =&\displaystyle  2\mathrm{Re} \ \mathrm{e}^{\mathrm i\lambda d(x,\gamma x)} d(x,\gamma x)^{k-\frac 12} \lambda^{-k+\frac 12} \beta_{0,k} \widehat \psi\left(\frac{d(x,\gamma x)}{L}\right) \\ & \displaystyle+\mathbf 1_{d(x,\gamma x)\le L}\mathcal O_{N,k}\big(\lambda^{-k-\frac 12}+L^{2k-1} \lambda^{-N})\end{array},\]
where the $\beta_{0,k}$ are some numerical constants. In particular, for $k=0$, we have $\beta_{0,0}=(-1+\mathrm{i})\sqrt{\pi}/2$, thus
\[I_0(\gamma,x)=\sqrt{2\pi} \cos\left(\lambda d(x,\gamma x)+\frac{3\pi}{4}\right)d(x,\gamma x)^{-\frac 12} \lambda^{\frac 12} \widehat \psi\left(\frac{d(x,\gamma x)}{L}\right) +\mathbf 1_{d(x,\gamma x)\le L}\mathcal O(\lambda^{-\frac 12}). \]\end{prop}

\begin{rem}The value of $\beta_{0,0}$ was not specified by B\'erard, and it is given without proof in \cite{Donnelly1978}. It basically follows from the computation of
\[\int_{0}^{+\infty} u_+^{-\frac 32}\mathrm{e}^{\mathrm i\lambda u}\mathrm du.\]\end{rem}

Combining these estimates, using the identity $\cos(u+\pi)=-\cos(u)$, and by using the fact that the set $\{x\in \Gamma_\gamma\backslash \widetilde X, \ \ d(x,\gamma x)\le L\}$ has volume $\mathcal O(\mathrm{e}^{CL})$, along with Proposition \ref{borne un} to control the terms of the parametrix of order $\ge 1$, we get
\begin{prop} \label{avant intégrale temps}Take $\psi$ as in Theorem \ref{trace formula for long times}. There is some constant $C>0$ such that for $\lambda,L\ge 1$,
\[\begin{array}{ll}\langle \operatorname{Tr} E^\rho \displaystyle ,\cos(\lambda t)\widehat \psi(t/L)\rangle =\displaystyle \frac 12\dim(V_\rho)\mathrm{vol}(X)\big(\widehat \psi(0)+\mathcal O(\lambda^{-1} \mathrm{e}^{CL})\big)  + \displaystyle \frac{1}{\sqrt{ 2\pi}}\sum_{\gamma\in \mathcal P} \displaystyle\sum_{k\ge 1} \chi_\rho(\gamma^k)\times\\  
	
	\displaystyle \times \left(	\int_{\Gamma_\gamma\backslash \widetilde X}\lambda^{1/2} \tilde u_0(x,\gamma^k x) \cos\big(\lambda d (x,\gamma^kx)-\frac{\pi}{4}\big) d(x,\gamma^kx)^{-\frac 12} \widehat \psi\left(\frac{d(x,\gamma^k x)}{L}\right)\mathrm dx + \mathcal O(\lambda^{-\frac 12}\mathrm{e}^{CL})\mathbf 1_{k\ell(\gamma)\le L}\right)	.\end{array}\]
The error terms are functions that do not depend on the representation $\rho$. Moreover, they are uniformly bounded with respect to $\gamma,k$.
\end{prop}

\subsection{Stationary phase in the space variable}
We now want to evaluate, for $\gamma\in \mathcal G$ (recall that $\mathcal G$ denotes the set of oriented closed geodesics), the integral
\begin{equation}\begin{array}{ll} L(\gamma) =\displaystyle\mathrm{Re}\  \mathrm{e}^{\frac{\mathrm{i}\pi}{4}}\int_{\Gamma_\gamma\backslash \widetilde X} \tilde u_0(x,\gamma x) \mathrm{e}^{-\mathrm i\lambda d(x,\gamma x)} d(x,\gamma x)^{-\frac 12} \widehat \psi\left(\frac{d(x,\gamma x)}{L}\right) \mathrm dx.\end{array}\end{equation}
The following computation is standard (see \cite{ColindeVerdière1973}, \cite[eq. 3.10]{JakobsonWeyl2007}) but we review it for the sake of completeness. The main tool is the following stationary phase lemma:

\begin{lem}[{\cite[Lemme 2]{ColindeVerdière1973}}] \label{Lemme 2 CdV} Let $\alpha:\mathbf R\to \mathbf R$ be a $C^5$ convex function, such that $0$ is its only critical point, and  let $h:\mathbf R\to \mathbf R$ be a $C^2$ function supported in $[-R,R]$. Then,
\[\int_{\mathbf R} \mathrm{e}^{-\mathrm{i}\lambda \alpha(t)} h(t)\ \mathrm dt=  \mathrm{e}^{- \mathrm{i}\lambda \alpha(0)}\left(\frac{2\pi}{\mathrm{i}\lambda \alpha''(0)}\right)^{\frac 12} h(0) +\mathcal O(\lambda^{-1}Q(\lambda)), \]
where $Q(\lambda)$ is bounded by a polynomial function in the variables $\frac{1}{\alpha''(0)}$, $\|\alpha\|_{C^5([-R,R])}$, $\|h\|_{C^2([-R,R])}$ and $R$. \end{lem}

\textbf{Fermi coordinates on $\Gamma_\gamma\backslash \widetilde X$.} Let $\tilde \gamma_*:\mathbf R/ \ell(\gamma)^\sharp\mathbf Z\to \Gamma_\gamma\backslash \widetilde X$ denote the unique closed primitive geodesic on $\Gamma_\gamma\backslash \widetilde X$ (up to reversing orientation). We introduce Fermi coordinates along this geodesic, 
\[\Phi(t,s)=\exp_{\tilde \gamma_*(s)} (t \mathbf n (s)), \qquad t\in \mathbf R, s\in \mathbf R/ \ell(\gamma)^\sharp \mathbf Z. \]
where $\mathbf n$ is the normal unit vector to the geodesic (chosen so that $( \tilde\gamma_*'(s), \mathbf n)$ is a direct basis of $T_{\tilde \gamma_*(s)} \widetilde X$). We denote $j(t,s)$ the Jacobian $|\det \mathrm d\Phi(t,s)|$. Note that $j(0,s)=1$.
After changing variables,
\begin{equation}\label{def L(gamma,k)}L(\gamma)=\mathrm{Re}\ \mathrm{e}^{\frac{\mathrm{i}\pi}{4}}\int_0^{\ell(\gamma_*)}\int_{-\infty}^{+\infty} \mathrm{e}^{-\mathrm{i}\lambda d(\Phi,\gamma \Phi)}\Theta^{-1/2}(\Phi,\gamma \Phi)d(\Phi,\gamma \Phi)^{-\frac 12} \widehat \psi\left(\frac{d(\Phi,\gamma \Phi)}{L}\right) \ j(t,s) \mathrm dt\mathrm ds,\end{equation}
We recall the following facts:
\begin{itemize} \item The phase function $t\mapsto d(\Phi(t,s),\gamma \Phi(t,s))$ is smooth, strictly convex, with a unique critical point at $t=0$. Also, its second derivative at $t=0$ is bounded by below by a uniform positive constant independent of $\gamma$ or $s$.
\item More precisely, the Hessian $\partial_t^2\Phi(0,s)$ can be expressed via dynamical quantities (see \cite[\S4]{Donnelly1978})
\[\partial_t^2\Phi(0,s)=\frac{|I-\mathcal P_\gamma|}{\ell(\gamma)\Theta(\Phi(0,s),\gamma \Phi(0,s))}.\]
A proof of this fact can be found in the appendix of \cite{Millson}.
\item All the functions $\Theta^{-1/2}(\Phi,\gamma \Phi)$, $d(\Phi,\gamma \Phi)^{-\frac 12}$, $j(t,s)$ have their derivatives of order $\le 5$ bounded by $\mathcal O(\mathrm{e}^{C d(\Phi(t,s),\gamma \Phi(t,s))})$, see \cite{keeler2022logarithmic,ColindeVerdière1973}, so the same holds for their product. These quantities are $\mathcal O(\mathrm{e}^{CL})$ on the support of integration.
\item The support of the integrand as a function of $t$ has size $\mathcal O(\mathrm{e}^{CL})$, uniformly with respect to $\gamma,s$.
\end{itemize}
With these at hand, we can apply Lemma \ref{Lemme 2 CdV}, $s$ being fixed, to get
\begin{equation}\label{intermede L(gamma,k)} \begin{array}{ll}\displaystyle\int_{-\infty}^{+\infty} \mathrm{e}^{-\mathrm{i}\lambda d(\Phi,\gamma \Phi)}\Theta^{-1/2}(\Phi,\gamma \Phi)d(\Phi,\gamma \Phi)^{-\frac 12} \widehat \psi\left(\frac{d(\Phi,\gamma \Phi)}{L}\right) \ j(t,s) \mathrm dt \\
	
	\displaystyle=\mathrm{e}^{-\mathrm{i}\lambda \ell(\gamma)}\left(\frac{2\pi}{\lambda}\right)^{\frac 12	}\ell(\gamma)^{-\frac 12}\widehat \psi\left(\frac{\ell(\gamma)}{L}\right)  \frac{\Theta^{-1/2}(\Phi(0,s),\gamma \Phi(0,s))}{\sqrt{\mathrm{i}\partial_t^2 \Phi(0,s)}}  +\mathbf 1_{\ell(\gamma)\le L}\mathcal O(\lambda^{-1}\mathrm{e}^{CL})\end{array}.\end{equation}

We integrate the $s$-dependent term between $0$ and  $\ell(\gamma)^{\sharp}$ to get, using the expression of the Hessian of $x\mapsto d(x,\gamma x)$:
\begin{equation}\label{intermede L(gamma,k) bis}\int_0^{\ell(\gamma)^\sharp} \frac{\Theta^{-1/2}(\Phi(0,s),\gamma \Phi(0,s))}{\sqrt{\mathrm{i}\partial_t^2 \Phi(0,s)}}   \mathrm ds=\int_0^{\ell(\gamma)^{\sharp}} \frac{\ell(\gamma)^{\frac 12}}{ (\mathrm{i}|I-\mathcal P_\gamma|)^{\frac 12}} \mathrm ds=\mathrm{e}^{-\frac{\mathrm{i}\pi}{4}} \frac{\ell(\gamma)^{\frac 12}}{|I-\mathcal P_\gamma|^{\frac 12}} \ell(\gamma)^\sharp .\end{equation}
Finally, combining (\ref{def L(gamma,k)})--(\ref{intermede L(gamma,k) bis}),
\[L(\gamma)= \mathrm{Re}\ \mathrm{e}^{-\mathrm{i}\lambda \ell(\gamma)}\left(\frac{2\pi}{\lambda}\right)^{\frac 12	}\widehat \psi\left(\frac{\ell(\gamma)}{L}\right)\frac{\ell(\gamma)^{\sharp}}{|I-\mathcal P_\gamma|^{\frac 12}} +\mathcal O(\lambda^{-1}\mathrm{e}^{CL})\mathbf 1_{\ell(\gamma)\le L}.  \]
By what precedes and Proposition \ref{avant intégrale temps} we eventually obtain

\[\langle \operatorname{Tr} E^\rho,\cos(\lambda t)\widehat \psi(t/L)\rangle=\frac 12\dim(V_\rho)\mathrm{vol}(X)\lambda \big(\widehat \psi(0)+\mathcal O(\lambda^{-1}\mathrm{e}^{CL})\big)\]\[+\sum_{\gamma\in \mathcal P}\sum_{k\ge 1}\chi_\rho(\gamma)^k\left( \ \cos(\lambda k\ell(\gamma))\widehat \psi\left(\frac{k\ell(\gamma)}{L}\right)\frac{\ell(\gamma)}{|I-\mathcal P_\gamma^k|^{\frac 12}} +\mathcal O(\lambda^{-\frac 12}\mathrm{e}^{CL})\mathbf 1_{k\ell(\gamma)\le L}\right).\]
This is exactly the content of Proposition \ref{trace formula for long times}.

\section{Equidistribution results for closed geodesics} \label{sec:équidistributions}

For surfaces of constant curvature $K=-1$, if $\rho$ is an Abelian character of $\Gamma$, one can consider the \emph{twisted counting function}
\[\mathcal N_\rho(L)=\sum_{\gamma \in \mathcal G,\ \ell(\gamma)\le L} \rho(\gamma).\]
If $\rho$ is trivial, this function reduces to the standard counting function and $\mathcal N_\rho(L)$ grows like $\frac{\mathrm{e}^L}{L}$. For a nontrivial $\rho$, we have $\mathcal N_\rho(L)=\mathcal O(\mathrm{e}^{\delta L})$ for some $\delta<1$, as shown in \cite[Proposition 4.2]{naud2022random}. This result is derived by combining a lower bound on the first eigenvalue of $\Delta_\rho$ with the twisted Selberg trace formula. In our context, as we explained in the introduction, we need to consider the distribution of lengths of closed geodesics, weighted by the dynamical coefficients $\frac{\ell(\gamma)^{\sharp}}{|I-\mathcal P_\gamma|}$ --- recall \eqref{eq:Sigmacarréintro} ---, which do not only depend on the length of $\gamma$ when $X$ has variable curvature.
\begin{thm}\label{lem : equidistribution lemma}Let $\rho:\Gamma\to \mathrm{GL}(V_\rho)$ be a finite dimensional unitary representation of $\Gamma$, denote $\chi_\rho(\gamma)=\operatorname{Tr}(\rho(\gamma))$, and let $\varphi$  be a smooth compactly supported function. Then, 
\[\frac{1}{L}\sum_{\gamma\in \mathcal G} \chi_\rho(\gamma)\frac{\ell(\gamma)^{\sharp}\varphi(\ell(\gamma)/L)}{|I-\mathcal P_\gamma|}=d_{\mathrm{trivial}}\int_{0}^{+\infty} \varphi(t)\mathrm dt + \mathcal O\left(\frac{1}{L}\right).\]
Here, $d_{\mathrm{trivial}}$ is the dimension of the subspace of $V_\rho$ consisting of the vectors invariant under the action of $\Gamma$. The remainder improves to $\mathcal O(L^{-2})$ if $\varphi$ vanishes at $0$. We will eventually be concerned by sums over the set of closed \emph{primitive} geodesics only. Contribution of nonprimitive closed geodesics happen to be negligible in the large $L$ limit --- adding an error $\mathcal O(L^{-1})$ (or $\mathcal O( L^{-2})$ when $\varphi(0)=0$)  ---, thus Theorem \ref{lem : equidistribution lemma} still holds when summing over $\gamma\in \mathcal P$ instead. 
\end{thm} 

\begin{rem}For any weak mixing Anosov flow on a closed manifold, letting $\mathcal G$ denote the set of periodic orbits of the flow, one can show that
\[\frac{1}{L}\sum_{\gamma\in \mathcal G} \frac{\ell(\gamma)^{\sharp}\varphi(\ell(\gamma)/L)}{|I-\mathcal P_\gamma|}=\int_{0}^{+\infty} \varphi(t)\mathrm dt+\mathcal O\left(\frac{1}{L}\right),\]
as the local trace formula of Jin--Zworski still applies in this more general setting.

\end{rem}

We will need the following companion result. Again, we stress out that it holds for general mixing Anosov flows on a closed manifold.
\begin{lem}\label{lem: size clusters} Let $\omega$ be a smooth, compactly supported function on $\mathbf R$. Then,
\[\sum_{\gamma\in \mathcal G}\frac{\ell(\gamma)^\sharp\omega(\ell(\gamma)-T)}{|I-\mathcal P_\gamma|}\underset{T\to+\infty}\longrightarrow \int_{\mathbf R} \omega .\]
Also, uniformly in $T\ge 0$,
\[\sum_{T-1\le \ell(\gamma)\le T+1} \frac{\ell(\gamma)^{\sharp}}{|I-\mathcal P_\gamma|}=\mathcal O(1).\]\end{lem}

Theorem \ref{lem : equidistribution lemma} is obtained by combining the mixing property of the geodesic flow together with a trace formula for the transfer operator associated to the geodesic flow. For an Abelian character $\rho$, the results are a consequence of the works of Jin--Zworski \cite{Jin2016} and Jin--Tao \cite{jin2022flat,jin2023number}. We shall consider more general counting functions where $\chi_\rho(\gamma)$ is replaced by a polynomial function of $\operatorname{Tr}(\rho(\gamma))$. This is addressed in Proposition \ref{prop : equidistribution Lie}, which essentially follows from Theorem \ref{lem : equidistribution lemma} after an application of Peter--Weyl theorem.

\begin{rem} As it will appear clear in the proof, quantitative estimates on the remainder in Lemma \ref{lem: size clusters} can be obtained provided that one can find an explicit function $A=A(R)$ such that there are no Ruelle--Pollicott resonances other than $0$ in the domain $\{|\mathrm{Re}(z)|\ge R, \ -A(R) \le \mathrm{Im}(z)\le 0\}$. If one can take $A(R)=\mathrm{constant}>0$ we get exponential decay in $T$, if one can take $A(R)=c R^{-\theta}$ for some $c,\theta>0$ we get superpolynomial decay in $T$, etc. \end{rem}

We begin by reviewing some background on geodesic flows on negatively curved surfaces and the definition of Pollicott--Ruelle resonances. Then, we present the local trace formula of Jin--Zworski, used in the proof of Theorem \ref{lem : equidistribution lemma}.

\subsection{Geodesic flow on the cotangent bundle}

Let $M$ denote the unit cotangent bundle of $X$, defined by
\[M=\big \{(x,\xi)\in T^*X, \ |\xi|_g=1\big\},\]
where $|\cdot|_g$ is the dual metric. We let $\phi^t$ denote the geodesic flow on $M$, and denote by $V$ the geodesic vector field, which is the generator of this flow. Similarly, let $\widetilde M$ be the unit cotangent bundle of $\widetilde X$, and $\widetilde{V}$ the geodesic vector field on $\widetilde M$. We consider the \emph{transfer operator} $(\phi^t)^*$, that acts on smooth functions according to the formula
\[(\phi^{t})^*f:= f\circ \phi^t.\]
Viewing $V$ as a differential operator on $C^\infty(M)$, we have $(\phi^t)^*=\mathrm{e}^{tV}$.

Since $X$ is negatively curved, the geodesic flow on $M$ is \emph{Anosov} \cite{Ano67}. It means that there exists a continuous splitting of the tangent bundle preserved by the flow $T_xM=\mathbf RV(x)\oplus E_s(x)\oplus E_u(x)$, such that for any continuous norm $|\cdot|$ on $TM$, there exist constants $C,\theta>0$ such that
\begin{itemize}\label{defi:Anosov}
\item If $v_s\in E_s(x)$ then $|\mathrm d_x\phi^t(v_s)|\le C \mathrm{e}^{-\theta t}|v_s|$ for $t\ge 0$.
\item If $v_u\in E_u(x)$ then $|\mathrm d_x\phi^t(v_u)|\le C \mathrm{e}^{-\theta |t|}|v_u|$ for $t\le 0$.
\end{itemize}

The spaces $E_s$ and $E_u$ are respectively called \emph{stable} and \emph{unstable}; in our case where $X$ is a surface, both have dimension $1$. There is an associated continuous splitting $T_x^*M=E_0^*(x)\oplus E_s^*(x)\oplus E_u^*(x)$ of the cotangent bundle of $M$, defined by
\[E_0^*(x)\big(E_s(x)\oplus E_u(x)\big)=\{0\}, \ E_s^*(x)(\mathbf RV(x) \oplus E_u(x))=\{0\}, \ E_u^*(x)\big(\mathbf RV(x)\oplus E_s(x)\big)=\{0\}.\] 

The geodesic flow preserves a contact one form $\alpha$, defined as the restriction to $M$ of the canonical $1$-form on $T^*X$. The volume form $\alpha\wedge \mathrm d\alpha$ is called the \emph{Liouville measure}, and we denote by $\nu_0$ its normalization (such that $\int_M \mathrm d\nu_0=1$).

\subsection{Anisotropic Sobolev spaces and Pollicott--Ruelle resonances} \label{subsec:anisotropic}

Let $\mathbf P:=-\mathrm iV$, since the flow preserves the Liouville measure on $M$, the operator $\mathbf P$ is formally self-adjoint on $L^2(M,\nu_0)$. However, $L^2(M,\nu_0)$ is not well-fitted to the study of $\mathbf P$, for example the spectrum of $\mathbf P$ on $L^2(M,\nu_0)$ consists of the whole real line. Following earlier works of Blank--Keller--Liverani \cite{blank2002ruelle}, Gou{\"e}zel--Liverani \cite{gouezel2006banach} and Baladi--Tsujii \cite{baladi2007anisotropic}, Faure and Sjöstrand introduced convenient functional spaces --- so called \emph{anisotropic Sobolev spaces} --- to study $\mathbf P$, that are well adapted to the dynamics of the flow \cite{Faure2011}. Here we adopt the notations of Dyatlov--Zworski \cite[\S 3.1]{dyatlov2016dynamical}. Briefly, one constructs an \emph{order function} $m_G\in C^\infty(T^*M\backslash \{0\}, [-1,1])$, homogeneous of degree $0$, such that $m_G\equiv 1$ in a neighborhood of $E_s^*$ and $m_G\equiv -1$ in a neighborhood of $E_u^*$, that decreases along the flow lines of $\phi^t$. Then, one takes some pseudodifferential operator $G$ with principal symbol $\sigma(G)=m_G(x,\xi) \log |\xi|$, (here $|\cdot|$ is any norm on $T^*M$). The function $\sigma(G)=m_G(x,\xi) \log |\xi|$ is called an \emph{escape function}. Eventually, one defines
\[H_{sG}(M):=\mathrm{e}^{-sG} L^2(M).\]
The space $H_{sG}(M)$ comes with a Hilbert space structure by setting $\|f\|_{H_{sG}(M)}=\|\mathrm{e}^{sG} f\|_{L^2(M)}$. Moreover, by construction of $G$ --- it essentially follows from the fact that $m_G$ takes values in $[-1,1]$ ---, one can show that $H^{-s}(M)\subset H_{sG}(M)\subset H^s(M)$, where $H^\bullet(M)$ denote the standard Sobolev spaces on $M$. Eventually, define the domain $D_{sG}:=\{u\in H_{sG}, Pu\in H_{sG}\}$.

Let $s>0$. By \cite[\S 3.2]{dyatlov2016dynamical}, if $C_1$ large enough, then the operator $\mathbf P-z:D_{sG}\to H_{sG}$ is invertible for $\mathrm{Im}(z)>C_1$, and 
\[(\mathbf P-z)^{-1}=\mathrm{i}\int_0^{+\infty}\mathrm{e}^{-\mathrm{i} t(\mathbf P-z)}\mathrm dt.\]
Here $\mathrm{e}^{-\mathrm{i} t\mathbf P}$ denotes the pullback operator $(\phi^{-t})^*$, and the integral in the right-hand side converges in operator norm $H^{\pm s}\to H^{\pm s}$. By the work of Faure--Sj\"ostrand \cite{Faure2011}, for any $C_0>0$, for $s$ large enough, the resolvent $\mathbf R(z)=(\mathbf P-z)^{-1}:H_{sG}\to H_{sG}$ is a meromorphic family of operators with poles of finite rank, in the half-plane $\mathrm{Im}(z)>-C_0$. The poles of $R(z)$ are called \emph{Pollicott--Ruelle resonances}, and can be thought of as generalized eigenvalues of $P$. Moreover, these poles do not depend on $s$.

Actually, one can show that a complex number $\lambda$ is a resonance if and only if there is some nonzero distribution $u\in \mathcal D'_{E_u^*}(M)$ --- meaning $u\in \mathcal D'(M)$ and $\mathrm{WF}(u)\subset E_u^*$ ---, such that $(\mathbf P-\lambda) u=0$ in the weak sense. This implies in particular that the definition of resonances does not depend on the construction of the anisotropic Sobolev spaces.


\subsection{Ruelle operator acting on sections of flat bundles}
More generally, let $\mathcal E$ be the lift to $M$ of a flat bundle $\mathcal V_\rho\to X$ associated to a unitary representation $\rho$ of $\Gamma$. The bundle $\mathcal E$ is defined as the quotient $\Gamma\backslash (\widetilde M\times V_\rho)$, where we identify
\[\big((x,\xi),v\big)\sim \Big(\big(\gamma x, (\mathrm d\gamma(x)^{\top})^{-1}\xi\big), \rho(\gamma )v\Big).\]

Since the geodesic flow commutes with isometries, it preserves the space of $\rho$-equivariant functions $\widetilde M\to V_\rho$. Thus, the operator $\widetilde{V}$ acting on $C^\infty(\widetilde M,V_\rho)$ descends to an operator $\mathbf V$ on $C^{\infty}(M,\mathcal E)$, that is simply the Lie derivative of the geodesic vector field. 

Alternatively, for $1$-dimensional representations, we can take the viewpoint of flat connections as discussed in \S \ref{time reversal}. Let $\tilde \pi:\widetilde M\to \widetilde X$ be the projection. Observe that with the notations of \S \ref{time reversal}, for any $f\in C^\infty(\widetilde M)$, we have
\[\mathrm{e}^{-\mathrm i\tilde \pi^* \varphi}\widetilde V(\mathrm{e}^{\mathrm i\pi^* \varphi} f)=   f+\mathrm i \widetilde V(\pi^* \varphi) f.\]
Now recall that $\tilde \pi_* \widetilde V(x,\xi)=\xi^{\sharp}$, where $\bullet^\sharp$ denotes the natural identification $T^*\widetilde X\simeq T\widetilde X$. Since $\mathrm d\varphi=\widetilde {\mathbf A}$, we infer $\widetilde V(\tilde \pi^* \varphi)(x,\xi)=\mathrm d_x\varphi(\xi^\sharp)=\langle \widetilde {\mathbf A}(x),\xi\rangle_{T_x^*\widetilde X}$. Therefore, the operator $\mathbf V$ acting on $C^{\infty}(M,\mathcal E)$ is unitarily conjugated to the operator $V+\mathrm i\langle \mathbf A(x),\xi\rangle$ acting on $C^{\infty}(M)$ --- for the $L^2$ norm.

Let us first mention the following upper bound on the number of resonances:
\begin{prop}[{\cite[\S 6.1]{jin2023number}}] \label{comptage JinTao}For $A>0$, denote $N_A(E)=\# \mathrm{Res}(\mathbf P)\cap \{\lambda \in \mathbf C, \ |\mathrm{Re}\ \lambda|<E, \ \mathrm{Im}(\lambda)>-A\}$. Then,
\begin{equation}\label{eq : upper bound resonances} N_A(E)=\mathcal O_A(E^{n+1}).\end{equation}\end{prop}
By a standard hyperbolicity argument (see \emph{e.g.} \cite[Lemma 3.2]{Dyatlov2017}) we show
\begin{lem} \label{lem: resonance en 0} Let $\rho:\Gamma\to \mathrm{GL}(V_\rho)$ be a finite dimensional unitary representation of $\Gamma$. Let $\mathcal E$ be the associated vector bundle over $M=S^*X$, and let $\mathbf P=-\mathrm{i}\mathbf V$ act on smooth sections of $\mathcal E$. Then, $0$ belongs to $\mathrm{Res}(\mathbf P)$ if and only if $\rho$ contains the trivial representation. More precisely, the multiplicity of $0$ as a resonance is given by the dimension of the subspace of $V_\rho$ consisting of vectors invariant under the action $\Gamma$, that we denote $d_{\mathrm{trivial}}$. There are no other resonances on the real axis.\end{lem}

\begin{proof} If $u\in \mathcal D'_{E_u^*}(M,\mathcal E)$ satisfies $\mathbf Vu=0$, then (see \cite[Proposition 9.3.3]{Lefeuvre}) $u\in C^{\infty}(M,\mathcal E)$. We can then lift $u$ to a smooth equivariant function $\tilde u\in C^{\infty}(\widetilde M,V_\rho)$, that is invariant under the action of the geodesic flow, that is satisfying $\tilde u\circ \widetilde \phi^t=\tilde u$. By differentiating both sides of this identity, we have
\[\mathrm d\tilde u(x)=\mathrm d\tilde u(\phi^t(x))\circ \mathrm d\widetilde \phi^{t}(x).\]
Since $\mathrm d\tilde u$ is a lift of $\mathrm d u$, it is bounded in $L^\infty$-norm, and the Anosov property of the flow implies that $\mathrm d\tilde u$ vanishes on $E_s\oplus E_u$. Since $\widetilde V \tilde u=0$, we eventually get $\mathrm d\tilde u=0$. Thus, $\tilde u$ is a constant function with value $v\in V_\rho$, and by the equivariance property $\tilde u(\gamma x)=\rho(\gamma)\tilde u(x)$, we infer $\rho(\gamma)v=v$ for all $\gamma\in\Gamma$. If $v\neq 0$ it follows that $\rho$ contains the trivial representation. Conversely, if $\rho$ contains the trivial representation, then there is some $v\in V_\rho$ such that $\Gamma$ acts trivially on the line $\mathbf C v$. Then, the constant function $\widetilde M\to V_\rho$ taking value $v$ descends on a nontrivial smooth section of $\mathcal E$, that belongs to the kernel of $\mathbf V$.

For the second part of the assertion, assume that there is some nonzero $u\in \mathcal D'_{E_u^*}(M,\mathcal E)$ such that $\mathbf Vu=\mathrm i\lambda u$, with $\lambda\in \mathbf R$. Then, by \cite{Lefeuvre} again, we have $u\in C^{\infty}(M,\mathcal E)$ and, as before, we can lift $u$ to a smooth function $\tilde u:\widetilde M\to V_\rho$ and show $\mathrm d\tilde u_{|E_s\oplus E_u}=0$. Since $E_s\oplus E_u=\ker \alpha$ (where $\alpha$ is the contact $1$-form preserved by the flow) we can write $\mathrm du=s\alpha$ for some $s\in C^\infty(\widetilde M)$. Then, on the one hand
\[\alpha\wedge \mathrm d(s\alpha)=\alpha\wedge \mathrm d^2\tilde u=0,\]
and on the other hand, since $\alpha\wedge \alpha=0$,
\[\alpha\wedge \mathrm d(s\alpha)=s\alpha\wedge \mathrm d\alpha.\]
Since $\alpha\wedge \mathrm d\alpha$ is nowhere vanishing we deduce $s=0$. Therefore $\mathrm d\tilde u=0$, implying in particular $\mathbf Vu=0$, but $\mathbf Vu=\mathrm i\lambda u$ and $u$ is assumed to be nonzero, thus $\lambda=0$.
\end{proof}

\subsection{The local trace formula for the transfer operator}
The pullback operator $(\phi^{-t})^*=\mathrm{e}^{-\mathrm it\mathbf P}$ is not trace class, still, we can define its \emph{flat trace}, that is morally the integral of its Schwartz Kernel along the diagonal of $M\times M$ (see \cite{Guillemin1977Lectures}, and \cite{dyatlov2016dynamical} for a recent proof):
\begin{thm}[Guillemin trace formula] Let $\mathcal E$ be a flat vector bundle over $M$ associated to a unitary representation of $\Gamma$. Denote $\mathbf P=-\mathrm i\mathbf V$ acting on smooth sections of $\mathcal E$. Also, take a smooth potential $a\in C^\infty(M,\mathbf C)$. Then, in the sense of distributions on $(0,+\infty)$: 
\begin{equation}\label{eq:guillemin}\operatorname{tr^\flat} \mathrm{e}^{-\mathrm it(\mathbf P+a)} =\sum_{\gamma\in \mathcal G}\mathrm{e}^{-\mathrm{i} \int_\gamma a} \chi_\rho(\gamma) \frac{\ell(\gamma)^{\sharp} \delta(t-\ell(\gamma))}{|I-\mathcal P_\gamma|}.\end{equation}
where the left hand-side denotes the \emph{flat trace} of $\mathrm{e}^{-\mathrm it(\mathbf P+a)}$, which is well-defined by the assumption that the flow is Anosov.
\end{thm}

It is tempting to write
\[\operatorname{tr^\flat}  \mathrm{e}^{-\mathrm it\mathbf P}=\sum_{\mu\in \mathrm{Res}(\mathbf P)} \mathrm{e}^{-\mathrm i\mu t},\]
where $\mu$ runs over the set of resonances of $\mathbf P$. This formula holds for the geodesic flow on a hyperbolic surface, in which case resonances can be expressed --- up to some exceptional values --- from the eigenvalues of the Laplacian. It is not known, however, whether it holds for more general Anosov flows. Nevertheless, Jin--Zworski proved a local trace formula for Anosov flows {\cite[Theorem 1]{Jin2016}}. Jin--Tao discussed the case of operators acting on vector bundles --- in the more general setting of \emph{Axiom A} flows. We shall use the following version:


\begin{thm}[Local trace formula, {\cite[\S 6.1]{jin2023number}}]\label{localtraceformula} Let $\mathcal E$ be a vector bundle over $M$ associated to a unitary representation $\rho$ of $\Gamma$. Denote by $\mathbf P$ the operator $-\mathrm i\mathbf V$ acting on $C^{\infty}(M,\mathcal E)$. Then, for any $A>0$, there is a distribution $F_A\in \mathcal S'(\mathbf R)$, supported on $\mathbf R_+$, such that for any $\alpha\in [0,\alpha_0]$, the following equality of distributions on $(0,+\infty)$ holds:
\begin{equation}\label{eq : local trace formula}\sum_{\mu\in \mathrm{Res}(\mathbf P), \ \mathrm{Im}(\mu)>-A} \mathrm{e}^{-\mathrm i\mu t}+ F_A(t)=\sum_{\gamma\in \mathcal G} \chi_\rho(\gamma) \frac{\ell(\gamma)^{\sharp}\delta(t-\ell(\gamma))}{|I-\mathcal P_\gamma|}.\end{equation}
Moreover, the Fourier transform of $F_A$ admits an analytic continuation to the half plane $\mathrm{Im}(\xi)\le A$, and satisfies $|\widehat F_A(\xi)|=\mathcal O_{A,\varepsilon}(\langle \xi\rangle ^{7+\varepsilon})$ in the region $\mathrm{Im}(\xi)\le A-\varepsilon$, for any $\varepsilon>0$.
\end{thm}

\subsection{Proof of Theorem \ref{lem : equidistribution lemma}} We turn to the proof of Theorem \ref{lem : equidistribution lemma}. It is a quite straightforward consequence of the local trace formula.

\begin{proof}Let $\eta$ be a smooth function such that $\eta\equiv 0$ on $(-\infty,\ell_0/2)$ and $\eta(t)=1$ for $t\ge \ell_0$, where $\ell_0$ denotes the length of the shortest closed geodesic on $X$. We define a smooth, compactly supported function $\varphi_L$ by letting $\varphi_L(t)=\eta(t)\frac{1}{L}\varphi(t/L)$. Fix $A>0$, then the local trace formula \eqref{eq : local trace formula} applied to $\varphi_L$ gives
\begin{equation}\label{LTF applied to phi}d_{\mathrm{trivial}} \int_0^{+\infty} \eta(t)\frac{1}{L}\varphi(t/L)\mathrm dt +\sum_{\underset{\operatorname{Im}(\mu)>-A}{\mu\in \mathrm{Res}(\mathbf P)\backslash \{0\}} } \widehat \varphi_L(\mu)+\langle F_A,\varphi_L\rangle= \frac 1L\sum_{\gamma\in \mathcal G} \chi_\rho(\gamma)\frac{\ell(\gamma)^{\sharp}\varphi(\ell(\gamma)/L)}{|I-\mathcal P_\gamma|}.\end{equation}
Notice that $\eta$ does not appear in the right-hand side since $\eta$ has been chosen to ensure $\eta(\ell(\gamma))=1$ for any $\gamma$. We compute the Fourier transform of $\varphi_L$:
\[\widehat \varphi_L(\xi)=\frac{1}{2\pi}\int_{\mathbf R}\mathrm{e}^{-\mathrm i\xi L t}\eta(Lt) \varphi(t)\mathrm dt.\] 
For nonzero $\xi$ with nonpositive imaginary part, by repeated integration by parts we show
\begin{equation}\label{eq : ipp f chapeau}|\widehat \varphi_L(\xi)|\le |L\xi|^{-N}\int_{\mathbf R} \left|\frac{\mathrm d^N}{\mathrm dt^N}\big(\eta(Lt)\varphi(t)\big)\right|\mathrm dt,\end{equation}
where we have used that $|\mathrm{e}^{-\mathrm{i} t\xi}|\le 1$ for $t\ge 0$ and $\operatorname{Im} \xi\le 0$. If $k\in \{1,\ldots,N\}$ derivatives in \eqref{eq : ipp f chapeau} hit $\eta(Lt)$, since $\eta'$ is supported in $[0,\ell_0]$, the corresponding term in the integral in the right-hand side of \eqref{eq : ipp f chapeau} is bounded by
\begin{equation}\label{eq : borne dérivée f chapeau} \int_{\mathbf R} L^k |\eta^{(k)}(Lu)\varphi^{(N-k)}(u)|\mathrm du\lesssim_k L^k \|\eta\|_{C^k}\int_{0}^{\ell_0/L} |\varphi^{(N-k)}(u)|\mathrm du=\mathcal O_N(L^{k-1}).\end{equation}
For $k=0$ we have
\begin{equation}\label{en 0} \int |\eta(Lu)\varphi^{(N)}(u)|\mathrm du=\mathcal O_N(1). \end{equation}
Altogether, gathering \eqref{eq : borne dérivée f chapeau}, \eqref{en 0} we obtain
\begin{equation}\label{borne f chapeau}|\widehat \varphi_L(\xi)|=\frac{1}{L}\mathcal O_N(|\xi|^{-N}), \qquad  \mathrm{Im}(\xi)\le 0.\end{equation}  
Take $\varepsilon=\frac A2$ in the local trace formula. By combining (\ref{borne f chapeau}) together with the bound $\widehat F_A(\xi)=\mathcal O(\langle \xi\rangle^{7+\frac A2})$ of Theorem \ref{localtraceformula}, we get for  $N$ large enough
\begin{equation}\label{FA contre phi}\langle F_A,\varphi_L\rangle=2\pi\int_{\mathbf R+ \mathrm{i} \frac{A}{2}} \widehat F_A(\xi)\widehat \varphi_L(-\xi)\mathrm d\xi=\mathcal O_A\left(\frac{1}{L}\right).\end{equation}
Since $0$ is isolated in the set of resonances, we can find some $\varepsilon>0$ such that any nonzero resonance $\mu$ satisfies $|\mu|\ge \varepsilon$, and in turn $|\mu|\ge \min(\varepsilon,|\mathrm{Re}\ \mu|)$. Then, \eqref{borne f chapeau} and the counting upper bound \eqref{eq : upper bound resonances} in the strip $\operatorname{Im} z >-A$ yield
\begin{equation}\label{somme res partie plus grande que A}\bigg|\sum_{\underset{\operatorname{Im}(\mu)>-A}{\mu\in \mathrm{Res}(\mathbf P)\backslash \{0\} }} \widehat \varphi_L(\mu)\bigg|\lesssim \frac 1L\int \min(\varepsilon,|E|^{-N}) \mathrm dN_A(E)= \mathcal O_A\left(\frac 1L\right).\end{equation}
It remains to observe that
\begin{equation}\label{remains to observe}\int_{0}^{+\infty} \varphi_L(t) \mathrm dt= \int_0^{+\infty} \varphi(u) \mathrm du+\int_0^{+\infty}(\eta(Lu)-1)\varphi(u)\mathrm du.\end{equation}
The second integral in the right-hand side is bounded by
\begin{equation}\label{eq : majoration 1/L reste} \int_0^{\frac{\ell_0}L} |\varphi(u)| \mathrm du=\mathcal O\left(\frac{1}{L}\right).\end{equation}
Recalling \eqref{LTF applied to phi}, it suffices to gather \eqref{FA contre phi}, \eqref{somme res partie plus grande que A}, \eqref{remains to observe} in order to obtain Theorem \ref{lem : equidistribution lemma}.  One can check that if the function $\varphi$ vanishes at $0$, then the $L^{-1}$ bounds in \eqref{borne f chapeau} and \eqref{eq : majoration 1/L reste} improve to $L^{-2}$. The fact that one can neglect nonprimitive geodesics without changing the result easily follows from the estimates of Lemma \ref{Poincaré Anosov} in the next section.
\end{proof}

In the next section we will have to consider \say{twisted} counting functions involving squares of characters of $\Gamma$. By tweaking the proof of \cite[Proposition 4.3]{naud2022random} to our needs, we show
\begin{prop}\label{prop : equidistribution Lie} Let $\rho:\Gamma\to \mathbf G$ be a unitary representation of $\Gamma$, where $\mathbf G$ is a compact Lie subgroup of $\mathrm{U}(N)$. Assume that $\rho(\Gamma)$ is dense in $\mathbf G$. Let $f:\mathbf G\to \mathbf C$ be a polynomial function of $\mathrm{Tr}(g)$ and $\overline{\mathrm{Tr}(g)}$, and let $\varphi:\mathbf R\to \mathbf R$ be a smooth, compactly supported function. Then,
\[\frac 1L\sum_{\gamma\in \mathcal G} f(\rho(\gamma)) \frac{\ell(\gamma)^\sharp \varphi(\ell(\gamma)/L)}{|I-\mathcal P_\gamma|}=\left(\int_{\mathbf G} f(g)\mathrm dg\right)\left(\int_{\mathbf R_+} \varphi(t)\mathrm dt\right)+\mathcal O\left(\frac 1L\right).\]
Here, $\mathrm dg$ denotes the normalized Haar measure on $\mathbf G$. \end{prop}

\begin{rem}The result still holds when $f$ is merely a continuous class function on $\mathbf G$, indeed Peter--Weyl theorem ensures that irreducibles characters of $G$ generate a dense subgroup of the space of continuous class functions on $\mathbf G$. Note however that we lose the explicit remainder $\mathcal O(L^{-1})$.\end{rem}

\begin{proof}We refer to \cite{naud2022random}. The idea is to use the Peter--Weyl theorem to expand $f(g)$ as 
\[f(g)=\sum_{\lambda} c_\lambda \operatorname{Tr}(\lambda(g)),\]
where $\lambda$ runs over the set of irreducible representations of $\mathbf G$, and the $c_\lambda$'s are complex coefficients. Naud showed that the hypothesis that $f$ is polynomial in $\operatorname{Tr}g$ and $\overline{\operatorname{Tr} g}$ implies that only a finite number of $c_\lambda$'s are nonzero. Now, let $\lambda$ be a nontrivial irreducible representation $\lambda$ of $\mathbf G$, then the representation $\lambda\circ \rho$ of $\Gamma$ does not contain the trivial representation. Indeed, assume that there is some nonzero vector $v$ such that $\lambda(\rho(\gamma))v=v$ for all $\gamma\in \Gamma$. Since $\rho(\Gamma)\subset \mathbf G$ is dense, this implies $\lambda(g)v=v$ for any $g\in \mathbf G$, contradicting the fact that $\lambda $ is irreducible. It remains to notice that
\[c_{\mathrm{trivial}} =\int_{\mathbf G} f(g)\mathrm dg, \]
and that the sum over remaining irreducible representations is finite, thus we can just apply Theorem \ref{lem : equidistribution lemma} to each of the representations $\lambda \circ \rho$ to conclude. \end{proof}

Finally, we turn to the proof of Lemma \ref{lem: size clusters}. The approach is similar to the proof Theorem \ref{lem : equidistribution lemma}, and also relies on the local trace formula.
\begin{proof}[Proof of Lemma \ref{lem: size clusters}] 
Fix $R>0$ large, since there are no other resonance than $0$ on the real axis, one may find $A=A(R)>0$ such that there are no resonances in the rectangle $[-R,R]+\mathrm{i}[-A,0]$. We apply the local trace formula with $\varepsilon=\frac A2$ to the function $\omega_T:t\mapsto \omega(t-T)$, that is supported in $ (0,+\infty)$ if $T$ is large enough. This gives
\begin{equation}\label{eq : clusters locaux}\sum_{\gamma\in \mathcal G}\frac{\ell(\gamma)^\sharp\omega(\ell(\gamma)-T)}{|I-\mathcal P_\gamma|}=\widehat \omega(0)+\sum_{\substack{\mu\in \mathrm{Res}(\mathbf P)\backslash \{0\}, \\ \operatorname{Im}\mu >-A}} \widehat \omega_T(\mu)+\langle F_A,\omega_T\rangle.\end{equation}
Since resonances have nonpositive imaginary parts and $\omega$ is compactly supported, we have $|\widehat \omega_T(\mu)|= \mathcal O(\langle\mu\rangle ^{-N})$ for any $\mu\in \mathrm{Res}(\mathbf P)$. Therefore, from the counting estimate \eqref{comptage JinTao} and the choice of $A$, we get
\[\Big|\sum_{\substack{\mu\in \mathrm{Res}(\mathbf P)\backslash \{0\}, \\ \operatorname{Im}\mu >-A}} \widehat \omega_T(\mu)\Big|\lesssim \int_{R}^{+\infty} |E|^{-N} \mathrm dN_A(E)=\mathcal O(R^{-1}).\]
Note that the implied constant is independent of $A$ because $N_A(E)\le N_1(E)$ if $A\le1$. Since $\omega$ has compact support, the estimate $\widehat F_A(\xi)=\mathcal O(\langle \xi\rangle^{7+\frac A2})$ in $\{\operatorname{Im}(\xi)\le \frac A2\}$ leads to
\[\langle F_A, \omega_T\rangle =2\pi\int_{\mathbf R+\mathrm{i} \frac A2} \widehat F_A(\xi) \widehat \omega_T(-\xi) \mathrm d\xi=\mathcal O_A(\mathrm{e}^{-\frac{AT}2}).\]
By letting $T\to +\infty$ we deduce
\[\limsup_{T\to +\infty} \Big|\sum_{\gamma\in \mathcal G}\frac{\ell(\gamma)^\sharp\omega(\ell(\gamma)-T)}{|I-\mathcal P_\gamma|}-\widehat \omega(0)\Big|=\mathcal O(R^{-1}).\]
Letting $R\to +\infty$ shows that the limit above actually vanishes, as we wished. 
\end{proof}
\section{Convergence of the ensemble variance to $\Sigma^2_{\rm GOE/GUE}$.} \label{sec : proof theorems}

\subsection{The model of random covers}\label{sec : random covers}

Let $X$ be a closed, connected, negatively curved surface. Denote by $\widetilde X$ its universal cover, and $\Gamma$ the group of deck transformations of $\widetilde X$. If $\phi_n$ is a homomorphism $\Gamma\to \mathfrak S_n$, there is a natural action of $\Gamma$ on $[n]=\{1,\ldots,n\}$ given by $\gamma\cdot i=\phi_n(\gamma)(i)$. Therefore, $\Gamma$ acts on $\widetilde X\times [n]$ through the action 
\[\gamma\cdot (\tilde x,i)=(\gamma \tilde x, \phi_n(\gamma)(i)).\]
We define a $n$-sheeted cover of $X$ by letting $X_n=\Gamma\backslash (\widetilde X\times [n])$ be the quotient of $\widetilde X\times [n]$ by this action. 

Since $\Gamma$ is finitely generated, the set $\mathrm{Hom}(\Gamma,\mathfrak S_n)$ is finite, and we can endow it with a uniform distribution of probability $\mathbf P_n$. Thus, $X_n$ can now be seen as a random variable on $\mathrm{Hom}(\Gamma,\mathfrak S_n)$. We denote by $\mathbf E_n(Y)$ the expected value of a random variable $Y$ for the probability distribution $\mathbf P_n$. As noted by Naud \cite{naud2022random}, the probability that $X_n$ is connected goes to $1$ as $n$ goes to $+\infty$.
\subsubsection{The twisted Laplacian on $X_n$.} We recall that the twisted Laplacian $\Delta_\rho$ on $X$ is the Laplacian acting on sections of the flat bundle $\mathcal V_\rho$, that we defined in \S \ref{sec : twisted laplacian}. The projection $X_n\to X$ allows lifting $\mathcal V_\rho$ to a flat bundle $\mathcal V_{\rho,n}\to X_n$, and the twisted Laplacian $\Delta_{\rho,n}$ is defined as the flat Laplacian acting on sections of $\mathcal V_{\rho,n}$. Alternatively, it can be seen as the operator $\Delta_{\widetilde X}$ acting on equivariant functions $f:\widetilde X\times [n]\to V_\rho$ satisfying $f(\gamma\cdot (\tilde x,i))=\rho(\gamma)f(\tilde x,i)$.

Since we will be dealing with the wave trace on increasingly large covers of $X$, we need to provide trace formulas for each surface $X_n$, with adequate control on the error terms. As we discussed in the introduction, there is a convenient way to perform these calculations, by showing that the Laplacian on $X_n$ is conjugated to the Laplacian acting on sections of a certain flat bundle over $X$.
\begin{prop}[{\cite[Proposition 2.1]{naud2022random}}] \label{prop:rhon}Let $\rho$ be a finite dimensional unitary representation of $\Gamma$. Then, there exists a $n\times \dim(V_\rho)$-dimensional unitary representation of $\Gamma$, denoted by $\rho_n$, such that the Laplacian $\Delta_{\rho,n}$ on $X_n$ is unitarily conjugated to the Laplacian $\Delta_{\rho_n}$ on $X$. Moreover, if we let $F_n(\gamma)$ denote the number of fixed points of $\phi_n(\gamma)$, then,
\[\operatorname{Tr}(\rho_n(\gamma))=F_n(\gamma) \operatorname{Tr}(\rho(\gamma)).\]\end{prop}

\subsubsection{Probabilistic asymptotics for $F_n(\gamma)$.}

The interest of this model lies in the following probabilistic asymptotics for the number of fixed points of $\phi_n(\gamma)$.

\begin{prop}{{\cite[Proposition 3.1]{naud2022random}}} \label{décorrélations modèle aléatoire} Denote by $F_n(\gamma)$ the number of fixed points of  $\phi_n(\gamma)$, and $\widetilde F_n(\gamma)$ the centered variable $F_n(\gamma)-\mathbf E_n(F_n(\gamma))$. Then, \begin{itemize}
	\item If $\gamma\in \mathcal P$ and $k\ge 1$,
	\[\mathbf E_n(F_n(\gamma^k))=d(k) + \mathcal O\left(\frac 1n\right),\]
	where $d(k)$ is the number of divisors of $k$ (in particular, $d(1)=1$).
	\item If $\gamma_1,\gamma_2\in \mathcal P$ and $\gamma_2\notin \{\gamma_1,\gamma_1^{-1}\}$, then, for any $k_1,k_2\ge 1$,
	\[\mathbf E_n\big(F_n(\gamma_1^{k_1})F_n(\gamma_2^{k_2})\big)=\mathbf E_n(F_n(\gamma_1^{k_1}))\mathbf E_n(F_n(\gamma_2^{k_2}))+\mathcal O\left(\frac 1n\right).\]
	\item For any $\gamma\in \mathcal P$, $k_1,k_2\ge 1$,
	\[\mathbf E_n\left[\widetilde F_n(\gamma^{k_1}) \widetilde F_n(\gamma^{k_2})\right]\underset{n\to +\infty}\longrightarrow \mathcal V(k_1,k_2),\]
	where $\mathcal V(k_1,k_2)=\displaystyle \sum_{d| \mathrm{gcd}(k_1,k_2)} d$. In particular, $\mathcal V(1,1)=1$.
\end{itemize}\end{prop}

Let us explain the geometric meaning of $F_n(\gamma)=\mathrm{Fix}(\phi_n(\gamma))$ when $\gamma$ is primitive. Take a homomorphism 
$\phi_n:\Gamma\to \mathfrak S_n$, and denote by $X_n$ the associated $n$-sheeted cover of $X$. Consider a closed primitive geodesic $\gamma:[0,1]\to X$. Let $p:X_n\to X$ denote the projection, and write $p^{-1}(\gamma(0))=\{x_1,\ldots,x_n\}$. For each $i$, we may lift $\gamma$ to a geodesic $\gamma_i:[0,1]\to X_n$ such that $\gamma_i(0)=x_i$. However, $\gamma_i$ need not be closed, and we have $\gamma_i(1)=x_j$ for some possibly different $j$. the map $\gamma_i(0)\mapsto \gamma_i(1)$ is a permutation of the fiber $p^{-1}(\gamma(0))$, that is conjugated to $\phi_n(\gamma)$. Thus, $F_n(\gamma)$ counts the number of closed primitive geodesics on $X_n$ that project on $\gamma$.

\begin{rem} Firstly, the degeneracies in the primitive length spectrum do not increase on the covers of $X$. Indeed, by the above Proposition, in the large $n$-limit, on average, a primitive closed geodesic on $X$ lifts to a unique (necessarily primitive) closed geodesic on $X_n$. 
The second point of the proposition tells us that if $\gamma_1,\gamma_2$ are two primitive closed geodesics such that $\gamma_2\notin \{\gamma_1,\gamma_1^{-1}\}$, in the large $n$ limit, the covariance of $F_n(\gamma_1)$ and $F_n(\gamma_2)$ vanishes. These asymptotic decorrelations allow performing (rigorously) a kind of diagonal approximation when computing the fluctuations of $N_n(\lambda,L)$ in the large $n$ limit. 


\end{rem}

\subsection{The large $n$ limit} With Proposition \ref{décorrélations modèle aléatoire} at hand, we are ready to compute the ensemble variance of $N_n(\lambda,L)$ in the large $n$ limit. Recalling \eqref{eq:N as wave}, we have
\[N_n(\lambda,L)=\frac 2L \big\langle \operatorname{Tr} \cos(t\sqrt{\Delta_{n,\rho}}),\cos(\lambda t)\widehat \psi(t/L)\big\rangle=\frac 2L\big\langle \operatorname{Tr} \cos(t\sqrt{\Delta_{\rho_n}}),\cos(\lambda t)\widehat \psi(t/L)\big\rangle,\]
where $\Delta_{n,\rho}$ denotes the Laplacian on $X_n$ twisted by the representation $\rho$, and $\Delta_{\rho_n}$ denotes the Laplacian on $X$ twisted by the representation $\rho_n$ introduced in Proposition \ref{prop:rhon}. Let us introduce the following notation to alleviate the writing:
\begin{equation}\label{def A(gamma,k)} A(\gamma,k)= (\chi_\rho(\gamma^k)+\overline{\chi_\rho(\gamma^k)})\int_{\mathbf R} \cos(\lambda t)\widehat \psi(t/L)\int_{\Gamma_\gamma\backslash \widetilde X} \widetilde E(t,x,\gamma^k x) \mathrm dx \mathrm dt.
\end{equation}
The characters $\chi_\rho$ can always be bounded by $|\dim V_\rho|$. Since the representation $\rho$ is fixed --- although $\rho_n$ depends on $n$ ---, this term is uniformly bounded. The notation $A(\gamma,k)$ allows for shorter expressions, but we shall use explicit asymptotics for $A(\gamma,k)$ in the regime $1\ll L\le c_0 \log \lambda$, which are the contents of Theorem \ref{trace formula for long times}. In particular, we will make extensive use of the following upper bound, valid in the regime $L\le c_0\log \lambda$,
\begin{equation}\label{eq: upper bound A(gamma,k)}|A(\gamma,k)|\lesssim \mathbf 1_{k\ell(\gamma)\le L} \frac{\ell(\gamma)}{|I-\mathcal P_\gamma^k|^{\frac 12}}.\end{equation}
We introduce the centered variable $\widetilde N_n=N_n -\mathbf E_n(N_n)$. The Weyl term in the trace formula depends only on $n$ thus vanishes when we subtract the ensemble average. Formula \eqref{eq : regroupements intégrales} with the representation $\rho_n$ gives
\[\widetilde N_n(\lambda,L)=
\frac 2L\sum_{\gamma\in \mathcal P_0}\sum_{k\ge 1} 
\widetilde F_n(\gamma^k) A(\gamma,k).\]
Here, $\mathcal P_0=\mathcal P/\{\pm 1\}$ denotes the set of \emph{non-oriented} primitive geodesics. The coefficients $A(\gamma,k)$ do not depend on $\phi_n$ neither on $n$. Thus, the variance $\Sigma_n^2(\lambda,L)=\mathbf E_n(\widetilde N_n^2)$ is given by
\[\Sigma_n^2(\lambda,L)=\frac{4}{L^2}\sum_{k_1,k_2,\gamma_1,\gamma_2} \mathbf E\big(\widetilde F_n(\gamma_1^{k_1} ) \widetilde F_n(\gamma_2^{k_2})\big)A(\gamma_1,k_1)A(\gamma_2,k_2).\]
where the sum runs over $\gamma_i\in \mathcal P_0, k_i\ge 1$.  When $n\to +\infty$, according to Proposition \ref{décorrélations modèle aléatoire}, the preceding sum converges to
\begin{equation}\label{lim Vn}\Sigma^2(\lambda,L):=\frac{4}{L^2}\sum_{\gamma\in \mathcal P_0} \sum_{k_1,k_2\ge 1} \mathcal V(k_1,k_2)A(\gamma,k_1)A(\gamma,k_2),\end{equation}
where $\gamma$ runs over the set of non-oriented primitive geodesics. 

\subsection{The contribution of nonprimitive closed geodesics} We show that the contribution of nonprimitive closed geodesics ($k_1+k_2\ge 3$) in (\ref{lim Vn}) is negligible. This is a consequence of the hyperbolicity of the flow. We first give a standard lower bound for the determinant of $I-\mathcal P_\gamma$, relying on the Anosov property.

\begin{lem}\label{Poincaré Anosov}There are constants  $C,\theta>0$ such that for any $\gamma\in \mathcal P$ and $k\ge 1$,
\[|I-\mathcal P_{\gamma^k}|\ge C \mathrm{e}^{\theta (k-1) \ell(\gamma)}|I-\mathcal P_\gamma|.\]\end{lem}
\begin{proof} Take some point $z\in M$ lying on the periodic orbit $\gamma$. Since\footnote{strictly speaking, $\mathcal P_\gamma$ depends on the choice of $z$, but its determinant does not.} $\mathcal P_\gamma:=\mathrm d\phi^{\ell(\gamma)}$ preserves the direct sum $E_s(z)\oplus E_u(z)$, we can write in an adapted basis to this decomposition, recalling that $\mathcal P_\gamma$ has determinant $1$,
\[\mathcal P_\gamma=\begin{pmatrix}a^{-1} & 0 \\ 0 & a\end{pmatrix}.\]
We claim that $|a|\ge \mathrm{e}^{\theta \ell(\gamma)}$, where $\theta$ is as in \S\ref{defi:Anosov}. To see this, take a nonzero vector $v\in E_s$. Then, for any $k\ge 1$, we have $\mathrm d\phi^{k\ell(\gamma)} v=a^{-k} v$. Also, by the Anosov property, we have $|\mathrm d\phi^{k\ell(\gamma)}v|\le C\mathrm{e}^{-\theta k \ell(\gamma)}$. Therefore $|a|\ge C^{\frac 1k}\mathrm{e}^{\theta\ell(\gamma)}$, which proves the claim by letting $k$ go to $+\infty$. The desired lower bound for $|I-\mathcal P_{\gamma^k}|$ then follows by an easy computation.
\end{proof}
Fix $k_1,k_2$ with $k_1+k_2\ge 3$ and set
\[S_{k_1,k_2}=\sum_{\gamma\in \mathcal P_0} A(\gamma,k_1)A(\gamma,k_2).\]
By \eqref{eq: upper bound A(gamma,k)} and Lemma \ref{Poincaré Anosov}, we have the upper bound 
\[|S_{k_1,k_2}|\lesssim \sum_{\ell(\gamma)\le L} \frac{\ell(\gamma)^2}{|I-\mathcal P_\gamma|} \mathrm{e}^{-\frac{\theta}{2}(k_1+k_2-2) \ell(\gamma)}. \]
The argument in the exponential is nonzero (and negative) for $k_1+k_2\ge 3$. Hence, in order to bound $S_{k_1,k_2}$, it is enough to obtain upper bounds for sums of the type
\[\sum_{\ell(\gamma)\le L} \frac{\ell(\gamma)^N\mathrm{e}^{-s \ell(\gamma)}}{|I-\mathcal P_\gamma|}.\]
\begin{lem}\label{lem:sommes exponentielles}Let $s_0>0$ and $N\in \mathbf N$. Then for any $s\ge s_0$,
\begin{equation}\label{somme est finie pour tout s} \sum_{\gamma\in \mathcal G} \frac{\ell(\gamma)^N\mathrm{e}^{-s \ell(\gamma)}}{|I-\mathcal P_\gamma|} =\mathcal O_{N,s_0}\big(\mathrm{e}^{-\ell_0 s}\big),\end{equation}
where $\ell_0$ is the length of the shortest closed geodesic. 
\end{lem}
\begin{proof}An immediate consequence of Theorem \ref{lem : equidistribution lemma} is that $\sum_{\ell(\gamma)\le L} \frac{1}{|I-\mathcal P_\gamma|}\lesssim L$. Thus, we may just write
\[\sum_{\gamma\in \mathcal G} \frac{ \ell(\gamma)^N \mathrm{e}^{-s \ell(\gamma)}}{|I-\mathcal P_\gamma|}\le \sum_{k=1}^{\infty} (k+1)^N\sum_{k\ell_0\le \ell(\gamma)\le (k+1)\ell_0} \frac{\mathrm{e}^{-ks \ell_0}}{|I-\mathcal P_\gamma|}\lesssim \sum_{k=1}^{\infty} (k+1)^{N+1} \mathrm{e}^{-ks \ell_0}\lesssim_{N,s_0} \mathrm{e}^{-s \ell_0}.\]\end{proof}
This allows to obtain the sought bound on $S_{k_1,k_2}$:
\begin{lem}For $k_1+k_2\ge 3$ we have $S_{k_1,k_2}=\mathcal O(\mathrm{e}^{-\frac{\theta}{2}(k_1+k_2-2)\ell_0})$, where $\ell_0$ is the length of the shortest closed geodesic.\end{lem}
Summing over all pairs $(k_1,k_2)$ with $k_1+k_2\ge 3$ yields
\begin{prop}[Contribution of nonprimitive orbits] \label{prop:contribution non primitive} The following estimate holds, independently of $L$.
\[ \sum_{k_1+k_2\ge 3} \mathcal V(k_1,k_2) S_{k_1,k_2} = \mathcal O(1).\]\end{prop}

\begin{proof}We roughly bound $\mathcal V(k_1,k_2)\lesssim \mathrm{gcd}(k_1,k_2)^2\le (k_1+k_2)^2$. Then, for $N$ large enough, since $S_{k_1,k_2}=\mathcal O(\mathrm{e}^{- \frac \theta 2(k_1+k_2-2) \ell_0})$, 
\[\Big|\sum_{k_1+k_2\ge 3} \mathcal V(k_1,k_2) S_{k_1,k_2}\Big|\lesssim \sum_{k_1+k_2\ge 3} \frac{1}{(k_1+k_2)^N}=\mathcal O(1)\]
uniformly in $\lambda,L$.

\end{proof}

By Proposition \ref{prop:contribution non primitive}, we can discard pairs $(k_1,k_2)$ with $k_1+k_2\ge 3$ in \eqref{lim Vn} --- up to adding an error $\mathcal O(L^{-2})$ ---, giving 
\[\Sigma^2(\lambda,L)=\frac{4}{L^2}\sum_{\gamma\in \mathcal P_0} A(\gamma,1)^2+\mathcal O\left(\frac{1}{L^2}\right).\]
Now, Theorem \ref{trace formula for long times} gives
\[A(\gamma,1)^2=\frac{\cos^2(\lambda \ell(\gamma))\ell(\gamma)^2\widehat \psi^2(\ell(\gamma)/L)}{|I-\mathcal P_\gamma|}(\chi_\rho(\gamma)+\overline{\chi_\rho(\gamma)})^2 +\mathcal O\left(\lambda^{-\frac 12}\mathrm{e}^{CL}\right).\] 
Writing $\cos^2(x)=\frac{1+\cos(2x)}{2}$, we can split the sum
\begin{equation}\label{eq sigma^2 avec osc}\Sigma^2(\lambda,L)=\overline{\Sigma^2}(\lambda,L)+\Sigma^2_{\rm osc}(\lambda,L)+\mathcal O\left(\lambda^{-\frac 12}\mathrm{e}^{CL}\right),\end{equation}
where the constant $C$ has increased, and we have set 
\[\overline{\Sigma^2}(\lambda,L)=\frac{2}{L^2}\sum_{\gamma\in \mathcal P_0} (\chi_\rho(\gamma)+\overline{\chi_\rho(\gamma)})^2\frac{ \ell(\gamma)^2 \widehat \psi^2(\ell(\gamma)/L)}{|I-\mathcal P_\gamma|}\]
and
\[\Sigma^2_{\rm osc}(\lambda,L):=\frac{2}{L^2}\sum_{\gamma\in \mathcal P_0} (\chi_\rho(\gamma)+\overline{\chi_\rho(\gamma)})^2\frac{\ell(\gamma)^2 \widehat \psi^2(\ell(\gamma)/L)\cos(2\lambda \ell(\gamma))}{|I-\mathcal P_\gamma|}.\]
The term $\overline{\Sigma^2}(\lambda,L)$ does not depend on $\lambda$ anymore, while $\Sigma_{\rm osc}^2(\lambda,L)$ should be negligible when averaged over sufficiently large intervals of frequencies, due to the oscillations.

In the following, we stick to the case $\dim V_\rho=1$. We will then explain how to deal with representations of higher rank. The nonnegativity of $\Sigma^2(\lambda,L)$ is obvious by definition. 

If $\rho$ is Abelian, we may expand $(\rho(\gamma)+\overline{\rho(\gamma)})^2=2+\rho^2(\gamma)+\rho^2(\gamma^{-1})$, where $\rho^2$ is still an Abelian character of $\Gamma$. It follows that
\[\overline{\Sigma^2}(\lambda,L)=\frac{2}{L^2}\sum_{\gamma\in \mathcal P} (1+\rho^2(\gamma))\frac{\ell(\gamma)^2 \widehat \psi^2(\ell(\gamma)/L)}{|I-\mathcal P_\gamma|}\]
where now the sum runs over \emph{oriented} primitive geodesics. Applying Theorem \ref{lem : equidistribution lemma} to $t\mapsto t\widehat \psi^2$ gives
\begin{equation}\label{eq : partie constante tend vers GOE} \overline{\Sigma^2}(\lambda,L)=2(1+\delta_{\mathbf 1,\rho^2})\int_{\mathbf{R}_+} t\psi^2(t)\mathrm dt+\mathcal O\left(L^{-2}\right),\end{equation} 
where $\delta_{\mathbf 1,\rho^2}=1$ if $\rho^2=\mathbf 1$, and $0$ otherwise.
Therefore,
\begin{equation}\label{eq: GOE + osc +1/L}\Sigma^2(\lambda,L)=2(1+\delta_{\mathbf 1,\rho^2})\int_{\mathbf R_+} t\widehat \psi^2(t)\mathrm dt+\Sigma_{\rm osc}^2(\lambda,L)+\mathcal O\left(\lambda^{-\frac 12}\mathrm{e}^{CL}+L^{-1}\right).\end{equation}
The first term of the right hand side equals $\Sigma_{\mathrm{GOE/GUE}}^2$ depending on whether $\rho$ breaks time reversal symmetry. 

In the following we take a function $L=L(\lambda)$ satisfying $L(\lambda)\le c_0\log \lambda$ for some sufficiently small $c_0$, in order to ensure that the error in \eqref{eq: GOE + osc +1/L} is indeed negligible, say $\mathcal O(L^{-2})$. The proof of Theorem \ref{mainthm} then consists in showing that local averages of $\Sigma^2_{\rm osc}(\lambda,L)$ converge to $0$.

\subsection{Oscillations of $\Sigma_{\rm osc}^2$} We start by showing item \emph{(ii)} of Theorem \ref{mainthm}. First, note that $|\Sigma^2_{\rm osc}(\lambda,L)|\le \overline{\Sigma^2}(\lambda,L)$, which proves the estimate $0\le \Sigma^2(\lambda,L)\le 2\Sigma^2_{\rm GOE/GUE}+\mathcal O(L^{-2})$. Now, we show that $\Sigma_{\rm osc}^2(\lambda,L)$ oscillate between the two extreme values $\pm \Sigma_{\rm GOE/GUE}^2$, which justifies the necessity to perform energy averages.

\begin{prop}\label{prop: pointwise behavior Sigma osc}For any sequence $L_j\to +\infty$, we can find sequences $(\lambda_j^{\pm})$ satisfying $L_j\le c_0\log (\lambda_j^{\pm})$ and $ \log \log \lambda_j^{\pm}\lesssim  L_j$, and such that 
\[\underset{j\to +\infty}\lim \Sigma^2_{\rm osc}(\lambda_j^{\pm},L_j,\rho)= \pm\Sigma_{\rm GOE/GUE}^2.\]\end{prop}
It is a quite straightforward consequence of the following  Dirichlet box principle, see \cite[\S 4.3]{JakobsonWeyl2007}.
\begin{lem}Let $Y,M>0$. Let $r_1,\ldots,r_N$ be some distinct real numbers. Then one can find $\lambda\in [M,MY^N]$ such that
\[|\mathrm{e}^{\mathrm i \lambda r_j}-1|\le \frac 1Y, \qquad j=1,\ldots,N.\]\end{lem}
\begin{proof}[Proof of Proposition \ref{prop: pointwise behavior Sigma osc}]  Fix $C$ large enough. For any $L>0$, according to the previous lemma with $Y=L$ and $M=\mathrm{e}^{CL}$, we can find $\lambda\in [\mathrm{e}^{CL}, \mathrm{e}^{CL} L^{\mathrm{e}^{{C L}}}]$ such that for any closed geodesic with period $\ell(\gamma)\le L$ we have  $\cos(\lambda\ell(\gamma))\ge 1-\frac{1}{L}$. Here one must take $C$ large enough so as to ensure $L\le c_0 \log \lambda$ (in order to have negligible remainders in the trace formula), and that there are less than $\mathrm{e}^{CL}$ periodic geodesics of length $\le L$. Consequently,
\[\Sigma_{\rm osc}^2(\lambda,L,\rho)= \frac{2}{L^2}\sum_{\gamma\in \mathcal P_0} (\chi_\rho(\gamma)+\overline{\chi_\rho(\gamma)})^2\frac{\ell(\gamma)^2 \widehat \psi^2(\ell(\gamma)/L)}{|I-\mathcal P_\gamma|} +\mathcal O\left(\frac{1}{L}\right)= \Sigma^2_{\rm GOE/GUE}+\mathcal O\left(\frac{1}{L}\right).\]
Thus, we can find a sequence $(L_j,\lambda_j^+)$ such that $\Sigma^2_{\rm osc}(\lambda_j^+,L_j,\rho)\to \Sigma_{\rm GOE/GUE}^2$. For the opposite sign, one has to take $\lambda$ such that $|\cos(\lambda\ell(\gamma))|\le \frac 1L$ for any $\gamma$ with length $\ell(\gamma)\le L$. \end{proof}
\subsection{Convergence on average of $\Sigma^2(\lambda,L)$} We now prove item \emph{(iii)} of Theorem \ref{mainthm}. It comes down to showing that local averages of $\Sigma^2_{\rm osc}$ converge to $0$. By averaging $\Sigma_{\rm osc}^2(\mu,L(\lambda))$ over a window of frequencies $[\lambda,\lambda+\delta]$, of width $\delta$ large compared to $\frac 1L$ we get
\begin{equation}\label{equation moyennée} \displaystyle\frac{1}{\delta}\int_\lambda^{\lambda+\delta} \Sigma^2_{\rm osc}(\mu,L(\lambda))\mathrm d\mu  =\displaystyle\mathcal{O}\left(\frac{1}{\delta L}\right) \left|\frac 1L\sum_{\gamma\in \mathcal P_0}\frac{\ell(\gamma)\widehat \psi^2(\ell(\gamma)/L)}{|I-\mathcal P_\gamma|}\right|.
\end{equation}
Indeed, just check that
\[\frac{1}{\delta}\int_{\lambda}^{\lambda+\delta} \cos(2\mu\ell(\gamma))\mathrm d\mu=\frac{L}{\ell(\gamma)}\mathcal O\left( \frac{1}{\delta L}\right).\]
Here we stress out that we are working with $\Sigma^2(\mu,L(\lambda))$ and not $\Sigma^2(\mu,L(\mu))$. If one wants to work with $\Sigma^2(\mu,L(\mu))$, some further assumptions need to be made on the function $L$, namely that it has small variations over the window $[\lambda,\lambda+\delta]$. 
Now, by Theorem \ref{lem : equidistribution lemma},
\[\frac 1L\sum_{\gamma\in \mathcal P_0}\frac{\ell(\gamma)\widehat \psi^2(\ell(\gamma)/L)}{|I-\mathcal P_\gamma|}=\mathcal O(1)\]
uniformly in $L$. Therefore, recalling \eqref{eq: GOE + osc +1/L}, we have shown
\[\frac{1}{\delta}\int_\lambda^{\lambda+\delta} \Sigma^2(\mu,L)\mathrm d\mu =\Sigma_{\mathrm{GOE/GUE}}^2+\mathcal O\left(\frac{1}{\delta L} +\frac{1}{L^2}\right).\]

\subsection{Quadratic convergence on average to $\Sigma_{\mathrm{GOE}}^2$}\label{subsec:quadratic on average} Let us finally turn to the proof of item \emph{(iv)} of Theorem \ref{mainthm}. Averaging the variance over windows of large size allows to get quadratic convergence. We perform an average over the energy range $[\lambda,\lambda+\Lambda]$:
\[\frac{1}{\Lambda}\int_{\lambda}^{\lambda+\Lambda}|\Sigma^2_{\rm osc}(\mu,L)|^2\mathrm d\mu= \frac{1}{\Lambda}\int_{\lambda}^{\lambda+\Lambda} \left|\frac{2}{L^2}\sum_{\gamma\in \mathcal P_0} (\chi_\rho(\gamma)+\overline{\chi_\rho(\gamma)})^2\frac{\ell(\gamma)^2 \widehat \psi^2(\ell(\gamma)/L)\cos(2\mu \ell(\gamma))}{|I-\mathcal P_\gamma|}\right|^2 \mathrm d\mu.\]
We expand the square in the right hand-side and only keep pairs of closed geodesics whose lengths are close, since other pairs give a negligible contribution.  Indeed, if $|\ell(\gamma)-\ell(\gamma')|>1$ then,
\[\frac{1}{\Lambda}\int_{\lambda}^{\lambda +\Lambda} \cos(\mu\ell(\gamma))\cos(\mu\ell(\gamma'))\mathrm d\mu = \mathcal O\left(\Lambda^{-1}\right),\]
thus the sum corresponding to pairs with $|\ell(\gamma)-\ell(\gamma')|>1$ can just be bounded by
\[\Lambda^{-1}\left|\frac{2}{L^2}\sum_{\gamma\in \mathcal P_0} (\chi_\rho(\gamma)+\overline{\chi_\rho(\gamma)})^2\frac{\ell(\gamma)^2 \widehat \psi^2(\ell(\gamma)/L)}{|I-\mathcal P_\gamma|}\right|^2=\mathcal O(\Lambda^{-1}).\]
The sum corresponding to pairs $(\gamma,\gamma')$ with $|\ell(\gamma)-\ell(\gamma')|\le 1$ is bounded by
\begin{equation}\label{eq : convergence L2 intermède}\frac{1}{L^4}\sum_{|\ell(\gamma)-\ell(\gamma')|<1} \frac{\ell(\gamma)^2 \widehat \psi^2(\ell(\gamma)/L)}{|I-\mathcal P_\gamma|}\frac{\ell(\gamma')^2 \widehat \psi^2(\ell(\gamma')/L)}{|I-\mathcal P_\gamma'|}.\end{equation}
We are going to show that it converges to $0$ as $L$ gets large. 

We apply Lemma \ref{lem: size clusters} with $T=\ell(\gamma)$ for any $\gamma$ with length $\le L$ to bound sums over $\gamma'$ in \eqref{eq : convergence L2 intermède}:
\[\sum_{\gamma'\in \mathcal P_0, \ |\ell(\gamma)-\ell(\gamma')|<1} \frac{\ell(\gamma')^2 \widehat \psi^2(\ell(\gamma')/L)}{|I-\mathcal P_\gamma'|}\lesssim \ell(\gamma).\]
Thanks to the equidistribution Theorem \ref{lem : equidistribution lemma}, the sum in \eqref{eq : convergence L2 intermède} is in turn bounded by a constant times
\[\frac{1}{L^4}\sum_{\gamma\in \mathcal P_0} \frac{\ell(\gamma)^3\widehat \psi^2(\ell(\gamma)/L)}{|I-\mathcal P_\gamma|}\lesssim \frac{1}{L}\int t^2\widehat \psi^2(t)\mathrm dt.\] Combining all our estimates yields
\[\frac{1}{\Lambda}\int_{\lambda}^{\lambda+\Lambda}|\Sigma^2(\mu,L)-\Sigma^2_{\mathrm{GOE/GUE}}|^2\mathrm d\mu=\mathcal O\left(\frac{1}{L}+\frac{1}{\Lambda}\right).\]

\subsection{Convergence along a subset of density 1} We turn to the proof of Corollary \ref{cor : density}. The \emph{ad hoc} assumption on $L$ ensures that $\Sigma^2(\mu,\cdot)$ has negligible variations over the interval $[L(\lambda),L(\lambda+\Lambda)]$ for some good choice of $\Lambda$, allowing to replace $\Sigma^2(\mu,L(\lambda))$ by $\Sigma^2(\mu,L(\mu))$ in \eqref{eqthmquad}. We will construct an adequate function $\Lambda=\Lambda(\lambda)$ going to $+\infty$ then apply the item \emph{(iv)} of Theorem \ref{mainthm}. First write, for $\mu \in [\lambda,\lambda+\Lambda]$,
\[\Sigma^2_{\rm osc}(\mu,L(\lambda))-\Sigma^2_{\rm osc}(\mu,L(\mu))=\int_{L(\lambda)}^{L(\mu)}\frac{\partial \Sigma^2_{\rm osc}(\mu,L)}{\partial L} \mathrm dL.\]
The partial derivative of $\Sigma^2(\mu,L)$ with respect to $L$ is bounded by
\begin{equation}\label{eq:derivéesigma} \Big|\frac{\partial \Sigma^2_{\rm osc}(\mu,L)}{\partial L}\Big| \lesssim \frac{1}{L^4}\sum_{\ell(\gamma)\le L} \frac{\ell(\gamma)^3}{|I-\mathcal P_\gamma|}\lesssim \frac 1L.\end{equation}
Setting $f(\lambda):=\log(L(\lambda))$, the second assumption on $L$ reads $f'(\lambda)\ll 1$ as $\lambda\to +\infty$, and \eqref{eq:derivéesigma} integrates to
\begin{equation}\label{eq:variations Sigma}\Big|\Sigma^2_{\rm osc}\big(\mu,L(\lambda)\big)-\Sigma^2_{\rm osc}\big(\mu,L(\mu)\big)\Big|\lesssim \Big| \int_{L(\lambda)}^{L(\mu)} \frac{\mathrm dL}{L}\Big| \le \big|f(\lambda)-f(\mu)\big|\le \Lambda \sup_{[\lambda,+\infty]} |f'|, \end{equation}
We set
\[\Lambda(\lambda):=\min \big(\lambda,\big(\sup_{\mu \ge \lambda} |f'(\mu)| \big)^{-\frac 12}\big).\]
By construction, $1\ll \Lambda$ and $\Lambda\sup_{[\lambda,+\infty]} |f'|\ll 1$. Combining \eqref{eqthmquad} and \eqref{eq:variations Sigma} yields
\begin{equation}\label{L depend de mu mtnt}\frac{1}{\Lambda}\int_{\lambda}^{\lambda+\Lambda} \big|\Sigma^2_{\rm osc}(\mu,L(\mu))\big|^2 \mathrm d\mu \underset{\lambda\to +\infty}\longrightarrow 0.\end{equation}
Corollary \ref{cor : density} is now a standard consequence of \eqref{L depend de mu mtnt}.

\subsection{Proof of Theorem \ref{mainthm2}} The proof is the same as that of Theorem \ref{mainthm}, although one has to invoke Proposition \ref{prop : equidistribution Lie} with $f(g)=(\operatorname{Tr}(g)+\overline{\operatorname{Tr}(g)})^2$, instead of Theorem \ref{lem : equidistribution lemma}. We refer to \cite[\S 5.4]{naud2022random} for a proof of the analogous result in constant curvature. Naud showed that if $\mathbf G$ is not conjugated to a nontrivial product of Lie groups, then
\[\int_{\mathbf G} f(g)\mathrm dg\in \{2,4\}.\]
More precisely, the integral is equal to $4$ if $g\in \mathbf G \mapsto \operatorname{Tr}(g)$ is real-valued, in which case we get $\Sigma^2_{\mathbf G}=\Sigma^2_{\rm GOE}$. This includes the cases $\mathbf G=\mathrm{SO}(N), \mathrm{Sp}(N)$. Otherwise, the integral equals $2$ and $\Sigma^2_{\mathbf G}=\Sigma^2_{\rm GUE}$: this covers the cases $\mathbf G = \mathrm{U}(N), \mathrm{SU}(N)$ for $N\ge 3$. We point out that the computation of the integral above for $\mathbf G=\rm{SU}(2)$ was performed explicitly in \cite{Bolte1999}.

\section{Smooth transition from GOE to GUE} \label{sec:prooftransition} Now we explain how to observe an intermediate behavior by letting the representation $\rho$ depend on $\lambda$. We recall the setting of Theorem \ref{thm:smooth transition}: $\mathbf A$ is a smooth $1$-form on $X$ and we let 
\[\rho_\alpha(\gamma):=\exp\left(-\mathrm i\alpha \int_{\mathbf \gamma} \mathbf A\right).\]
Take three functions $\alpha =\alpha(\lambda)$, $L=L(\lambda)$ and $\delta(\lambda)$ such that $\alpha\sqrt L$ converges to $s\in \mathbf R_+$ as $\lambda\to +\infty$ and, as before, $1\ll L\le c_0\log \lambda$ and $\delta\gg \frac 1L$. According to the computations of the previous section, we have
\[\frac{1}{\delta}\int_\lambda^{\lambda+\delta} \Sigma^2(\rho_\alpha,\mu,L)\mathrm d\mu=\Sigma^2_{\rm GUE}+\frac{2}{L^2}\sum_{\gamma\in \mathcal P} \rho_\alpha(\gamma)^2 \frac{\ell(\gamma)^{\sharp}\ell(\gamma)\widehat \psi^2(\ell(\gamma)/L)}{|I-\mathcal P_\gamma|}+\mathcal O\left(\frac{1}{L^2}\right).\]
Our goal is to apply the local trace formula for the Ruelle operator acting on sections of the bundle $\mathcal E_{\rho_\alpha^2}$. As we already explained, this operator is conjugated to the operator $\mathbf P +2\alpha \langle \mathbf A(x),\xi\rangle_{T_x^*X}$ acting on smooth functions on $M$. It shall indeed be more practical to let the operator vary instead of the space. This requires the two following points:
\begin{itemize}
\item Understanding the behavior of the first resonance of the perturbed Ruelle operator \[\mathbf P_\alpha:=\mathbf P+\alpha \langle \mathbf A(x),\xi\rangle.\]
\item Providing uniform bounds on the remainder in the local trace formula with respect to $\alpha$.
\end{itemize}
The first point is the content of the well-known formula provided below. We state it in a rather general setting, although we will specifically use it for $M=S^*X$ and $a=\langle \mathbf A(x),\xi\rangle$.

\begin{prop}\label{prop: res perturbation 0} Let $\phi^t=\mathrm{e}^{tV}$ be a transitive, weak mixing, Anosov flow on a closed Riemannian manifold $M$. Denote by $\nu_0$ the SRB measure on $M$ (it is a generalization of the normalized Liouville measure to arbitrary Anosov flows), that may be defined as the weak limit
\[\nu_0:=\lim_{T\to +\infty} \sum_{T\le \ell(\gamma)\le T+1} \frac{\delta_\gamma }{|I-\mathcal P_\gamma|}, \qquad \langle \delta_\gamma,f\rangle:=\int_\gamma f.\]
Note that $\nu_0$ is not absolutely continuous with respect to the Lebesgue measure in general. 

Let $a\in C^\infty(M)$. There is a neighborhood $U$ of $0$ such that if $\alpha$ is sufficiently small, then $\mathbf P+\alpha a$ has a unique resonance in $U$, denoted by $\mu_\alpha$, that depends smoothly on $\alpha$. Additionally, we have the Taylor expansion 
\[\mu_\alpha=-\frac 12\mathrm i\alpha^2\mathrm{Var}_{\nu_0}(a)+\mathcal O(\alpha^3),\] where the variance of $a$ has been defined in \S \ref{subsec:transition}.\end{prop} 
Now, a local trace formula with uniform error terms follows from a careful study of the arguments in \cite{Jin2016,jin2023number}. We postpone this discussion to the Appendix, and move on to the proof of Theorem \ref{thm:smooth transition}. 

\begin{proof}[Proof of Theorem \ref{thm:smooth transition}] This is similar to the proof of Theorem \ref{lem : equidistribution lemma}, but here we apply the local trace formula to the perturbed operator $\mathbf P +2\alpha \langle \mathbf A(x),\xi\rangle$, with test function $t\mapsto t\widehat \psi^2(t)$. As above, let $a(x,\xi):=\langle \mathbf A(x),\xi\rangle_{T_x^*X}$, and denote by $\mu_\alpha$ the first resonance of $\mathbf P+\alpha a$. Then,
\[\frac{2}{L^2}\sum_{\gamma\in \mathcal P_0} \rho_\alpha(\gamma)^2 \frac{\ell(\gamma)^{\sharp}\ell(\gamma)\widehat \psi^2(\ell(\gamma)/L)}{|I-\mathcal P_\gamma|}=2\int_{\mathbf R_+} \mathrm{e}^{-\mathrm i \mu_{2\alpha}L t} t\widehat \psi^2(t)\mathrm dt+\mathcal O\left(\frac{1}{L^2}\right).\]
Now as $\lambda \to +\infty$ we have $\mu_{2\alpha} L \sim -\mathrm i 2\alpha^2\mathrm{Var}_{\nu_0}(a) L\longrightarrow -\mathrm i 2 \mathrm{Var}_{\nu_0}(a)s^2$, thus by dominated convergence,
\[\underset{\lambda \to +\infty}\lim\int_{\mathbf R_+} \mathrm{e}^{-\mathrm i \mu_{2\alpha}L t} t\widehat \psi^2(t)\mathrm dt=2\int_{\mathbf R_+} \mathrm{e}^{-2\mathrm{Var}_{\nu_{\scalebox{.7}{$\scriptscriptstyle 0$}}}(a)s^2 t}t\widehat \psi^2(t)\mathrm dt.\]
This concludes the proof.\end{proof}

\subsection{Central limit theorem for closed orbits}
From the local trace formula for the perturbed operator $\mathbf P_\alpha$, we can derive a central limit theorem for periodic orbits of transitive, weak mixing Anosov flows, when picking periodic orbits at random in $\{\gamma, T\le \ell(\gamma)\le T+1\}$, with weights proportional to $\frac{1}{|I-\mathcal P_\gamma|}$. 

\begin{prop}[Central limit theorem for periodic orbits]\label{prop: central limit theorem for closed orbits} Let $\phi^t:M\to M$ be a transitive Anosov flow, assumed to be weak mixing with respect to the SRB measure $\nu_0$, meaning equivalently that $0$ is the only resonance on the real axis. Let $\omega$ be a smooth, nonzero, nonnegative, compactly supported function. For $T>0$ large enough, endow the set of periodic orbits $\mathcal G$ with the probability distribution $\mathbf P_T$ given by $\mathbf P_T(\gamma)=C_T \frac{\ell(\gamma)^\sharp \omega(\ell(\gamma)-T)}{|I-\mathcal P_\gamma|}$, where $C_T$ is a normalization factor.
Let $a\in C^\infty(M,\mathbf R)$ and assume that $a$ has zero mean (with respect to $\nu_0$). Define a random variable $X_T$ by $X_T(\gamma)= \frac{1}{\sqrt T}\int_\gamma a$. Then, the sequence $(X_T,\mathbf P_T)$ converges in distribution to a centered Gaussian random variable, with variance $\mathrm{Var}_{\nu_0}(a)$. \end{prop}

\begin{rem} \begin{itemize}
\item Cantrell--Sharp obtained a similar result when picking orbits uniformly in $\{\gamma\in \mathcal G, \ T\le \ell(\gamma)\le T+1\}$, with $\mathbf P_T(\gamma)=C_T \mathbf 1_{\ell(\gamma)\in [T,T+1]}$ \cite{Cantrell21}. Their dynamical assumption is a bit stronger --- namely, they assume that the stable and unstable foliations are not jointly integrable, implying in particular that the flow is weak mixing ---, but they are able to take a sharp function $\omega=\mathbf 1_{[0,1]}$, and they allow $a$ being merely Hölder continuous. Note that Lalley \cite[Theorem 6]{LALLEY1987154} obtained similar results, for smooth $a$ and under a slightly stronger dynamical assumption.
\item Using the recent results of \cite{humbert2024criticalaxisruelleresonances}, we expect our method to extend seamlessly to the more general case where one picks closed orbits at random with weights proportional to $\frac{\mathrm{e}^{\int_\gamma w}}{|I-\mathcal P_\gamma|}$, for some smooth real potential $w$. Since this is not the main point of the paper, we chose to stick to the case with no potential. Note that this would still not allow retrieving the result of Cantrell--Sharp because one would have to let $w$ be the SRB potential $\left.\frac{\mathrm d}{\mathrm dt}\right|_{t=0} \log \big|\det \mathrm d\phi^t_{|E_u}\big|$, which is not smooth in general. 
\end{itemize}
\end{rem}

\begin{proof}[Proof of Proposition \ref{prop: central limit theorem for closed orbits}] The proof is similar to that of Lemma \ref{lem: size clusters}. We show that the characteristic functions of $X_T$ converge pointwisely to that of a Gaussian as $T\to +\infty$.

Let $s>0$ and $R>0$. Let $\alpha=\frac{s}{\sqrt T}$.  Since resonances are isolated and there are no resonances other than $0$ on the real axis, we can find some small $A>0$, depending on $R$, such that $0$ is the only resonance of $\mathbf P$ in the domain $[-2R,2R]+\mathrm{i}[-2A,0]$. In turn, for $T$ large enough (depending on $R$), $\mathbf P_\alpha$ has no resonance other than $\mu_\alpha$ in $[-R,R]+\mathrm{i}[-A,0]$. The local trace formula for $\mathbf P_\alpha$, applied to the test function $\omega_T:=\omega(\cdot -T)$ then yields 
\[\begin{array}{ll}C_T^{-1}\mathbf E[\mathrm{e}^{-\mathrm{i} s X_T}]& =\displaystyle \int_{\mathbf R} \mathrm{e}^{-\mathrm{i} \mu_\alpha t}\omega(t-T)\mathrm dt +\sum_{\mu\in \mathrm{Res}(\mathbf P_\alpha)\backslash \{0\}, \mathrm{Im}(\mu)>-A} \widehat \omega_T(\mu)+\langle F_{\alpha,A}, \omega_T\rangle\\ & =:\mathrm{(I)+(II)+(III)}.\end{array}\]
Let us evaluate each term.

\textit{I.} After a change of variables in the integral, we have
\[\mathrm{(I)}=\mathrm{e}^{-\mathrm{i} \mu_\alpha T}\int_{\mathbf R} \mathrm{e}^{-\mathrm{i}\mu_\alpha t}\omega(t)\mathrm dt.\]
By Proposition \ref{prop: res perturbation 0}, we have $\mu_\alpha=-\frac 12 \mathrm{Var}_\nu(a) \alpha^2 +\mathcal O(\alpha^3)$. Recalling that $\alpha=\frac{s}{\sqrt T}$, it follows that
\[\mathrm{(I)}=\mathrm{e}^{-\frac{s^2}{2}\mathrm{Var}_\nu(a)}\int_{\mathbf R} \omega(t)\mathrm dt +\mathcal O(T^{-\frac 12}).\]

\textit{II and III.} These are handled as in the proof of Lemma \ref{lem: size clusters}, and we get
\[\mathrm{(II)}=\mathcal O(R^{-1}), \qquad \mathrm{(III)}=\mathcal O_{A}(\mathrm{e}^{-\frac A2 T}).\]
To sum up, we have
\[C_T^{-1}\mathbf E_{T}(\mathrm{e}^{-\mathrm{i} s X_T})=\mathrm{e}^{-\frac 12 s^2\mathrm{Var}_\nu(a)}\int_{\mathbf R} \omega(t)\mathrm dt +\mathcal O(T^{-\frac 12})+\mathcal O_{A}(e^{-\frac A2 T})+ \mathcal O(R^{-1}).\]
Note that Lemma \ref{lem: size clusters} exactly states that $C_T^{-1}$ converges to $\int_{\mathbf R} \omega$ as $T\to +\infty$, therefore
\[\limsup_{T\to +\infty} \big| \mathbf E_{T}(\mathrm{e}^{-\mathrm{i} s X_T})-\mathrm{e}^{-\frac 12s^2 \mathrm{Var}_{\nu_{\scalebox{.7}{$\scriptscriptstyle 0$}}}(a)}\big|=\mathcal O(R^{-1}).\]
Since this holds for any $R>0$ we conclude
\[\mathbf E_{T}(\mathrm{e}^{\mathrm{i} s X_T})\underset{T\to +\infty}\longrightarrow \mathrm{e}^{-\frac 12s^2 \mathrm{Var}_{\nu_{\scalebox{.7}{$\scriptscriptstyle 0$}}}(a)}. \]
The pointwise convergence of characteristic functions of $(X_T)$ to that of a Gaussian implies convergence in distribution.\end{proof}

Unfortunately, convergence in distribution does not readily give convergence of moments. Still, when dealing with geodesic flows in negative curvature, we are able to show convergence of the variance $\mathbf E[X_T^2]$ and find an expression of $\mathrm{Var}_{\nu_0}(a)$ involving the integrals of $a$ along closed geodesics. This is discussed in Proposition \ref{prop : appendice variance et flux}. 

\section{The Central Limit Theorem for random covers}\label{sec:CLT}
This section is devoted to the proof of Theorem \ref{thm : CLT}. The proof goes in two steps. First, using a result of Maoz \cite{Maoz1}, we show that for fixed $(\lambda,L)$, as $n$ goes to $+\infty$, the sequence of random variables $(\widetilde N_n(\lambda,L))$ converges in distribution to an explicit random variable $\widetilde N_{\infty}(\lambda,L)$. In a second step, we show that as $L$ and $\lambda$ go to $+\infty$ in the regime $L\le c_0\log \lambda$, provided that $\Sigma(\lambda,L)$ does not get too close to $0$, the cumulants --- defined in the following --- of the reduced variable $\frac{\widetilde N_{\infty}(\lambda,L)}{\Sigma(\lambda,L)}$ converge to those of the normal distribution $\mathcal N(0,1)$, implying convergence in distribution.
\subsection{Reminders of probability theory} 
Let $Y$ be a random variable. Formally, the \emph{cumulants} of $Y$ are the coefficients $\kappa_m=\kappa_m(Y)$ appearing in the power series
\begin{equation}\label{eq:defcumulants}\log \mathbf E(\mathrm{e}^{tY})=\sum_{m\ge 1} \frac{\kappa_m(Y)}{m!}t^m.\end{equation}
Assuming the function $t\mapsto \log \mathbf E(\mathrm{e}^{tY})$ to be smooth in a neighborhood of $0$, they are defined by
\[\kappa_m(Y):=\left.\frac{\mathrm d^m}{\mathrm dt^m}\right|_{t=0}\log \mathbf E(\mathrm{e}^{tY}).\]
Note that the cumulants of $Y$ can be written as polynomial expressions in the moments of $Y$, and \emph{vice versa}. We chose to work with cumulants instead of moments because of the following property, that allows to compute easily the cumulants of a sum of independent random variables.   
\begin{lem}\label{rem: cumulants ind}If $X$ and $Y$ are independent variables, with well-defined cumulants, then the cumulants of $X+Y$ are given by $\kappa_m(X+Y)=\kappa_m(X)+\kappa_m(Y)$. For $a\in \mathbf R$ we have $\kappa_m(aX)= a^m\kappa_m(X)$.\end{lem}
We recall the formulas for the cumulants of Poisson and Gaussian random variables:
\begin{itemize}
\item A Poisson distribution with expected value $\lambda$ has all its cumulants equal to $\lambda$.
\item For a Gaussian distribution with mean $\mu$ and variance $\Sigma^2$, one has $\kappa_1=\mu$, $\kappa_2=\Sigma^2$ and $\kappa_m=0$ for $m\ge 3$. 
\end{itemize}

We recall the notion of convergence \emph{in distribution} of random variables:
\begin{defi}Let $(Y_n)$ be a sequence of real random variables on $\mathbf R^d$ --- not necessarily defined on the same space. We say that $(Y_n)$ converges to $Y_\infty$ \emph{in distribution}, or in law, if for any continuous bounded function $h$ on $\mathbf R^d$, we have
\[\mathbf E[h(Y_n)]\underset{n\to+\infty}\longrightarrow \mathbf E[h(Y_\infty)].\]\end{defi}
As is well known, to check convergence in distribution, it is sometimes enough to check convergence of moments, see e.g. \cite[Chapter 30]{billingsley1995probability}, but we can also formulate the result with cumulants.
\begin{lem} Let $Y_\infty$ be a real random variable such that the power series \eqref{eq:defcumulants} has positive radius of convergence. Assume that $(Y_n)$ is a sequence of random variables such that for any $m\ge 1$, the sequence $(\kappa_m(Y_n))$ converges to $\kappa_m(Y_{\infty})$. Then, $(Y_n)$ converges in distribution to $Y_\infty$. \end{lem}

\subsection{The large $n$ limit} Recall that we can expand
\[\widetilde N_n(\lambda,L)=\frac 2L\sum_{\gamma\in \mathcal P_0}\sum_{k\ge 1} \widetilde F_n(\gamma^k) A(\gamma,k),\]
where $A(\gamma,k)$ has been defined in \eqref{def A(gamma,k)}. The issue with this expression of $\widetilde N_n$ is that it is not a sum of asymptotically independent random variables, indeed if $\gamma\in \mathcal P$ and $k\neq k'$, then $\mathbf E_n\big[\widetilde F_n(\gamma^k) \widetilde F_n(\gamma^{k'})\big]$ does not converge to $0$ as $n\to +\infty$. 

The right family of random variables to consider is given by $(C_{n}(\gamma,d))_{\gamma\in \mathcal P_0, d\ge 1}$, where $C_{n}(\gamma,d)$ is the number of $d$-cycles of the permutation $\phi_n(\gamma)$. We can relate $F_n(\gamma^k)$ to the $C_{n}(\gamma,d)$'s by the formula
\[F_n(\gamma^k)=\sum_{d|k} d C_{n}(\gamma,d).\]
Now, the random variables $C_{n}(\gamma,d)$ decorrelate in the large $n$ limit. We shall use the following result of Maoz, used in \cite{Maoz2} to obtain analogues of Theorems \ref{thm:asGOEfluctuations}-\ref{thm : CLT} for surfaces of constant negative curvature.
\begin{thm}[{\cite[\S 1.2]{Maoz1}}]\label{thm:convergence en distribution} There is a family $(Z_{\gamma,d})_{\gamma\in \mathcal P_0,d\ge 1}$ of independent Poisson variables, with $\mathbf E[Z_{\gamma,d}]=\frac 1d$, such that for any finite subset $\{\gamma_i,d_i\}_{i\in I}$ of $\mathcal P_0\times \mathbf N^*$, and for any nonnegative integers $r_i$, we have
\[\underset{n\to+\infty}\lim \mathbf E\left[\prod_{i\in I}C_n(\gamma_i,d_i)^{r_i}\right]=\prod_{i\in I}\mathbf E\left[Z_{\gamma_i,d_i}^{r_i}\right],\]
As a consequence, $(C_n(\gamma_i,d_i))_{i\in I}$ converges in distribution to $(Z_{\gamma_i,d_i})_{i\in I}$. We say that the random variables $C_{n}(\gamma,d)$ are asymptotically independent.\end{thm}
\begin{rem}A weaker statement already appeared in \cite{puder2022local}, where the authors obtained asymptotic mutual independence of the variables $C_{n}(\gamma,d)$, where $\gamma$ is a fixed primitive geodesic, and $d\in \mathbf N^*$. \end{rem}
By the continuous mapping theorem we obtain
\begin{prop}\label{prop:limit distribution} The sequence $(\widetilde N_n(\lambda,L))_{n\ge 1}$ converges in distribution to
\[\widetilde N_\infty(\lambda,L)\overset{\mathrm{def}}=\frac 2L\sum_{\gamma\in \mathcal P_0}\sum_{k\ge 1} \widetilde F_{\infty}(\gamma^k) A(\gamma,k).\]
Here, we denote $\widetilde F_{\infty}(\gamma^k)=\sum_{d|k} d\widetilde Z_{\gamma,d}$, where $Z_{\gamma,d}$ are mutually independent Poisson distributions with mean $\frac 1d$.  \end{prop}
Since $\lim_{n\to +\infty}\Sigma_n(\lambda,L)=\Sigma(\lambda,L)$, provided that $\Sigma(\lambda,L)$ is nonzero, $\frac{\widetilde N_n(\lambda,L)}{\Sigma_n(\lambda,L)}$ converges in distribution to $\frac{\widetilde N_\infty(\lambda,L)}{\Sigma(\lambda,L)}$. Therefore, for any bounded continuous function $h$,
\[\underset{n\to +\infty}\lim\mathbf E_n\left[h\left(\frac{\widetilde N_n(\lambda,L)}{\Sigma_n(\lambda,L)}\right)\right]=\mathbf E\left[h\left(\frac{\widetilde N_\infty(\lambda,L)}{\Sigma(\lambda,L)}\right)\right].\]

\subsection{Cumulants of $\widetilde N_\infty(\lambda,L)$}
We now study the cumulants of $\widetilde N_\infty(\lambda,L)$, in the large $\lambda$ limit. We denote by $\kappa_m(\lambda,L)$ the cumulants of $\widetilde N_\infty(\lambda,L)$. In particular, $\kappa_1=0$ and $\kappa_2=\Sigma^2(\lambda,L)$. We show
\begin{prop}\label{prop: cumulants goes to 0} There is a constant $c_0>0$ such that the cumulants $\kappa_m(\lambda,L)$, $m\ge 3$, all converge to $0$ when $\lambda \to +\infty$. More precisely,
\[\kappa_m(\lambda,L)=\mathcal O(L^{-m}).\]\end{prop}
\begin{proof}Starting from the expression of Proposition \ref{prop:limit distribution}, we expand
\[\widetilde N_\infty(\lambda,L)=\frac 2L\sum_{\gamma\in \mathcal P_0}\sum_{d\ge 1} d\widetilde Z_{\gamma,d} \sum_{j\ge 1} A(\gamma,dj). \] 
The variables $\widetilde Z_{\gamma,d}$ are mutually independent, also, since $Z_{\gamma,d}$ is a Poisson variable with mean $1/d$, we have $\kappa_m(\widetilde Z_{\gamma,d})=\frac 1d$ for all $m\ge 2$. Therefore, by Lemma \ref{rem: cumulants ind}, the cumulants of $\widetilde N_\infty(\lambda,L)$ are given by
\[\kappa_m(\lambda,L)=2^mL^{-m}\sum_{\gamma\in \mathcal P_0, d\ge 
1} d^{m-1}\Big(\sum_{j\ge 1} A(\gamma,dj)\Big)^m, \qquad m\ge 3.\]
Fix $\gamma \in \mathcal P_0$ and recall the bound, for $\theta$ twice smaller than in Lemma \ref{Poincaré Anosov},
\[|A(\gamma,k)|\lesssim \ell(\gamma)\frac{\mathbf 1_{k\ell(\gamma)\le L}}{|I-\mathcal P_\gamma^k|^{\frac 12}}\le \ell(\gamma)\frac{\mathbf 1_{k\ell(\gamma)\le L}\mathrm{e}^{-(k-1)\theta\ell(\gamma)}}{|I-\mathcal P_\gamma|^{\frac 12}}.\] 
From there, for any $d\ge 1$,
\[\Big|\sum_{j\ge 1} A(\gamma,dj)\Big|^m\lesssim \ell(\gamma)^m\frac{\mathbf 1_{d\ell(\gamma)\le L}}{|I-\mathcal P_\gamma|^{\frac m2}}\Big(\sum_{j\ge 1} \mathrm{e}^{-(dj-1)\theta\ell(\gamma)}\Big)^m\]
which is itself bounded by
\[\ell(\gamma)^m\mathbf 1_{d\ell(\gamma)\le L}\frac{\mathrm{e}^{\theta m \ell(\gamma)}}{|I-\mathcal P_\gamma|^{\frac m2}} \Big(\sum_{j\ge 1} \mathrm{e}^{-dj\theta\ell(\gamma)}\Big)^m\lesssim \ell(\gamma)^m \mathbf 1_{d\le L/\ell(\gamma)} \frac{\mathrm{e}^{\theta m(1-d) \ell(\gamma)}}{|I-\mathcal P_\gamma|^{\frac m2}}.\]
Summing over all $d\ge 1$ leads to
\[\sum_{d\ge 1} d^{m-1}\Big(\sum_{d|k} A(\gamma,k)\Big)^m\lesssim \ell(\gamma)^m\frac{\mathbf 1_{\ell(\gamma)\le L}}{ |I-\mathcal P_\gamma|^{\frac m2}}\sum_{d\ell(\gamma)\le L} d^{m-1}\mathrm{e}^{\theta m(1-d) \ell(\gamma)}.\]
Now for fixed $\ell(\gamma)$,
\[\sum_{1\le d\le L/\ell(\gamma)} d^{m-1}\mathrm{e}^{\theta m(1-d) \ell(\gamma)}=\mathbf 1_{\ell(\gamma)\le L}\mathcal O_m(1).\]
Eventually,
\[\kappa_m(\lambda,L)\lesssim_m \frac{1}{L^m}\sum_{\ell(\gamma)\le L} \frac{\ell(\gamma)^m}{|I-\mathcal P_\gamma|^{\frac m2}}\lesssim_m \frac{1}{L^m}\sum_{\ell(\gamma)\le L} \ell(\gamma)^m \frac{\mathrm{e}^{\theta \ell(\gamma)(1-\frac m2)}}{|I-\mathcal P_\gamma|}=\mathcal O_m\left(\frac{1}{L^m}\right).\]
where we have used that the sum over $\gamma$ is bounded uniformly in $L$, because $m\ge 3$.
\end{proof}

\begin{proof}[Proof of Theorem \ref{thm : CLT}.] As $j$ goes to $+\infty$, the hypothesis $\Sigma^2(\lambda_j,L_j)\gg L_j^{-2}$ together with Proposition \ref{prop: cumulants goes to 0} ensure that the cumulants of order $\ge 3$ of the reduced variable $\frac{\widetilde N_\infty(\lambda_j,L_j)}{\Sigma(\lambda_j,L_j)}$ converge to $0$. Therefore, the sequence $\big(\frac{\widetilde N_\infty(\lambda_j,L_j)}{\Sigma(\lambda_j,L_j)}\big)$ converges in distribution to a centered Gaussian random variable with variance $1$. \end{proof}

\section{Almost sure GOE fluctuations}\label{sec:asGOE} Now we turn to the proof of Theorem \ref{thm:asGOEfluctuations}. We will take ultimately $L\le c_0\log \lambda$ with a constant $c_0$ sufficiently small, so that any error terms $\mathcal O(\lambda^{-1}\mathrm{e}^{CL})$ that may appear in the computations can be discarded. 

Fix $\varepsilon>0$. The convergence of $\widetilde N_n$ to $\widetilde N_\infty$ in distribution implies
\begin{equation}\label{eq:tobound} \begin{array}{l} \displaystyle \underset{n\to+\infty}\limsup\ \mathbf P_n\left(\frac{1}{\Lambda}\left|\int_\lambda^{\lambda+\Lambda} \big(\widetilde N_n^2(\mu,L)-\Sigma^2_{\rm GOE/GUE}\big) \mathrm d\mu\right|>\varepsilon\right)\\ \qquad \qquad \qquad \qquad \le \displaystyle \mathbf P\left(\frac{1}{\Lambda}\left|\int_\lambda^{\lambda+\Lambda} \big(\widetilde N_{\infty}^2(\mu,L)-\Sigma^2_{\rm GOE/GUE}\big) \mathrm d\mu\right|\ge \varepsilon\right).\end{array}\end{equation}
By Markov inequality combined with the triangular inequality, the right-hand side of \eqref{eq:tobound} is bounded by
\begin{equation}\label{eq:Markov}\frac{1}{\varepsilon}\mathbf E\left[\frac{1}{\Lambda}\left|\int_\lambda^{\lambda+\Lambda} \big( \widetilde N_{\infty}^2(\mu,L)-\Sigma^2(\mu,L) \big)\mathrm d\mu\right|\right]
+\frac{1}{\varepsilon\Lambda}\int_\lambda^{\lambda+\Lambda} \big| \Sigma^2(\mu,L)-\Sigma^2_{\rm GOE/GUE} \big|\mathrm d\mu.\end{equation}
We will show that \eqref{eq:Markov} goes to $0$ as $\lambda\to +\infty$. The second term has already been dealt with in the proof of Theorem \ref{mainthm}, where we showed
\[\frac{1}{\Lambda}\int_\lambda^{\lambda+\Lambda} \big| \Sigma^2(\mu,L)-\Sigma^2_{\rm GOE/GUE} \big|\mathrm d\mu= \mathcal O\Big(\sqrt{L^{-1}+\Lambda^{-1}}\Big).\]

It remains to show that the first term of \eqref{eq:Markov} goes to $0$. By Cauchy--Schwarz inequality, it will suffice to show that 
\[\mathbf E\left[\frac{1}{\Lambda^2}\left|\int_\lambda^{\lambda+\Lambda} \widetilde N_{\infty}^2(\mu,L)-\Sigma^2(\mu,L) \mathrm d\mu\right|^2\right]\longrightarrow 0.\]
Since $\mathbf E[\widetilde N_\infty^2(\mu,L)]=\Sigma^2(\mu,L)$ this is equivalent to showing
\[\frac{1}{\Lambda^2}\int_\lambda^{\lambda+\Lambda}\int_\lambda^{\lambda+\Lambda}\mathbf E\big[\widetilde N_{\infty}^2(\mu,L)\widetilde N_{\infty}^2(\mu',L)\big]\mathrm d\mu\mathrm d\mu'- \left(\frac 1\Lambda\int_\lambda^{\lambda+\Lambda}\Sigma^2(\mu,L)\mathrm d\mu\right)^2 \longrightarrow0.\]
The rest of the proof will consist in studying carefully the quantity $\mathbf E\big[\widetilde N_{\infty}^2(\mu,L)\widetilde N_{\infty}^2(\mu',L)\big]$. We can expand 
\[\mathbf E\big[\widetilde N_{\infty}^2(\mu,L)\widetilde N_{\infty}^2(\mu',L)\big]=\sum_{\gamma_i\in \mathcal P_0,k_i} \cos(\mu_i k_i \ell(\gamma_i)) a(\gamma_i,k_i) \mathbf E\Big[\prod_{i=1}^4 \widetilde F_\infty(\gamma_i^{k_i})\Big]+\mathcal O(\lambda^{-1/2}\mathrm{e}^{CL}),\]
for some constant $C$. Here the sum runs over  $(\gamma_i)_{1\le i\le 4}$ and $(k_i)_{1\le i\le 4}$ such that $k_i\ell(\gamma_i)\le L$. We have set $\mu_i=\mu$ for $i=1,2$ and $\mu_i=\mu'$ for $i=3,4$, and
\[a(\gamma,k)= \frac 2L(\chi_\rho(\gamma^k)+\overline{\chi_\rho(\gamma^k)})\frac{\ell(\gamma)\widehat \psi(k\ell(\gamma)/L)}{|I-\mathcal P_{\gamma}^{k}|^{\frac 12}}.\]
The fact that we have a remainder $\mathcal O(\lambda^{-\frac 12}\mathrm{e}^{CL})$ follows from the fact that the expected value $\mathbf E\Big[\prod_{i=1}^4 \widetilde F_\infty(\gamma_i^{k_i})\Big]$ is bounded by a polynomial function in the variables $k_i$. Indeed, since the sum runs over $(\gamma_i,k_i)$ such that $k_i\ell(\gamma_i)\le L$, in particular we have $k_i\le L/\ell_0$ thus the expected value above is bounded by a power of $L$.


When computing the expected value, the terms with nonzero contributions are those for which there is no single geodesic involved in the product (meaning each $\gamma_i$ is equal to another $\gamma_j$). Indeed, If $\gamma_1\notin\{\gamma_2,\gamma_3,\gamma_4\}$, then $\widetilde F_\infty(\gamma_1^{k_1})$ is independent of $\widetilde F_\infty(\gamma_2^{k_2})\widetilde F_\infty(\gamma_3^{k_3})\widetilde F_\infty(\gamma_4^{k_4})$. Since  $\widetilde F_\infty(\gamma_1^{k_1})$ has zero mean, the mean of the product is $0$. After removing these zero contributions to the sum, we are left with

\[\begin{array}{lll}\mathbf E[\widetilde N_{\infty}^2(\mu,L)\widetilde N_{\infty}^2(\mu',L)]& =\sum_{\gamma,k_1,k_2,k_3,k_4} \mathbf{Cov}\big(\widetilde F_\infty(\gamma^{k_1})\widetilde F_\infty(\gamma^{k_2}),\widetilde F_\infty(\gamma^{k_3})\widetilde F_\infty(\gamma^{k_4})\big)* &(1)\\
&+\sum_{\gamma_1,\gamma_3,k_1,k_2,k_3,k_4} \mathbf E[\widetilde F_\infty(\gamma_1^{k_1})\widetilde F_\infty(\gamma_1^{k_2})]\mathbf E[\widetilde F_\infty(\gamma_3^{k_3})\widetilde F_\infty(\gamma_3^{k_4})]*& (2)\\
&+2\sum_{\gamma_1,\gamma_2,k_1,k_2,k_3,k_4} \mathbf E[\widetilde F_\infty(\gamma_1^{k_1})\widetilde F_\infty(\gamma_1^{k_3})]\mathbf E[\widetilde F_\infty(\gamma_2^{k_2})\widetilde F_\infty(\gamma_2^{k_4})]*& (3)\\ 
& -2\sum_{\gamma,k_1,k_2,k_3,k_4} \mathbf E[\widetilde F_\infty(\gamma^{k_1})\widetilde F_\infty(\gamma^{k_3})]\mathbf E[\widetilde F_\infty(\gamma^{k_2})\widetilde F_\infty(\gamma^{k_4})]*& (4) \\ & + \mathcal O(\lambda^{-\frac 12}\mathrm{e}^{CL})
\end{array},\]
for some constant $C>0$. Here $\mathbf{Cov}$ denotes the covariance. Let us describe the origin of each term. 
\begin{enumerate}[---]
\item (2) comes from quadruplets $(\gamma_i)$ such that $\gamma_1=\gamma_2$, $\gamma_3=\gamma_4$. The coefficient $*$ is given by 
\[\cos(\mu k_1\ell(\gamma_1))\cos(\mu k_2\ell(\gamma_1))\cos(\mu' k_3\ell(\gamma_3))\cos(\mu' k_4\ell(\gamma_3))\]
multiplied by a dynamical coefficient (the product of $a(\gamma_i,k_i)$ for $i=1,\ldots,4$). Actually,
\[(2)=\mathbf E[\widetilde N_\infty(\mu,L)^2]\mathbf E[\widetilde N_\infty(\mu',L)^2]=\Sigma^2(\mu,L)\Sigma^2(\mu',L).\]

\item (3) comes from quadruplets with $\gamma_1=\gamma_3$, $\gamma_2=\gamma_4$ (or $\gamma_1=\gamma_4$, $\gamma_2=\gamma_3$, both giving the same value, whence the factor $2$). The coefficient $*$ is given by 
\[\cos(\mu k_1\ell(\gamma_1))\cos(\mu k_2\ell(\gamma_2))\cos(\mu' k_3\ell(\gamma_1))\cos(\mu' k_4\ell(\gamma_2))\] multiplied by dynamical coefficients. Whereas in $(2)$, all pairs $(\gamma_1,\gamma_3)$ contribute nontrivially to the sum, in this case, since we have mixed phases, only pairs for which $\frac 12\ell(\gamma_2)\le \ell_{\gamma_1}\le 2\ell(\gamma_2)$ will yield a non-negligible contribution, due to the integration in the $\mu,\mu'$ variables. This is clearer when looking at the case where all $k_i$'s equal $1$. The corresponding term in (2) is $\cos^2(\mu \ell(\gamma_1))\cos^2(\mu' \ell(\gamma_3))$ so there are no cancellations when integrating the $\mu,\mu'$ variables. Contrarily, the corresponding term in (3) is
\[\cos(\mu \ell(\gamma_1))\cos(\mu \ell(\gamma_2)) \times \cos(\mu' \ell(\gamma_1))\cos(\mu' \ell(\gamma_2)).\]
Here the integration in the $\mu,\mu'$ variable will kill the contribution of all pairs aforementioned. 

\item We have to take in account quadruplets for which all $\gamma_i$'s are equal. Also, such quadruplets may have been counted multiple times in (2) and (3) hence we have to remove some of them, resulting in (1)+(4). \end{enumerate}

We will show that each of the terms $(1)$,$(3)$,$(4)$ converges to $0$. It is easy to show that $(1)$ and $(4)$ go to $0$, while controlling $(3)$ is more technical.

\subsection{(1) and (4) go to $0$.} According to Theorem \ref{thm:convergence en distribution}, we can bound the product of expected values by a polynomial function $P$ in the variables $k_1,k_2,k_3,k_4$. Then, by using the lower bound $|I-\mathcal P_{\gamma}^k|\ge |I-\mathcal P_\gamma|\mathrm{e}^{\theta(k-1)\ell(\gamma)}$ of Lemma \ref{Poincaré Anosov}, the sums $(1)$ and $(4)$ are bounded by
\[\frac{1}{L^4}\sum_{k_i\lesssim L} P(\{k_i\})\sum_{\ell(\gamma)\le L}\frac{\ell(\gamma)^4 \mathrm{e}^{-\theta(\sum k_i-4)\ell(\gamma)} }{|I-\mathcal P_\gamma|^2}.\]
The sums  $\sum_{\ell(\gamma)\le L}\frac{\ell(\gamma)^4 }{|I-\mathcal P_\gamma|^2}$ are uniformly bounded, due to the exponent $2$ in the denominator, as an easy consequence of Lemma \ref{lem:sommes exponentielles}. Now, observe that $\mathrm{e}^{-\theta(\sum k_i-4)\ell(\gamma)}=\mathcal O(\mathrm{e}^{-\theta \ell_0/5\sum k_i})$. It follows that $(1)$ and $(4)$ are bounded by a constant times
\[\frac{1}{L^4}\sum_{k_i\lesssim L} P(\{k_i\})\mathcal O(\mathrm{e}^{-\theta \ell_0/5\sum k_i})=\mathcal O\Big(\frac{1}{L^4}\Big).\]
\subsection{(3) goes to $0$} This case requires more attention since it involves pairs of closed geodesics. 

\subsubsection{Discarding pairs of closed geodesics } By the remark above, integrating in the $\mu,\mu'$ variables over intervals of size $\Lambda$ causes the terms with $|k_1\ell(\gamma_1)-k_2\ell(\gamma_2)|\ge \frac{\ell_0}{2}$ or $|k_3\ell(\gamma_1)-k_4\ell(\gamma_2)|\ge \frac{\ell_0}{2}$ to be negligible, namely yield a contribution $\mathcal O(\Lambda^{-1})$ (here $\ell_0/2$ can be replaced by any fixed positive constant, but this particular value happens to be convenient). This essentially follows from the same argument as in \S \ref{subsec:quadratic on average} thus we shall not elaborate. Therefore, from now on we can restrict our attention to the cases $|k_1\ell(\gamma_1)-k_2\ell(\gamma_2)|< \ell_0/2$, $|k_3\ell(\gamma_1)-k_4\ell(\gamma_2)|<\ell_0/2$. 

\subsubsection{Controlling pairs of close lengths}  Denote by (3') the contribution to (3) arising from $\gamma_1,\gamma_2,(k_i)$ such that $|k_1\ell(\gamma_1)-k_2\ell(\gamma_2)|< \ell_0/2$, $|k_3\ell(\gamma_1)-k_4\ell(\gamma_2)|<\ell_0/2$. As before, there is a (maybe different) polynomial function $P$ such that $(3')$ is bounded by
\[\frac{1}{L^4}\sum_{\gamma_1,\gamma_2,k_i} P(\{k_i\}) \ell(\gamma_1)^2\ell(\gamma_2)^2\frac{\mathrm{e}^{-\theta((k_1+k_3-2)\ell(\gamma_1)+(k_2+k_4-2)\ell(\gamma_2))}}{|I-\mathcal P_{\gamma_1}||I-\mathcal P_{\gamma_2}|}.\]
where we restrict the sum to indices aforementioned. Start by assuming $k_1+k_2+k_3+k_4=4$. The condition then writes $|\ell(\gamma_1)-\ell(\gamma_2)|< \ell_0/2$. As we have already seen in \eqref{eq : convergence L2 intermède}, the corresponding sum converges to $0$ at rate $\mathcal O(L^{-1})$. We turn to the case $k_1+k_2+k_3+k_4\ge 5$. Without loss of generality, assume $k_2+k_4\ge 3$, and fix $\gamma_1$. Then, consider the sum
\begin{equation}\label{consider the sum}\sum_{\gamma_2\in \mathcal P_0}\ell(\gamma_2)^2\frac{\mathrm{e}^{-\theta(k_2+k_4-2)\ell(\gamma_2)}}{|I-\mathcal P_{\gamma_2}|},\end{equation}
where the sum runs over nonoriented primitive geodesics $\gamma_2$ such that
\begin{equation}\label{eq: distance gamma1 gamma2}\left|\ell(\gamma_2)-\frac{k_1}{k_2}\ell(\gamma_1)\right|<\frac{\ell_0}{2k_2}, \qquad \left|\ell(\gamma_2)-\frac{k_3}{k_4}\ell(\gamma_1)\right|<\frac{\ell_0}{2k_4}.\end{equation}
One the one hand, $\ell(\gamma_2)\ge \ell_0$, on the other hand, by \eqref{eq: distance gamma1 gamma2}, $(k_2+k_4)\ell(\gamma_2)\ge \left(\frac{k_1+k_3}{2}\right) \ell(\gamma_1)$. Since $k_2+k_4\ge 3$ we infer
\[(k_2+k_4-2)\ell(\gamma_2)\ge \frac 13(k_2+k_4)\ell(\gamma_2)\ge \frac{1}{6}(k_2+k_4)\ell_0+ \frac{1}{12}\left(k_1\ell(\gamma_1)+k_3\ell(\gamma_1)\right).\]
Then, using Lemma \ref{lem:sommes exponentielles} we can bound \eqref{consider the sum} by a constant times
\[\exp\left({-\frac{\theta}{24}\big((k_2+k_4)\ell_0 +(k_1+k_3)\ell(\gamma_1)\big)}\right).\]
Consequently, we have
\[\begin{array}{l} \displaystyle \frac{1}{L^4}\sum_{k_i, k_2+k_4\ge 3} P(\{k_i\}) \mathrm{e}^{-\theta(k_2+k_4)\ell_0/24} \sum_{\gamma_1\in \mathcal P_0}  \ell(\gamma_1)^2\frac{\mathrm{e}^{-\theta(k_1+k_3)\ell(\gamma_1)/24}}{|I-\mathcal P_{\gamma_1}|}\lesssim \frac{1}{L^4}\sum_{k_i} P(\{k_i\})\mathrm{e}^{-\theta \ell_0/48\sum k_i}\end{array},\]
and the right-hand side is $\mathcal O\left(L^{-4}\right)$ by Lemma \ref{lem:sommes exponentielles} again. If $k_1+k_3\ge 3$ we just consider the sum over $\gamma_1$ to obtain similar bounds. Combining all our estimates leads to 
\[\mathbf E\left[\frac{1}{\Lambda^2}\left|\int_\lambda^{\lambda +\Lambda } \big(\widetilde N_{\infty}^2(\mu,L)-\Sigma^2(\mu,L)\big) \mathrm d\mu\right|^2\right]=\mathcal O(L^{-1}+\Lambda^{-1}).\]
After applying Cauchy--Schwarz inequality, we are left with
\[\mathbf E\left[\frac{1}{\Lambda}\left|\int_\lambda^{\lambda +\Lambda} \big( \widetilde N_{\infty}^2(\mu,L)-\Sigma^2(\mu,L)\big) \mathrm d\mu\right|\right]\lesssim \sqrt{L^{-1}+\Lambda^{-1}}.\]
Hence we can let $\varepsilon$ in the statement of Theorem \ref{thm:asGOEfluctuations} depend on $\lambda$, provided that $\varepsilon^2\gg \frac 1L+ \frac{1}{\Lambda}$.

\appendix
\section{Perturbation of the Ruelle operator by a smooth potential}\label{sec:appendix} In this appendix, we study the behavior of the resonances of the Ruelle operator perturbed by a smooth potential. The setting is the following : let $(M,g)$ be a smooth Riemannian manifold of dimension $n$, and denote by $\mathrm dv_g$ the Riemannian volume form on $M$. Consider an Anosov flow $\phi^t:M\to M$, generated by a smooth vector field $V\in C^\infty(M,TM)$. In the case where $\phi^t$ does not preserve the Riemannian volume form, the Ruelle operator $\mathbf P=-\mathrm{i} V$ is not symmetric on $L^2(M,\mathrm dv_g)$, and one has
\[\mathbf P^*=\mathbf P-\mathrm{i}\, \mathrm{div}_g(V),\]
where $\mathrm{div}_g(V)$ denotes the divergence of the vector field $V$. 

Fix $a\in C^\infty(M)$, and consider the operator $\mathbf P_\alpha:=\mathbf P+\alpha a$ on $C^\infty(M)$, where $a$ acts by multiplication. First, we fix some relatively compact set $\Omega\subset \mathbf C$ and study resonances of $\mathbf P_\alpha$ in $\Omega$ by perturbative arguments. The methods are not new, and more complicated situations were already considered by Dyatlov--Zworski \cite{Dyatlov_2015}, who were considering perturbations $\mathbf P+\mathrm i \alpha \Delta_g$ for some metric $g$ on $M$, and Guedes Bonthonneau \cite{BonthonneauPerturbation}, who considered vector fields close to $ V$ in $C^1$-norm. The study of the first resonance under perturbations for symbolic dynamics can be found \emph{e.g.} in \cite[Chapter 5]{Ruelle_2004}, here we give a proof using the formalism of resolvents on anisotropic Sobolev spaces. The results of Jin--Tao \cite{jin2023number} give uniform counting bounds for resonances in horizontal strips with respect to $\alpha\in [0,\alpha_0]$. These estimates can be obtained by studying the operator $h\mathbf P_\alpha$ (here $h$ is a semiclassical parameter) near the energy shell $\{|\xi|_g=1\}\subset T^*M$. 

Since the resolvent estimates of \cite{Jin2016,jin2023number} are uniform with respect to the lower order perturbation, the estimates on the error $F_A$ in the local trace formula are uniform with respect to $\alpha$, as long as $\alpha$ belongs to a fixed compact set. 

\subsection{Perturbation of resonances in a fixed compact set} 

Dyatlov--Zworski introduce a semiclassical adaptation\footnote{Here $\Psi_h^{0+}$ denotes the intersection $\bigcap_{m>0}\Psi_h^m$.} $G(h)\in \Psi_h^{0+}(M)$ of the operator $G$ of \S \ref{subsec:anisotropic}. The space $H_{sG(h)}$ does not depend on $h$, and the norms $\|\bullet\|_{H_{sG(h)}}$ and $\|\bullet\|_{H_{sG}}$ are equivalent, with constants depending on $h$ \cite[\S 3.3]{dyatlov2016dynamical}.

We perturb the semiclassical operator $h\mathbf P\in \Psi_h^1(M)$ by adding a lower order term $\alpha h a\in h\Psi_h^0(M)$, that is uniformly bounded in $h \Psi_h^0(M)$, for $\alpha$ ranging in a compact set. Thus, the principal symbol is left unchanged, and we have uniform estimates on the lower order terms. Therefore, elliptic estimates and propagation estimates are uniform with respect to the parameter $\alpha$. Notably, we have

\begin{prop}[{\cite[Proposition 3.4]{dyatlov2016dynamical}}] Take an absorbing operator $Q$ that is elliptic near the $0$ section, defined by $Q =\chi(h^2\Delta_g)$ for some smooth $\chi$ compactly supported in $[-1,1]$ and such that $\chi(0)=1$. Fix $R>0$ and $\alpha_0>0$. Then for $s>s_0(R,\alpha_0)$ large enough and $0<h<h_0(R,\alpha_0)$ small enough, the operator
\[h\mathbf P_\alpha-\mathrm iQ-z:H_{sG(h)}\to H_{sG(h)}\]
is invertible for $-Rh\le \operatorname{Im} z\le 1$, $|\operatorname{Re} z|\le h^{\frac 12}$, with resolvent estimates
\[\|\big(h\mathbf P_\alpha-\mathrm iQ-z\big)^{-1}\|_{H_{sG(h)}\to H_{sG(h)}}\le C(R,\alpha_0,s_0) h^{-1}.\]\end{prop}
The addition of the operator $Q$ ensures invertibility. Then, introduce the operator
\[K(\alpha,z):=\mathrm i Q (h\mathbf P_\alpha-\mathrm iQ-hz)^{-1}.\] 
Since $Q$ has finite rank (although depending on $h$), $K(\alpha,z)$ also has finite rank. We may then write
\[\mathbf P_\alpha-z=h^{-1}\big(I+K(\alpha,z)\big)(h\mathbf P_\alpha -\mathrm{i} Q -hz).\]
Analytic Fredholm theory allows showing that $(\mathbf P_\alpha-z)^{-1}$ is a meromorphic (with respect to the variable $z$) family of operators $H_{sG(h)}\to H_{sG(h)}$ for $z\in D(0,R)$, provided that $s>s_0(R,\alpha_0)$ and $h<h_0(R,\alpha_0)$. Now, for $z\in D(0,R)$, we have
\[(\mathbf P_\alpha-z)^{-1}=h(h\mathbf P_\alpha-\mathrm iQ-hz)^{-1}(I+K(\alpha,z))^{-1},\]
and, introducing $F(\alpha,z)=\det_{H_{sG(h)}}(I+K(\alpha,z))$ --- well-defined because $K$ is finite rank ---,  the resonances of $\mathbf P_\alpha $ in the disk $D(0,R)$ are exactly given by the zeroes of $F_\alpha$ in the same disk. Note that since $Q=\chi(h^2\Delta_g)$, the operator $K(\alpha,z)$ maps $H_{sG(h)}$ onto $H^N$ for any $N$; since $H^N$ embeds in $H_{sG(h)}$ for any $N\ge s$, it follows that for any $N\ge s$, $F(\alpha,z)$ is equal to the $H^N$ determinant of $I+K(\alpha,z)$. 

\subsubsection{Dependence on $\alpha$ of the resonances} We now show that $F(\alpha,z)$ depends smoothly on $\alpha$. First, we show
\begin{lem}\label{lemapp: K est lisse} Fix $R>0$ and $\alpha_0>0$. For $s>s_0(R,\alpha_0)$ and $0<h<h_0(R,\alpha_0)$ we have 
\[K(\alpha,z)\in C^{\infty}\Big([0,\alpha_0];\mathrm{Hol}\big(D(0,R), \mathcal L^1(H^s)\big)\Big),\]
where $H^s=H^s(M)$ is the usual Sobolev space, $\mathcal L^1(H^s)$ is the space of trace-class operators $H^s\to H^s$, and $\mathrm{Hol}\big(D(0,R), \mathcal L^1(H^s)\big)$ denotes the space of holomorphic maps $D(0,R)\to \mathcal L^1(H^s)$.\end{lem}

\begin{proof}Write
\[\frac{(h\mathbf P_\alpha-\mathrm iQ-hz)^{-1}-(h\mathbf P_{\alpha'}-\mathrm iQ-hz)^{-1}}{\alpha-\alpha'}= (h\mathbf P_\alpha-\mathrm iQ-hz)^{-1} h a(h\mathbf P_{\alpha'}-\mathrm iQ-hz)^{-1}.\]
Letting $\alpha'\to \alpha$ we deduce
\[\partial_\alpha (h\mathbf P_\alpha-\mathrm iQ-hz)^{-1}= (h\mathbf P_\alpha-\mathrm iQ-hz)^{-1} h a(h\mathbf P_{\alpha}-\mathrm iQ-hz)^{-1}\]
where the derivative belongs to the space of bounded operators $\mathcal B(H_{sG})$ because it is a product of continuous operators $H_{sG}\to H_{sG}$. We can iterate the procedure to show 
\[(h\mathbf P_\alpha-\mathrm iQ-hz)^{-1}\in C^\infty\Big([0,\alpha_0],\mathrm{Hol}\big(D(0,R), \mathcal B(H_{sG})\big)\Big).\]
Now, $\partial_\alpha^j K(\alpha,z)=\mathrm{i}Q \partial_\alpha^j(h\mathbf P_\alpha-\mathrm iQ-hz)^{-1}$. Since $H^s$ embeds continuously in $H_{sG}$, and $Q=\chi(h^2\Delta_g)$ maps $H_{sG}$ onto $H^s$, we have $\partial_\alpha^j K(\alpha,z)\in \mathcal B(H^s)$. Moreover, since $\chi$ is compactly supported in $[-1,1]$, we have $K(\alpha,z)=\mathbf 1_{[0,1]}(h^2\Delta_g) K(\alpha,z)$ thus
\[\big \|\partial_\alpha^j K(\alpha,z)\big\|_{\mathcal L^{1}(H^s)}\le \big\|\partial_\alpha^j K(\alpha,z)\big\|_{H^s\to H^s}\cdot  \|\mathbf 1_{[0,1]}(h^2\Delta_g)\|_{\mathcal L^1(H^s)}.\]
This concludes the proof. \end{proof}
\begin{prop}We have $\det_{H^s}(I+K(\alpha,z))\in C^{\infty}\Big([0,\alpha_0]; \mathrm{Hol}\big(D(0,R),\mathbf C\big)\Big)$.\end{prop}
\begin{proof}Start from the identity
\[\partial_\alpha F(\alpha,z)=F(\alpha,z)\operatorname{Tr} \big((I+K(\alpha,z))^{-1}\partial_\alpha K(\alpha,z)\big).\]
We recall, for $A\in \mathcal L^1$, the inequalities (see \cite[Appendix B]{DyatlovResonances})
\[\left\|\big(\det(I-A)\big)(I-A)^{-1}\right\|\le \mathrm{e}^{\|A\|_{\mathcal L^1}},\]
It then follows by Lemma \ref{lemapp: K est lisse} that the map $D(0,R)\ni z\mapsto  F(\alpha,z)\|(I+K(\alpha,z))^{-1}\|$ is bounded, uniformly with respect to $\alpha$, from which we get
\[|\partial_\alpha F(\alpha,z)|\le \mathrm{e}^{\|K(\alpha,z)\|_{\mathcal L^1}} \|\partial_\alpha K(\alpha,z)\|_{\mathcal L^1}.\]
Holomorphy in $z$ of $\partial_\alpha F$ follows from the uniform estimates in $z$. The case of higher derivatives is treated by using the same methods.\end{proof}

\subsubsection{Smoothness of the first resonance}
Since $0$ is a simple zero of $F(0,\cdot)$ --- indeed, the space of resonant states at $0$ just consists of constant functions ---, the derivative $\partial_z F(\alpha,z)_{|(0,0)}$ is nonzero, thus the implicit function theorem implies that $F(\alpha,\cdot)$ has a unique zero in a small neighborhood of $0$, denoted $\mu_\alpha$, that depends smoothly on $\alpha$. Also, by Rouch{\'e} theorem, we have a uniform upper bound on the number of resonances in the disk $D(0,R)$, independent of $\alpha$. For $R$ large enough, combined with Proposition \ref{prop: comptage résonances HE} this gives an upper bound independent of $\alpha$ on the counting functions of resonances in horizontal strips similar to that of Proposition \ref{comptage JinTao}.

By adapting the proof of \cite[Proposition 5.3]{Dyatlov2017}, we get smoothness of the spectral projector $\Pi_\alpha$ associated to the resonance $\mu_\alpha$, defined by integrating the resolvent over a small circle centered at $0$. More precisely, we have
\[\Pi_\alpha\in C^\infty\big([0,\alpha_0], \mathcal B(H_{sG(1)},H_{sG(1)})\big).\] 
Since the range of $\Pi_\alpha$ is one-dimensional for small $\alpha$, we can find a smooth family of distributions $\eta_\alpha\in \mathcal D'_{E_u^*}(M)$ such that $\eta_0\equiv 1$, and $\eta_\alpha$ is a resonant state associated to the resonance $\mu_\alpha$ --- for example, take $\eta_\alpha=\Pi_\alpha \mathbf 1$ where $\mathbf 1$ is the constant function equal to $1$ ---, meaning that in the sense of distributions,
\[\mathbf P \eta_\alpha+ \alpha a \eta_\alpha= \mu_\alpha\eta_\alpha. \]

The space of co-resonant states at $0$ is by definition $\big\{u\in \mathcal D'(M), \ \mathbf P^* u=0, \ \mathrm{WF}(u)\subset E_s^*\}$. By a result of Humbert \cite[Theorem 1]{humbert2024criticalaxisruelleresonances}, this set equals $\mathrm{span}(\nu_0)$, where $\nu_0$ is the SRB measure. 

Working with $\mathbf P^*$ instead of $\mathbf P$, we may find a smooth family of distributions $\nu_\alpha\in \mathcal D'_{E_s^*}$ such that $\nu_0$ is the SRB measure, and
\[\mathbf P^* \nu_\alpha+ \alpha a \eta_\alpha= \mu_\alpha\eta_\alpha \]
in the sense of distributions.
\subsection{Counting resonances of $\mathbf P_\alpha$ in horizontal strips}
Previously we were considering a complex absorbing operator $Q$ localized near the $0$ section, $Q=\chi(h^2\Delta_g)$, with $\chi$ compactly supported in a neighborhood of $0$. To get counting upper bounds for large energies, we rather need to localize near the energy level $\{|\xi|=1\}$. As in \cite{jin2023number}, introduce a complex absorbing operator $W$ (one may take $W=\chi(h^2\Delta_g-1)$) localized near the energy level $\{|\xi|=1\}$, with $\operatorname{rank} W=\mathcal O(h^{-n})$. Now, by studying the modified operator $\widetilde {\mathbf P}_{\alpha,h}=h\mathbf P_\alpha -MhW-z$ (with $M$ taken large enough) one can show (see \cite[\S 3]{jin2023number})
\begin{prop}\label{prop: comptage résonances HE}With notations as above, for $h<h_0$,
\[\# \mathrm{Res}(h\mathbf P_\alpha) \cap [1-h,1+h]+\mathrm i[-\beta h,1] =\mathcal O(h^{-n}).\]\end{prop}

\subsection{A local trace formula with uniform error terms} 

We now explain briefly the idea of the proof of the local trace formula, and refer to \cite[\S 4]{Jin2016} for details --- note that one has to correct the flat trace estimates, see \cite{jin2022flat}.

Define, for $\mathrm{Im}(z)\gg 1$,
\[\zeta_\alpha(z)=\exp\Big(-\sum_{\gamma\in \mathcal G} \frac{\ell(\gamma)^\sharp \mathrm{e}^{-\mathrm{i}\alpha\int_\gamma a}\mathrm{e}^{\mathrm{i} z\ell(\gamma)}}{\ell(\gamma)|I-\mathcal P_\gamma|}\Big).\]
The sum converges for $\mathrm{Im}(z)\gg 1$. By Guillemin's trace formula \eqref{eq:guillemin}, the function $\zeta_\alpha$ is related to the flat trace of the propagator by
\[\frac{\mathrm d}{\mathrm dz}\log \zeta_\alpha(z)=-\mathrm{i}\operatorname{tr^{\flat}}\int_0^{+\infty} \mathrm{e}^{-\mathrm{i} t\mathbf P_\alpha} \mathrm{e}^{\mathrm{i} tz}\mathrm dt.\]
Take a smooth real function $\varphi$, compactly supported on $(0,+\infty)$. On the one hand, taking $B>0$ large enough, we have by Fourier inversion formula,
\[\int_{\mathbf R+\mathrm{i}B} \widehat \varphi(z)\frac{\mathrm d}{\mathrm dz}\log \zeta_\alpha(z) \mathrm dz=-\mathrm{i}\int_{\mathbf R+\mathrm{i}B} \widehat \varphi(z)\operatorname{tr^{\flat}}\int_0^{+\infty} \mathrm{e}^{-\mathrm{i} t\mathbf P} \mathrm{e}^{\mathrm{i} tz}\mathrm dt\mathrm dz=-\mathrm{i} \sum_{\gamma\in \mathcal G} \frac{\ell(\gamma)^\sharp \varphi(\ell(\gamma))}{|I-\mathcal P_\gamma|}.\]
On the other hand, passing the contour to $\mathbf R-\mathrm{i} A$ gives, by the residue theorem
\begin{equation}\label{naiveprooflct}\int_{\mathbf R+\mathrm{i}B} \widehat \varphi(z)\frac{\mathrm d}{\mathrm dz}\log \zeta_\alpha(z) \mathrm dz=-2\pi \mathrm{i}\sum_{\substack{\mu\in \mathrm{Res}(\mathbf P),\\ \mathrm{Im}(\mu)>-A} } \widehat \varphi(\mu)+\int_{\mathbf R-\mathrm{i}A} \widehat \varphi(z)\frac{\mathrm d}{\mathrm dz}\log \zeta_\alpha(z) \mathrm dz.\end{equation}
Here we have used the fact that the poles of $\frac{\mathrm d}{\mathrm dz}\zeta_\alpha$ are exactly given, with multiplicities, by Pollicott--Ruelle resonances --- see \cite{dyatlov2016dynamical} for a proof of this fact. However, one has to be careful because some resonances may happen to lie on (or close to) the line $\mathbf R-\mathrm{i}A$, making it impossible to control naively the integral in the right-hand side of \eqref{naiveprooflct}. As explained in \cite{Jin2016}, a way to solve this issue is to perform a contour deformation. First, decompose the domain $\Omega=\{-A\le \mathrm{Im}(z)\le B\}$ in strips
\[\Omega_0=\Omega\cap \{|\mathrm{Re}(z)|\le 1\}, \qquad \Omega_{\pm k}=\Omega\cap \{k\le \pm \mathrm{Re}(z)\le k+1\} \ \text{for $k\ge 1$}. \]

The crucial point is now that for any $\varepsilon>0$ one can find new piecewise smooth contours $\tilde \gamma_k$ (that depend on $\alpha$), that lie in a small neighborhood of the initial contours $\gamma_k$, such that $\tilde \gamma_k$ avoid resonances quantitatively, notably one can ensure that
\[\Big|\frac{\mathrm d}{\mathrm dz}\log \zeta_\alpha(z)\Big|=\mathcal O(\langle z\rangle^{2n+1+\varepsilon}), \qquad \text{for any $z\in \tilde \gamma_k$ and $k\in \mathbf Z$}.\]
We denote by $\gamma_k^i$, $i=1,\ldots,4$ the different smooth parts of the contour $\gamma_k$ --- in particular $\gamma_k^3$ is the lower side of $\Omega_k$, and denote accordingly $\tilde \gamma_k^i$ the corresponding deformations whose concatenations form the contour $\tilde \gamma_k$. Let $\Gamma:=\bigcup_{k\in \mathbf Z} \tilde \gamma_k^3$. Then by the residue theorem,
\[\int_{\mathbf R-\mathrm{i} A} \widehat \varphi(z)\frac{\mathrm d}{\mathrm dz}\log \zeta_\alpha(z) \mathrm dz=\int_{\Gamma} \widehat \varphi(z)\frac{\mathrm d}{\mathrm dz}\log \zeta_\alpha(z) \mathrm dz-2\pi \mathrm{i} \sum_{\substack{\mu\in \mathrm{Res}(\mathbf P_\alpha)\cap \tilde \Omega,\\  \mathrm{Im}(\mu)\le -A}} \widehat \varphi(\mu)\]
The error term in the local trace formula then writes
\[\langle F_{A,\alpha},\varphi\rangle =\mathrm{i}\int_{\Gamma} \widehat \varphi(z)\frac{\mathrm d}{\mathrm dz}\log \zeta_\alpha(z) \mathrm dz+2\pi\sum_{\substack{\mu\in \mathrm{Res}(\mathbf P_\alpha)\cap \tilde \Omega,\\  \mathrm{Im}(\mu)\le -A}} \widehat \varphi(\mu).\]
Estimates on $\widehat F_{A,\alpha}$ are obtained following \cite[\S 4.3]{Jin2016} or \cite[\S 4.4]{jin2023number}, and one eventually gets
\begin{prop}[Local trace formula for the perturbed Ruelle operator] There is $\alpha_0>0$ such that for any $A>0$, there is a family of distributions $( F_{A,\alpha})_{0\le \alpha\le \alpha_0}$ on $\mathbf R$, supported on $\mathbf R_+$, such that the following equality of distributions on $(0,+\infty)$ holds:
\[\sum_{\substack{\mu \in \mathrm{Res}(\mathbf P_\alpha), \\ \mathrm{Im}(\mu)>-A}} \mathrm{e}^{-\mathrm{i} \mu t}+ F_{A,\alpha}=\sum_{\gamma \in \mathcal G} \frac{\mathrm{e}^{-\mathrm{i} \alpha\int_\gamma a}\ell(\gamma)^\sharp}{|I-\mathcal P_\gamma|}\delta(t-\ell(\gamma)).\]
Moreover, independently of $\alpha\le \alpha_0$, we have the bound on the Fourier Transform of $F_{A,\alpha}$, for any $\varepsilon>0$,
\[|\widehat F_{A,\alpha}(\xi)|=\mathcal O_{A,\varepsilon}(\langle \xi\rangle^{2n+1+\varepsilon}), \qquad \mathrm{Im}(\xi)\le A-\varepsilon.\] \end{prop}

\subsection{Taylor expansion of the first resonance of $\mathbf P+\alpha a$} \label{sec:Appendix} Fix a real valued smooth potential $a$ on $M$. We now look for an explicit approximation of the first resonance $\mu_\alpha$ of $\mathbf P+\alpha a$. Recall that we can find a smooth family of distributions $\eta_\alpha\in C^{\infty}([0,\alpha_0], \mathcal D'_{E_u^*}(M))$, with $\eta_0=1$, such that $\mathbf P\eta_\alpha+ \alpha a \eta_\alpha =\mu_\alpha\eta_\alpha$ in the sense of distributions, meaning that for all $f\in C^\infty(M)$,
\begin{equation}\label{eq:coresonant states}\langle \eta_\alpha,\mathbf P^* f\rangle +\alpha \langle \eta_\alpha, fa\rangle =\mu_\alpha\langle \eta_\alpha,f\rangle.\end{equation}
Moreover these distributions are uniquely characterized by these identities, up to a multiplicative factor. Before turning to the proof of Proposition \ref{prop: res perturbation 0}, we give an expression of $\dot \eta_0$.

\begin{lem} \label{lem:deriveeresonant} We have $\dot \eta_0=-\mathbf R_{\rm hol}(0)a+ C$, for some constant $C$, where $\mathbf R_{\rm hol}(z)$ is the holomorphic part of the resolvent near $0$. \end{lem}

\begin{proof} By differentiating the identity \eqref{eq:coresonant states} at $\alpha=0$ we get
\[\langle \dot \eta_0,\mathbf P^*f\rangle +\langle \eta_0,fa\rangle =\dot \mu_0\langle \eta_0,f\rangle.\]
Taking $f=\nu_0$ yields $\dot \mu_0=\langle \eta_0,a \nu_0\rangle=\int a \mathrm d\nu_0=0$ --- recall that $\mathbf P^*\nu_0=0$ and $\eta_0= 1$). Therefore, for any smooth function $f$,
\begin{equation}\label{eq:identity derive eta} \langle \dot \eta_0,\mathbf P^*f\rangle +\langle \eta_0,fa\rangle =0.\end{equation}

Let us now identify $\dot \eta_0$. The resolvent $\mathbf R(z)=(\mathbf P-z)^{-1}$ writes, for $z$ in a neighborhood of $0$:
\[\mathbf R(z)=\mathbf R_{\rm hol}(z) +\frac{\Pi_0}{z},\]
where $\mathbf R_{\rm hol}(z)$ is holomorphic in a neighborhood of $0$, and $\Pi_0$ is the projector onto the space of resonant states --- here it is simply the space of constant functions --- given by $\Pi_0=\eta_0\otimes \nu_0$ \cite[Lemma 3.1]{humbert2024criticalaxisruelleresonances}.
Now, since $\int a \mathrm d\nu_0=0$, we have $\Pi_0 a=0$, thus $\mathbf R(z) a =\mathbf R_{\rm hol}(z)a$ for $z$ in a neighborhood of $0$, which in turns implies $\mathbf P \mathbf R_{\rm hol}(0) a =a$. Therefore, $-\mathbf R_{\rm hol}(0) a$ satisfies the same identity \eqref{eq:identity derive eta} as $\dot \eta_0$. Moreover, the wavefront set properties of the resolvent \cite[Proposition 3.3]{dyatlov2016dynamical} (see also \cite[Theorem 9.2.4]{Lefeuvre}) ensure that $\mathbf R_{\rm hol}(0) a\in \mathcal D'_{E_u^*}$, thus the distribution $T:= \dot \eta_0+\mathbf R_{\rm hol}(0) a$ satisfies $\mathbf PT=0$ while $T\in \mathcal D'_{E_u^*}$. Therefore, $T$ is a constant function, which proves the lemma.\end{proof}

\begin{proof}[Proof of Proposition \ref{prop: res perturbation 0}]
We will just compute the first derivatives of $\alpha\mapsto \mu_\alpha$ at $\alpha=0$. Differentiating \eqref{eq:coresonant states} twice with respect to $\alpha$ and evaluating at $\alpha=0$ leads to
\[\langle \ddot \eta_0,\mathbf P^*f\rangle+2\langle \dot \eta_0,fa\rangle=\ddot \mu_0\langle \eta_0,f\rangle.\]
Setting $f=\nu_0$ yields $\ddot \mu_0=2\langle \dot \eta_0,a \nu_0\rangle$. Note that it is licit to take the pairing because $\dot \eta_0\in \mathcal D_{E_u^*}'$ and $\nu_0\in \mathcal D_{E_s^*}'$. By Lemma \ref{lem:deriveeresonant}, we get $\ddot \mu_0=-2\langle \mathbf R_{\rm hol}(0) a, a\nu_0\rangle$. Now, for $\mathrm{Im} (z) \gg 0$, we can write
\[-\mathbf R(z) =-\mathrm{i} \int_0^{+\infty} \mathrm{e}^{-\mathrm{i} t(\mathbf P-z)}\mathrm dt,\]
then, for $\mathrm{Im} (z) \gg 0$, 
\[\langle -\mathbf R(z) a, a\nu_0\rangle=-\mathrm{i}\int_0^{+\infty} \mathrm{e}^{\mathrm{i} z t}\langle a\circ \phi^t,  a \rangle_{L^2(M,\nu_0)} \mathrm dt.\]
The left-hand side is well-defined on the half-plane $\{\mathrm{Im}(z)>0\}$, since $\mathbf R(z)$ is holomorphic there (due to the absence of resonances in the upper-half plane \cite[Theorem 1]{humbert2024criticalaxisruelleresonances}). Since the flow preserves the measure $\nu_0$, the right-hand side is also holomorphic on the half-plane $\{\mathrm{Im}(z)>0\}$, thus we may let $z=\mathrm{i}\lambda$ go to $0$ to obtain
\[\langle -\mathbf R_{\rm hol}(0) a,a\nu_0\rangle =-\mathrm{i}\, \lim_{\lambda\to 0} \int_{0}^{+\infty} \mathrm{e}^{-\lambda t}\langle a\circ \phi^t,a\rangle_{L^2(M,\nu_0)} \mathrm dt=-\mathrm{i} \mathrm{Var}_{\nu_0}(a),\]
which in the same time proves the existence of the limit as $\lambda\to 0$. Note that the nonnegativity of the variance directly follows from the identity
\[\mathrm{Var}_{\nu_0}(a)= \lim_{\lambda\to 0} \lambda \int_M \left(\int_{0}^{+\infty} \mathrm{e}^{-\lambda t} a\circ \phi^t \mathrm dt\right)^2 \mathrm d\nu_0. \]

\end{proof}

\subsection{Variance of $a$ and magnetic fluxes through geodesics} \label{subsec:variance of a closed orbits} Let us go back to the setting where $M=S^*X$ and $\phi^t$ is the geodesic flow. By the exponential mixing property of the geodesic flow, the variance can be expressed as
\[\mathrm{Var}_{\nu_0}(a){=}\lim_{T\to +\infty} \frac 1T\int_{S^*X}\bigg(\int_0^T a(\phi^t(z))\mathrm dt\bigg)^2\mathrm d\nu_{0}(z),\]
see for example \cite[(2.5)]{guillarmou_knieper_lefeuvre_2022} for a short proof. Recall from \cite{10.1007/BFb0082850} that $\nu_0$ may be obtained as the weak limit
\[\nu_0=\underset{T\longrightarrow +\infty}\lim \sum_{\gamma \in\mathcal G, T\le \ell(\gamma)\le T+1} \frac{\delta_\gamma}{|I-\mathcal P_\gamma|}.\]
Heuristically, this suggests picking random closed geodesics with weight proportional to $\frac{1}{|I-\mathcal P_\gamma|}$. The following proposition gives an expression of the variance involving closed geodesics.

\begin{prop}\label{prop : appendice variance et flux} Let $a\in C^\infty(M,\mathbf R)$, and assume that $\int_M a\mathrm d\nu_0=0$. Then, for any $\varepsilon>0$, 
\[\varepsilon^{-1}\sum_{ \ell(\gamma)\in [T,T+\varepsilon]} \frac{\ell(\gamma)^\sharp }{|I-\mathcal P_\gamma|}\Big(\frac{1}{\sqrt T}\int_\gamma a\Big)^2\underset{T\to +\infty}\longrightarrow \mathrm{Var}_{\nu_0}(a).\]\end{prop}

Let us specify this result in the case $a(x,\xi)=\langle \mathbf A(x),\xi\rangle$. If $\gamma$ is a closed geodesic with period $\ell(\gamma)$ and $z$ is any point of the orbit then
\[\int_0^{\ell(\gamma)} a(\phi^t(z))\mathrm dt=\int_\gamma \mathbf A,\]
where the right-hand side is understood as the integral of the $1$-form $\mathbf A$ along the closed path $\gamma\subset X$. Thus, the variance of $a$ measures the average amplitude of magnetic fluxes through surfaces enclosed by closed geodesics.

\begin{rem}  

By \cite[Lemma 9.3.8]{Lefeuvre}, the variance of $a(x,\xi)=\langle A(x),\xi\rangle_{T^*X}$ vanishes if and only if $a=\mathbf Xu$ for some smooth function $u$, which is also equivalent to $\int_\gamma \mathbf A=0$ for any closed path $\gamma$. Therefore, $\mathrm{Var}_{\nu_0}(a)>0$ if and only if $\int_\gamma \mathbf A\neq 0$ for some closed path $\gamma$.  \end{rem}

Proposition \ref{prop : appendice variance et flux} is a direct consequence of the next lemma.

\begin{lem} \label{lem : appendice variance et flux} Under the same assumptions on $a$, for any compactly supported function $\omega$,
\[\sum_{\gamma \in \mathcal G} \frac{\ell(\gamma)^\sharp \omega(\ell(\gamma)-T)}{|I-\mathcal P_\gamma|}\Big(\frac{1}{\sqrt T}\int_\gamma a\Big)^2= \mathrm{Var}_{\nu_0}(a)\int_{\mathbf R} \omega(t)\mathrm dt +\mathcal O\left(\frac 1T\right).\]\end{lem}

\begin{proof}[Lemma \ref{lem : appendice variance et flux} implies Proposition \ref{prop : appendice variance et flux}]
For any $\delta>0$, we can find two smooth functions $\omega_{\delta}^{\pm}$ with compact supports, satisfying  $\omega^-_\delta\le \mathbf 1_{[0,\varepsilon]}\le \omega^+_\delta$ and $\|\omega^+_\delta-\omega^-_\delta\|_{L^1(\mathbf R)}\le \delta$. A squeezing argument combined with Lemma \ref{lem : appendice variance et flux} gives the sought result.\end{proof}

For geodesic flows on closed negatively curved surfaces, or more generally contact Anosov flows on closed manifolds, one can show the existence of a spectral gap \cite{TsujiGap}, meaning that $0$ is the only resonance in the half-plane $\mathrm{Im}(z)>-\kappa$ for some $\kappa>0$. This result is stable under small perturbations (\emph{e.g.} by the uniform resolvent estimates of \cite[Theorem 1.6]{faure2023microlocal}) and one can show that the operators $\mathbf P+\alpha a$ enjoy a uniform spectral gap, still denoted $-\kappa$ (for $\alpha$ small enough). 

\begin{proof}[Proof of Lemma \ref{lem : appendice variance et flux}] The idea is to consider the sum
\[\sum_{\gamma \in \mathcal G} \frac{\ell(\gamma)^\sharp \omega(\ell(\gamma)-T)}{|I-\mathcal P_\gamma|}\cos\Big( \alpha\int_\gamma a\Big)\]
in the regime $\alpha \ll T^{-1}\ll 1$. We can write this sum in two different ways. On the one hand, since $\omega(\cdot -T)$ is compactly supported in a neighborhood of $T$, and $a$ is bounded, the arguments of the cosines are $\mathcal O(\alpha T)$, and we can use Taylor formula to let appear the quantities $(\int_\gamma a)^2$. On the other hand, we can directly apply the local trace formula for the Ruelle operator $\mathbf P +\alpha a$, which lets appear the variance of $a$. By equating both terms and controlling errors, we will obtain the claimed result. 

Let us put the previous ideas in practice. First, since there are no resonances different from $\mu_\alpha$ in some horizontal strip $\{-\kappa \le \operatorname{Im} z\le 0\}$, the local trace formula for $\mathbf P+\alpha a$ gives
\begin{equation}\label{eqapp0}\sum_{\gamma \in \mathcal G} \frac{\ell(\gamma)^\sharp \omega(\ell(\gamma)-T)}{|I-\mathcal P_\gamma|}\cos\Big( \alpha\int_\gamma a\Big)=\int_{\mathbf R} e^{-\mathrm{i}\mu_\alpha t}\omega(t-T)\mathrm dt +\mathcal O(\mathrm{e}^{-\kappa T}),\end{equation}
Recall that $\mu_\alpha=-\frac 12 \mathrm{i} \alpha^2 \mathrm{Var}_{\nu_{\scalebox{.7}{$\scriptscriptstyle 0$}}}(a)+\mathcal O(\alpha^3)$. Since $\alpha T\ll 1$, applying Taylor--Young formula to the exponentials in the right-hand side, we get that \eqref{eqapp0} equals
\begin{equation}\label{eqapp1} (1-\frac 12 \alpha^2 \mathrm{Var}_{\nu_0}(a)T)\int_{\mathbf R} \omega(t)\mathrm dt +\mathcal O(\alpha^2).\end{equation}
On the other hand, for any $\gamma$ of length of order $T$, we have $\alpha\int_\gamma a=\mathcal O(\alpha T)$; since $\alpha T\ll 1$, applying Taylor--Young formula again but this time to the cosines leads to
\[\sum_{\gamma \in \mathcal G} \frac{\ell(\gamma)^\sharp \omega(\ell(\gamma)-T)}{|I-\mathcal P_\gamma|}\cos\Big( \alpha\int_\gamma a\Big)=\sum_{\gamma \in \mathcal G} \frac{\ell(\gamma)^\sharp \omega(\ell(\gamma)-T)}{|I-\mathcal P_\gamma|}\Big[1-\frac{\alpha^2}2 \Big(\int_\gamma a\Big)^2\Big] +\mathcal O((\alpha T)^4).\]
By Lemma \ref{lem: size clusters}, this is equal to
\begin{equation}\label{eqapp2}\int_{\mathbf R_+} \omega(t)\mathrm dt-\frac{\alpha^2}{2}\sum_\gamma \frac{\ell(\gamma)^\sharp \omega(\ell(\gamma)-T)}{|I-\mathcal P_\gamma|}\Big(\int_\gamma a\Big)^2+\mathcal O((\alpha T)^4 +\mathrm{e}^{-\kappa T}).\end{equation}
Equating \eqref{eqapp1} and \eqref{eqapp2} and diving both sides by $\alpha^2T$ yields
\[\sum_{\gamma \in \mathcal G} \frac{\ell(\gamma)^\sharp \omega(\ell(\gamma)-T)}{|I-\mathcal P_\gamma|}\Big(\frac{1}{\sqrt T}\int_\gamma a\Big)^2=\mathrm{Var}_{\nu_0}(a)\int_{\mathbf R}\omega(t)\mathrm dt+\mathcal O(T^{-1}+\alpha^2T^3+\alpha^{-2}T^{-1}\mathrm{e}^{-\kappa T}).\]
One can just take $\alpha=\mathrm{e}^{-\frac \kappa 2 T}$ to obtain the result.\end{proof}

\printbibliography 

\end{document}